\documentclass[reqno]{amsart}
\usepackage{amssymb}
\usepackage{graphicx}

\usepackage[usenames, dvipsnames]{color}
\usepackage{verbatim}
\usepackage{mathrsfs}
\usepackage{bm}
\usepackage{cite}

\numberwithin{equation}{section}

\newtheorem{theorem}{Theorem}[section]
\newtheorem{corollary}[theorem]{Corollary}
\newtheorem{lemma}[theorem]{Lemma}

\newtheorem{example}[theorem]{Example}

\theoremstyle{definition}
\newtheorem{remark}[theorem]{Remark}

\theoremstyle{definition}

\theoremstyle{definition}

\makeatletter
\def\dashint{\operatorname%
{\,\,\text{\bf-}\kern-.98em\DOTSI\intop\ilimits@\!\!}}
\makeatother

\def\\det{\text{\det}}

\def\.5{\frac{1}{2}}

\newcommand{\RN}[1]{%
  \textup{\uppercase\expandafter{\romannumeral#1}}%
}

\renewcommand{\epsilon}{\varepsilon}

\newcounter{marnote}

\begin{document}

\title[Gradient asymptotics of solutions to the Lam\'{e} systems]{Gradient asymptotics of solutions to the Lam\'{e} systems in the presence of two nearly touching $C^{1,\gamma}$-inclusions in all dimensions}

\author[X. Hao]{Xia Hao}
\address[X. Hao]{School of Mathematical Sciences, Beijing Normal University, Beijing 100875, China. }
\email{xiahao0915@163.com}

\author[Z.W. Zhao]{Zhiwen Zhao}

\address[Z.W. Zhao]{Beijing Computational Science Research Center, Beijing 100193, China.}
%\address{2. School of Mathematical Sciences, Beijing Normal University, Beijing 100875, China.}
\email{zwzhao365@163.com, Corresponding author.}

%\footnote{}

\date{\today} % delete this line to display the current date

%%% BEGIN DOCUMENT

%\tableofcontents
\begin{abstract}
In this paper, we establish the asymptotic expressions for the gradient of a solution to the Lam\'{e} systems with partially infinity coefficients as two rigid $C^{1,\gamma}$-inclusions are very close but not touching. The novelty of these asymptotics, which improve and make complete the previous results of Chen-Li (JFA 2021), lies in that they show the optimality of the gradient blow-up rate in dimensions greater than two.
\end{abstract}

\maketitle

%\noindent{\bf{Keywords}}: The Lam\'{e} systems; gradient asymptotics; blow-up factor matrices; $C^{1,\gamma}$-inclusions.

\section{Introduction and principal results}%\label{intro}
\subsection{Background}
In the present work, we consider the Lam\'{e} systems with partially infinity coefficients arising from composites in the presence of two close-to-touching stiff $C^{1,\gamma}$-inclusions and aim at establishing the asymptotic expansions of the gradient of a solution to the Lam\'{e} systems in all dimensions, as the distance $\varepsilon$ between these two inclusions tends to zero. This work is stimulated by the numerical investigation of Babu\u{s}ka et al. \cite{BASL1999} concerning the damage and fracture in composite materials, where the Lam\'{e} system was used and they observed computationally that the size of the strain tensor keeps bounded where the distance between two inclusions goes to zero. In response to such observation there has been much progress over the past two decades. For two touching disks, by using the M\"{o}bius transformation and the maximum principle, Bonnetier and Vogelius \cite{BV2000} proved that the gradient of a solution to the scalar conductivity equation remains bounded. The subsequent work \cite{LV2000} completed by Li and Vogelius extended the result to general divergence form second order elliptic equations with piecewise smooth coefficients in any dimension. This extension especially covers the inclusions of arbitrary smooth shapes. Li and Nirenberg \cite{LN2003} further extended the results in \cite{LV2000} to more general divergence form second order elliptic systems including the Lam\'{e} systems and rigorously demonstrated the boundedness of the strain tensor observed in \cite{BASL1999}. Recently, Dong and Li \cite{DL2019} revealed the explicit dependence of the gradient of the solution to the conductivity equation on the contrast $k$ and the distance $\varepsilon$ between two circular fibers. However, the corresponding questions for more general elliptic equations and systems remain to be answered. See p. 94 of \cite{LV2000} and p. 894 of \cite{LN2003} for more details in terms of these open problems.

Since the antiplane shear model is consistent with the two-dimensional conductivity model, it is significantly important to make clear the singular behavior of the electric field with respect to the distance $\varepsilon$ between inclusions, which is the gradient of a solution to the Laplace equation. It has been demonstrated by many mathematicians that when
the conductivity of the inclusions degenerates to infinity, the generic blow-up rates of the electric field are $\varepsilon^{-1/2}$ in dimension two \cite{AKLLL2007,BC1984,BLY2009,AKL2005,Y2007,Y2009,K1993}, $|\varepsilon\ln\varepsilon|^{-1}$ in dimension three \cite{BLY2009,LY2009,BLY2010,L2012}, and $\varepsilon^{-1}$ in higher dimensions \cite{BLY2009}, respectively. Further, more precise characterizations for the singularities of the concentrated field have been established by Ammari et al. \cite{ACKLY2013}, Bonnetier and Triki \cite{BT2013}, Kang et al. \cite{KLY2013,KLY2014}, Li et al. \cite{LLY2019,Li2020}. The blow-up feature for inclusions of the bow-tie shape was studied by Kang and Yun in \cite{KY2019}. In addition, Calo, Efendiev and Galvis \cite{CEG2014} obtained an asymptotic expression of a solution to elliptic equations as the contrast $k$ is sufficiently small or large. For nonlinear $p$-Laplace equation, Gorb and Novikov \cite{GN2012} gave a qualitative characterization of the concentrated field by using the method of barriers. Ciraolo and Sciammetta \cite{CS2019,CS20192} further extended the results in \cite{GN2012} to the Finsler $p$-Laplacian. For more related works, see \cite{FK1973,GB2005,G2015,KLY2015,KY2020,KL2019} and the references therein.

Recently, the above gradient estimates and asymptotics were extended to the vectorial case, namely, the linear systems of elasticity. In physics, we mainly concern the singular behavior of the stress, which is the gradient of a solution to the Lam\'{e} systems. Li,
Li, Bao and Yin \cite{LLBY2014} created a delicate iterate technique with respect to the energy to obtain the exponentially decaying estimate for the gradient of a solution to a class of elliptic systems with the same boundary data in a narrow region. Bao, Li and Li \cite{BLL2015,BLL2017} applied the iterate technique to establish the pointwise upper bound estimates of the stress concentration for two adjacent strictly convex inclusions in all dimensions. A lower bound of the gradient was constructed by introducing a unified blow-up factor to prove the optimality of the blow-up rates in dimensions two and three in a subsequent work \cite{L2018}. Miao and Zhao \cite{MZ202101} further constructed the explicit stress concentration factors to establish the optimal gradient estimates in the presence of the generalized $m$-convex inclusions in all dimensions. The boundary case when the inclusions are nearly touching the matrix boundary was studied in \cite{BJL2017,LZ2020,MZ202102}. It is worth mentioning that Kang and Yu \cite{KY2019} obtained a precise characterization for the singularities of the stress by introducing singular functions and proved that the stress blows up at the rate of $\varepsilon^{-1/2}$ in two dimensions. Note that the smoothness of inclusions require for at least $C^{2,\gamma}$ in the elasticity problem considered above. Recently, by taking advantage of the Campanato's approach and $W^{1,p}$ estimates for elliptic systems with right hand side in divergence form, Chen and Li \cite{CL2019} developed an adapted version of the iterate technique to establish the upper and lower bound estimates on the gradient of a solution to the Lam\'{e} systems with partially infinity coefficients in the presence of two adjacent $C^{1,\gamma}$-inclusions. The results obtained in \cite{CL2019} comprise of the following two parts: on one hand, the upper bounds on the blow-up rate of the gradient are established in two and three dimensions and a lower bound is constructed in dimension two; on the other hand, an asymptotic expansion of the gradient is only derived under the condition of the symmetric $C^{1,\gamma}$-inclusions and the boundary data of odd function type.

In this paper, by using all the systems of equations in linear decomposition, we capture all the blow-up factor matrices in all dimensions whose elements consist of some certain integrals of the solutions to the case when two inclusions are touching. Thus we obtain the asymptotic formulas of the stress concentration in any dimension. Our idea is different from that in \cite{CL2019}, where only partially systems of equations in linear decomposition were considered. In fact, our idea overcomes the difficulty faced in \cite{CL2019} that the blow-up factors in dimensions greater than two can not be captured to give an optimal information about the blow-up rate of the stress. Moreover, we establish the asymptotic expansions of the stress concentration for the generalized $C^{1,\gamma}$-inclusions and boundary data, which means that we don't need to impose some special symmetric condition on the inclusions and the parity condition on the boundary data as in \cite{CL2019}.

To make our paper self-contained and our exposition clear, let $D\subseteq\mathbb{R}^{d}\,(d\geq2)$ be a bounded open set with $C^{1,\gamma}\,(0<\gamma<1)$ boundary, which contains a pair of $C^{1,\gamma}$-subdomains $D_{1}^{\ast}$ and $D_{2}$ such that these two subdomains touch only at one point and they are far away from the exterior boundary $\partial D$. Namely, after a translation and rotation of the coordinates, if necessary,
\begin{align*}
\partial D_{1}^{\ast}\cap\partial D_{2}=\{0\}\subset\mathbb{R}^{d},
\end{align*}
and
\begin{align*}
D_{1}^{\ast}\subset\{(x',x_{d})\in\mathbb{R}^{d}|\,x_{d}>0\},\quad D_{2}\subset\{(x',x_{d})\in\mathbb{R}^{d}|\,x_{d}<0\}.
\end{align*}
Here and throughout the paper, we use superscript prime to denote ($d-1$)-dimensional domains and variables, such as $B'$ and $x'$. By translating $D_{1}^{\ast}$ by a sufficiently small positive constant $\varepsilon$ along $x_{d}$-axis, we obtain $D_{1}^{\varepsilon}$ as follows:
\begin{align*}
D_{1}^{\varepsilon}:=D_{1}^{\ast}+(0',\varepsilon).
\end{align*}
When there is no possibility of confusion, we drop superscripts and denote
\begin{align*}
D_{1}:=D_{1}^{\varepsilon},\quad\mathrm{and}\quad\Omega:=D\setminus\overline{D_{1}\cup D_{2}}.
\end{align*}

We assume that $\Omega$ and $D_{1}\cup D_{2}$ are, respectively, occupied by two different isotropic and homogeneous elastic materials with different Lam\'{e} constants $(\lambda,\mu)$ and $(\lambda_{1},\mu_{1})$. The elasticity tensors for the inclusions $D_{1}\cup D_{2}$ and the matrix $\Omega$ can be expressed, respectively, as $\mathbb{C}^0$ and $\mathbb{C}^1$, with
$$C_{ijkl}^0=\lambda\delta_{ij}\delta_{kl} +\mu(\delta_{ik}\delta_{jl}+\delta_{il}\delta_{jk}),$$
and
$$C_{ijkl}^1=\lambda_1\delta_{ij}\delta_{kl} +\mu_1(\delta_{ik}\delta_{jl}+\delta_{il}\delta_{jk}),$$
where $i,j,k,l=1,2,...,d$ and $\delta_{ij}$ is the kronecker symbol: $\delta_{ij}=0$ for $i\neq j$, $\delta_{ij}=1$ for $i=j$.

Let $u=(u^{1},u^{2},...,u^{d})^{T}:D\rightarrow\mathbb{R}^{d}$ be the elastic displacement field. For a given boundary data $\varphi=(\varphi^{1},\varphi^{2},...,\varphi^{d})^{T}$, we consider the Dirichlet problem for the Lam\'{e} system with piecewise constant coefficients
\begin{align}\label{La.001}
\begin{cases}
\nabla\cdot \left((\chi_{\Omega}\mathbb{C}^0+\chi_{D_{1}\cup D_{2}}\mathbb{C}^1)e(u)\right)=0,&\hbox{in}~~D,\\
u=\varphi, &\hbox{on}~~\partial{D},
\end{cases}
\end{align}
where $e(u)=\frac{1}{2}\left(\nabla u+(\nabla u)^{T}\right)$ is the elastic strain, $\chi_{\Omega}$ and $\chi_{D_{1}\cup D_{2}}$ are the characteristic functions of $\Omega$ and $D_{1}\cup D_{2}$, respectively.

Let problem \eqref{La.001} satisfy the standard ellipticity condition as follows:
\begin{align*}%\label{ellipticity}
\mu>0,\quad d\lambda+2\mu>0,\quad \mu_1>0,\quad d\lambda_1+2\mu_1>0.
\end{align*}
For $\varphi\in H^{1}( D;\mathbb{R}^{d})$, it is well known that there is a unique variational solution $u\in H^{1}(D;\mathbb{R}^{d})$ to problem (\ref{La.001}), which is also the minimizer of the following energy functional
\begin{align*}
J[u]=\frac{1}{2}\int_{\Omega}\left((\chi_{\Omega}\mathbb{C}^{0}+\chi_{D_{1}\cup D_{2}}\mathbb{C}^{1})e(u),e(u)\right)dx
\end{align*}
on
\begin{align*}
H_{\varphi}^{1}(\Omega;\mathbb{R}^{d}):=\{u\in H^{1}(\Omega;\mathbb{R}^{d})|\,u-\varphi\in H_{0}^{1}(\Omega;\mathbb{R}^{d})\}.
\end{align*}

Define the linear space of rigid displacement in $\mathbb{R}^{d}$ as follows:
\begin{align}\label{LAK01}
\Psi:=\{\psi\in C^1(\mathbb{R}^{d}; \mathbb{R}^{d})\ |\ \nabla\psi+(\nabla\psi)^T=0\}.
\end{align}
Denote by
\begin{align}\label{OPP}
\left\{e_{i},\,x_{k}e_{j}-x_{j}e_{k}\;\big|\;1\leq\,i\leq\,d,\,1\leq\,j<k\leq d\right\}
\end{align}
a basis of $\Psi$, where $\{e_{1},...,e_{d}\}$ is the standard basis of $\mathbb{R}^{d}$. We rewrite them as $\left\{\psi_{\alpha}\big|\,\alpha=1,2,...,\frac{d(d+1)}{2}\right\}$.
%For the convenience of computations, we adopt the following order with respect to $\psi_{\alpha}$: for $\alpha=1,2,...,d$, $\psi_{\alpha}=e_{\alpha}$; for $\alpha=d+1,...,2d-1$, $\psi_{\alpha}=x_{d}e_{\alpha-d}-x_{\alpha-d}e_{d}$; for $\alpha=2d-1,...,\frac{d(d+1)}{2}\,(d\geq3)$, there exist two indices $1\leq i_{\alpha}<j_{\alpha}\leq d$ such that
%$\psi_{\alpha}=(0,...,0,x_{j_{\alpha}},0,...,0,-x_{i_{\alpha}},0,...,0)$.

Let $u_{\lambda_{1},\mu_{1}}$ be the solution of (\ref{La.001}) for fixed $\lambda$ and $\mu$. As proved in the Appendix of \cite{BLL2015}, we have
\begin{align*}%\label{limit}
u_{\lambda_1,\mu_1}\rightarrow u\quad\hbox{in}\ H^1(D; \mathbb{R}^{d}),\quad \hbox{as}\ \min\{\mu_1, d\lambda_1+2\mu_1\}\rightarrow\infty,
\end{align*}
where $u\in H^1(D; \mathbb{R}^{d})$ is a solution of
\begin{align}\label{La.002}
\begin{cases}
\mathcal{L}_{\lambda, \mu}u:=\nabla\cdot(\mathbb{C}^0e(u))=0,\quad&\hbox{in}\ \Omega,\\
u|_{+}=u|_{-},&\hbox{on}\ \partial{D}_{i},\,i=1,2,\\
e(u)=0,&\hbox{in}~D_{i},\,i=1,2,\\
\int_{\partial{D}_{i}}\frac{\partial u}{\partial \nu_0}\big|_{+}\cdot\psi_{\alpha}=0,&i=1,2,\,\alpha=1,2,...,\frac{d(d+1)}{2},\\
u=\varphi,&\hbox{on}\ \partial{D},
\end{cases}
\end{align}
where
\begin{align*}
\frac{\partial u}{\partial \nu_0}\Big|_{+}&:=(\mathbb{C}^0e(u))\nu=\lambda(\nabla\cdot u)\nu+\mu(\nabla u+(\nabla u)^T)\nu,
\end{align*}
and $\nu$ denotes the unit outer normal of $\partial D_{i}$, $i=1,2$. Here and below the subscript $\pm$ indicates the limit from outside and inside the domain, respectively. The existence, uniqueness and regularity of weak solutions to (\ref{La.002}) have been proved in \cite{BLL2015}. Moreover, the $H^{1}$ weak solution $u$ to problem (\ref{La.002}) belongs to $C^1(\overline{\Omega};\mathbb{R}^{d})\cap C^1(\overline{D}_{1}\cup\overline{D}_{2};\mathbb{R}^{d})$.

Suppose that there exists a constant $R>0$, independent of $\varepsilon$, such that $\partial D_{1}$ and $\partial D_{2}$ near the origin are, respectively, the graphs of two $C^{1,\gamma}$ functions $\varepsilon+h_{1}$ and $h_{2}$, and $h_{i}$, $i=1,2$ satisfy that for $\sigma>0$,
\begin{enumerate}
{\it\item[(\bf{H1})]
$h_{1}(x')-h_{2}(x')=\tau|x'|^{1+\gamma}+O(|x'|^{1+\gamma+\sigma}),\;\mathrm{if}\;x'\in B_{2R}',$
\item[(\bf{H2})]
$|\nabla_{x'}h_{i}(x')|\leq \kappa_{1}|x'|^{\gamma},\;\mathrm{if}\;x'\in B_{2R}',\;i=1,2,$
\item[(\bf{H3})]
$\|h_{1}\|_{C^{1,\alpha}(B'_{2R})}+\|h_{2}\|_{C^{1,\alpha}(B'_{2R})}\leq \kappa_{2},$}
\end{enumerate}
where $\tau$ and $\kappa_{i},i=1,2$, are three positive constants independent of $\varepsilon$. Moreover, we suppose that $h_{1}(x')-h_{2}(x')$ is even with respect to $x_{i}$ in $B_{R}'$ for $i=1,...,d-1.$ For $z'\in B'_{R}$ and $0<t\leq2R$, write
\begin{align*}
\Omega_{t}(z'):=&\left\{x\in \mathbb{R}^{d}~\big|~h_{2}(x')<x_{d}<\varepsilon+h_{1}(x'),~|x'-z'|<{t}\right\}.
\end{align*}
We use the abbreviated notation $\Omega_{t}$ to denote $\Omega_{t}(0')$ with its top and bottom boundaries represented by
\begin{align*}
\Gamma^{+}_{r}:=\left\{x\in\mathbb{R}^{d}|\,x_{d}=\varepsilon+h_{1}(x'),\;|x'|<r\right\},~~
\Gamma^{-}_{r}:=\left\{x\in\mathbb{R}^{d}|\,x_{d}=h_{2}(x'),\;|x'|<r\right\},
\end{align*}
respectively.

To begin with, we introduce a Keller-type scalar auxiliary function $\bar{v}\in C^{1,\gamma}(\mathbb{R}^{d})$ such that $\bar{v}=1$ on $\partial D_{1}$, $\bar{v}=0$ on $\partial D_{2}\cup\partial D$,
\begin{align}\label{zh001}
\bar{v}(x',x_{d}):=\frac{x_{d}-h_{2}(x')}{\delta(x')},\;\,\mathrm{in}\;\Omega_{2R},\quad\mbox{and}~\|\bar{v}\|_{C^{2}(\Omega\setminus\Omega_{R})}\leq C,
\end{align}
where
\begin{align}\label{deta}
\delta(x'):=\varepsilon+h_{1}(x')-h_{2}(x').
\end{align}
We then define a family of vector-valued auxiliary functions as follows:
\begin{align}\label{zzwz002}
\bar{u}_{1}^{\alpha}=&\psi_{\alpha}\bar{v},\quad \alpha=1,2,...,\frac{d(d+1)}{2},
\end{align}
where $\psi_{\alpha}$ is defined in (\ref{OPP}).
\subsection{Main results}
Before listing our main results, we first introduce some notations. Set
\begin{align*}
\Gamma_{\gamma}:=&\Gamma\left(\frac{1}{1+\gamma}\right)\Gamma\left(\frac{\gamma}{1+\gamma}\right),
\end{align*}
where $\Gamma(s)=\int^{+\infty}_{0}t^{s-1}e^{-t}dt$, $s>0$ is the Gamma function. Introduce a definite
constant as follows:
\begin{align}\label{zwzh001}
\mathcal{M}_{\gamma,\tau}=\frac{2\Gamma_{\gamma}}{(1+\gamma)\tau^{\frac{1}{1+\gamma}}},
\end{align}
where $\tau$ is defined in condition ({\bf{H1}}). Define some constants related to the Lam\'{e} pair $(\lambda, \mu)$ as follows:
\begin{align}\label{AZ}
(\mathcal{L}_{d}^{1},...,\mathcal{L}_{d}^{d-1},\mathcal{L}_{d}^{d})=(\mu,...,\mu,\lambda+2\mu).
\end{align}
We additionally suppose that
\begin{align}\label{ITERA}
\kappa_{3}\leq\mu,d\lambda+2\mu\leq\frac{1}{\kappa_{3}},\quad\text{for some constant}\;\kappa_{3}>0.
\end{align}
%For $i=0,2$, introduce the blow-up rate indices as follows:
%\begin{align}\label{rate00}
%\rho_{i}(d,m;\varepsilon):=&
%\begin{cases}
%\varepsilon^{\frac{d+i-1}{m}-1},&m>d+i-1,\\
%|\ln\varepsilon|,&m=d+i-1,\\
%1,&m<d+i-1.
%\end{cases}
%\end{align}

Denote $\Omega^{\ast}:=D\setminus\overline{(D_{1}^{\ast}\cup D_{2})}$. For $i,j=1,2,\,\alpha,\beta=1,2,...,\frac{d(d+1)}{2}$, define
\begin{align*}
a_{ij}^{\ast\alpha\beta}=\int_{\Omega^{\ast}}(\mathbb{C}^0e(v_{i}^{\ast\alpha}), e(v_j^{\ast\beta}))dx,\quad b_i^{\ast\alpha}=-\int_{\partial D}\frac{\partial v_{i}^{\ast\alpha}}{\partial \nu_0}\large\Big|_{+}\cdot\varphi,
\end{align*}
where $\varphi\in C^{2}(\partial D;\mathbb{R}^{d})$ is a given function and $v_{i}^{\ast\alpha}\in{C}^{2}(\Omega^{\ast};R^d)$, $i=1,2$, $\alpha=1,2,...,\frac{d(d+1)}{2}$, respectively, verify
\begin{equation}\label{qaz001111}
\begin{cases}
\mathcal{L}_{\lambda,\mu}v_{1}^{\ast\alpha}=0,&\mathrm{in}~\Omega^{\ast},\\
v_{1}^{\ast\alpha}=\psi^{\alpha},&\mathrm{on}~\partial{D}_{1}^{\ast}\setminus\{0\},\\
v_{1}^{\ast\alpha}=0,&\mathrm{on}~\partial{D_{2}}\cup\partial{D},
\end{cases}\quad
\begin{cases}
\mathcal{L}_{\lambda,\mu}v_{2}^{\ast\alpha}=0,&\mathrm{in}~\Omega^{\ast},\\
v_{2}^{\ast\alpha}=\psi^{\alpha},&\mathrm{on}~\partial{D}_{2},\\
v_{2}^{\alpha}=0,&\mathrm{on}~(\partial{D_{1}^{\ast}}\setminus\{0\})\cup\partial{D}.
\end{cases}
\end{equation}
We would like to emphasize that the definition of $a_{ij}^{\ast\alpha\beta}$ is only valid in some cases, see Lemma \ref{lemmabc} below for more concrete details.

Unless otherwise stated, in the following we use $C$ to denote a positive constant, whose values may differ from line to line, depending only on $\gamma,d,\tau,\kappa_{1},\kappa_{2},R$ and an upper bound of the $C^{1,\alpha}$ norms of $\partial D_{1}$, $\partial D_{2}$ and $\partial D$, but not on $\varepsilon$. Moreover, we use $O(1)$ to denote some quantity satisfying  $|O(1)|\leq\,C$. Note that from the standard elliptic theory (see Agmon et al. \cite{ADN1959,ADN1964}), we obtain
\begin{align*}
\|\nabla u\|_{L^{\infty}(\Omega\setminus\Omega_{R})}\leq\,C\|\varphi\|_{C^{2}(\partial D)}.
\end{align*}
Then we only need to make clear the singular behavior of $\nabla u$ in the narrow region $\Omega_{R}$.

For $d=2$, we introduce the blow-up factor matrices as follows:
\begin{gather}\label{ZWZML001}
\mathbb{F}_{0}^{\ast}:=\begin{pmatrix}a_{11}^{\ast33}&\sum\limits_{i=1}^{2}a_{i1}^{\ast31}&\sum\limits_{i=1}^{2}a_{i1}^{\ast32}&\sum\limits_{i=1}^{2}a_{i1}^{\ast33} \\ \sum\limits_{j=1}^{2}a_{1j}^{\ast13}&\sum\limits_{i,j=1}^{2}a_{ij}^{\ast11}&\sum\limits_{i,j=1}^{2}a_{ij}^{\ast12}&\sum\limits_{i,j=1}^{2}a_{ij}^{\ast13}\\
\sum\limits_{j=1}^{2}a_{1j}^{\ast23}&\sum\limits_{i,j=1}^{2}a_{ij}^{\ast21}&\sum\limits_{i,j=1}^{2}a_{ij}^{\ast22}&\sum\limits_{i,j=1}^{2}a_{ij}^{\ast23}\\
\sum\limits_{j=1}^{2}a_{1j}^{\ast33}&\sum\limits_{i,j=1}^{2}a_{ij}^{\ast31}&\sum\limits_{i,j=1}^{2}a_{ij}^{\ast32}&\sum\limits_{i,j=1}^{2}a_{ij}^{\ast33}
\end{pmatrix},
\end{gather}
and
\begin{gather}\label{ZWZML002}
\mathbb{F}_{1}^{\ast\alpha}:=\begin{pmatrix} b_{1}^{\ast\alpha}&a_{11}^{\ast\alpha3}&\sum\limits_{i=1}^{2}a_{i1}^{\ast\alpha1}&\sum\limits_{i=1}^{2}a_{i1}^{\ast\alpha2}&\sum\limits_{i=1}^{2}a_{i1}^{\ast\alpha3} \\
b_{1}^{\ast3}&a_{11}^{\ast33}&\sum\limits_{i=1}^{2}a_{i1}^{\ast31}&\sum\limits_{i=1}^{2}a_{i1}^{\ast32}&\sum\limits_{i=1}^{2}a_{i1}^{\ast33}\\
\sum\limits_{i=1}^{2}b_{i}^{\ast1}&\sum\limits_{j=1}^{2}a_{1j}^{\ast13}&\sum\limits_{i,j=1}^{2}a_{ij}^{\ast11}&\sum\limits_{i,j=1}^{2}a_{ij}^{\ast12}&\sum\limits_{i,j=1}^{2}a_{ij}^{\ast13}\\
\sum\limits_{i=1}^{2}b_{i}^{\ast2}&\sum\limits_{j=1}^{2}a_{1j}^{\ast23}&\sum\limits_{i,j=1}^{2}a_{ij}^{\ast21}&\sum\limits_{i,j=1}^{2}a_{ij}^{\ast22}&\sum\limits_{i,j=1}^{2}a_{ij}^{\ast23}\\
\sum\limits_{i=1}^{2}b_{i}^{\ast3}&\sum\limits_{j=1}^{2}a_{1j}^{\ast33}&\sum\limits_{i,j=1}^{2}a_{ij}^{\ast31}&\sum\limits_{i,j=1}^{2}a_{ij}^{\ast32}&\sum\limits_{i,j=1}^{2}a_{ij}^{\ast33}
\end{pmatrix},\;\,\alpha=1,2,
\end{gather}
and
\begin{gather}\label{ZWZML003}
\mathbb{F}_{1}^{\ast3}:=\begin{pmatrix}
b_{1}^{\ast3}&\sum\limits_{i=1}^{2}a_{i1}^{\ast31}&\sum\limits_{i=1}^{2}a_{i1}^{\ast32}&\sum\limits_{i=1}^{2}a_{i1}^{\ast33} \\ \sum\limits_{i=1}^{2}b_{i}^{\ast1}&\sum\limits_{i,j=1}^{2}a_{ij}^{\ast11}&\sum\limits_{i,j=1}^{2}a_{ij}^{\ast12}&\sum\limits_{i,j=1}^{2}a_{ij}^{\ast13}\\
\sum\limits_{i=1}^{2}b_{i}^{\ast2}&\sum\limits_{i,j=1}^{2}a_{ij}^{\ast21}&\sum\limits_{i,j=1}^{2}a_{ij}^{\ast22}&\sum\limits_{i,j=1}^{2}a_{ij}^{\ast23}\\
\sum\limits_{i=1}^{2}b_{i}^{\ast3}&\sum\limits_{i,j=1}^{2}a_{ij}^{\ast31}&\sum\limits_{i,j=1}^{2}a_{ij}^{\ast32}&\sum\limits_{i,j=1}^{2}a_{ij}^{\ast33}
\end{pmatrix}.
\end{gather}

For the remaining term, we denote
\begin{align}\label{GC002}
\varepsilon(\gamma,\sigma):=&
\begin{cases}
\varepsilon^{\min\{\frac{\sigma}{1+\gamma},\frac{\gamma^{2}}{2(1+2\gamma)(1+\gamma)^{2}}\}},&\gamma>\sigma,\\
\varepsilon^{\frac{\gamma^{2}}{2(1+2\gamma)(1+\gamma)^{2}}},&0<\gamma\leq\sigma.
\end{cases}
\end{align}

The first main result is now presented as follows.
\begin{theorem}\label{ZHthm002}
Let $D_{1},D_{2}\subset D\subseteq\mathbb{R}^{2}$ be defined as above, conditions $\rm{(}${\bf{H1}}$\rm{)}$--$\rm{(}${\bf{H3}}$\rm{)}$ hold, $\det\mathbb{F}^{\ast\alpha}_{1}\neq0,$ $\alpha=1,2,3$. Let $u\in H^{1}(D;\mathbb{R}^{2})\cap C^{1}(\overline{\Omega};\mathbb{R}^{2})$ be the solution of (\ref{La.002}). Then for a sufficiently small $\varepsilon>0$ and $x\in\Omega_{R}$,
\begin{align*}
\nabla u=&\sum\limits_{\alpha=1}^{2}\frac{\det\mathbb{F}_{1}^{\ast\alpha}}{\det \mathbb{F}_{0}^{\ast}}\frac{\varepsilon^{\frac{\gamma}{1+\gamma}}(1+O(\varepsilon(\gamma,\sigma)))}{\mathcal{L}_{2}^{\alpha}\mathcal{M}_{\gamma,\tau}}\nabla\bar{u}_{1}^{\alpha}\notag\\
&+\frac{\det\mathbb{F}_{1}^{\ast3}}{\det \mathbb{F}_{0}^{\ast}}(1+O(\varepsilon^{\frac{\gamma}{2(1+2\gamma)}}))\nabla\bar{u}_{1}^{3}+O(1)\delta^{-\frac{1-\gamma}{1+\gamma}}\|\varphi\|_{C^{0}(\partial D)},
\end{align*}
where $\delta$ is defined in \eqref{deta}, the explicit auxiliary functions $\bar{u}_{1}^{\alpha}$, $\alpha=1,2,3$ are defined in \eqref{zzwz002} in the case of $d=2$, the constant $\mathcal{M}_{\gamma,\tau}$ is defined in \eqref{zwzh001}, the Lam\'{e} constants $\mathcal{L}_{2}^{\alpha}$, $\alpha=1,2$ is defined in \eqref{AZ}, the blow-up factor matrices $\mathbb{F}_{0}^{\ast}$ and $\mathbb{F}^{\ast\alpha}_{1},$ $\alpha=1,2,3$ are defined by \eqref{ZWZML001}--\eqref{ZWZML003}, the rest term $\varepsilon(\gamma,\sigma)$ is defined in \eqref{GC002}.

\end{theorem}
\begin{remark}
The asymptotic expansion in Theorem \ref{ZHthm002} improves the corresponding results in \cite{CL2019} in terms of the following two aspects: first, the gradient estimates in Theorems 1.1 and 1.3 of \cite{CL2019} are improved to be a precise asymptotic formula here; second, we get rid of the symmetric assumptions on the domain and boundary data added in Theorem 1.5 of \cite{CL2019} and then obtain the asymptotic expression in Theorem \ref{ZHthm002} for the more generalized $C^{1,\gamma}$-inclusions.
\end{remark}
\begin{remark}
The asymptotic expression in Theorem \ref{ZHthm002}, together with the result in Theorem \ref{ZHthm003}, completely solves the optimality of the blow-up rate of the stress in all dimensions. Note that for $\alpha=1,2,...,\frac{d(d+1)}{2}$, the leading singularity of $\nabla\bar{u}_{1}^{\alpha}$ lies in $\partial_{x_{d}}\bar{u}_{1}^{\alpha}=\psi_{\alpha}\delta^{-1}$. Furthermore, for $\alpha=1,2,...,d$, $|\partial_{x_{d}}\bar{u}_{1}^{\alpha}|$ blows up at the rate of $\varepsilon^{-1}$ in the ($d-1$)-dimensional ball $\{|x'|\leq\varepsilon^{\frac{1}{1+\gamma}}\}\cap\Omega$, while it blows up at the rate of $\varepsilon^{-\frac{\gamma}{1+\gamma}}$ on the cylinder surface $\{|x'|=\varepsilon^{\frac{1}{1+\gamma}}\}\cap\Omega$ for $\alpha=d+1,...,\frac{d(d+1)}{2}$. Then in light of decomposition \eqref{Decom002}, we see from the results in Theorems \ref{ZHthm002} and \ref{ZHthm003} that the singular parts of $\nabla u$ consist of the following two parts: one of them is $\sum_{\alpha=1}^{d}(C_{1}^\alpha-C_{2}^\alpha)\nabla{v}_{1}^\alpha$ with its singularity being, respectively, of order $O(\varepsilon^{-\frac{1}{1+\gamma}})$ and $O(\varepsilon^{-1})$ in two dimensions and higher dimensions; the other part is $\sum_{\alpha=d+1}^{\frac{d(d+1)}{2}}(C_{1}^\alpha-C_{2}^\alpha)\nabla{v}_{1}^\alpha$ with its singularity of order $O(\varepsilon^{-\frac{\gamma}{1+\gamma}})$ in all dimensions. Then $\nabla u$ blows up at the rate of $\varepsilon^{-\frac{1}{1+\gamma}}$ in the case of $d=2$ and $\varepsilon^{-1}$ in the case of $d\geq3$, respectively.

\end{remark}
\begin{remark}
In fact, we can conclude from the assumed condition $\det\mathbb{F}^{\ast\alpha}_{1}\neq0$, $\alpha=1,2,3$ in Theorem \ref{ZHthm002} that $\varphi\not\equiv0$ on $\partial D$. If not, suppose that $\varphi\equiv0$ on $\partial D$. Then we obtain that $b_i^{\alpha}=-\int_{\partial D}\frac{\partial v_{i}^{\alpha}}{\partial \nu_0}|_{+}\cdot\varphi=0$ and thus $\det\mathbb{F}^{\ast\alpha}_{1}=0$. This is a contradiction. Additionally, it is worth emphasizing that it is not easy to give some special examples in terms of the domain and the boundary data such that $\det\mathbb{F}^{\ast\alpha}_{1}\neq0$. This is primarily caused by the fact that the blow-up factor matrix $\mathbb{F}^{\ast\alpha}_{1}$ doesn't have symmetrical characteristic of the structure such that it is difficult to deal with them by the same argument as in the proof of $\det\mathbb{F}^{\ast}_{0}\neq0$ below. Finally, it will be of interest to compute the blow-up factor matrix $\mathbb{F}^{\ast\alpha}_{1}$ by using numerical techniques in future.

\end{remark}

For $d\geq3$, we introduce the blow-up factor matrices as follows:
\begin{align*}
&\mathbb{A}^{\ast}=(a_{11}^{\ast\alpha\beta})_{\frac{d(d+1)}{2}\times\frac{d(d+1)}{2}},\quad \mathbb{B}^{\ast}=\bigg(\sum\limits^{2}_{i=1}a_{i1}^{\ast\alpha\beta}\bigg)_{\frac{d(d+1)}{2}\times\frac{d(d+1)}{2}},\notag\\
&\mathbb{C}^{\ast}=\bigg(\sum\limits^{2}_{j=1}a_{1j}^{\ast\alpha\beta}\bigg)_{\frac{d(d+1)}{2}\times\frac{d(d+1)}{2}},\quad \mathbb{D}^{\ast}=\bigg(\sum\limits^{2}_{i,j=1}a_{ij}^{\ast\alpha\beta}\bigg)_{\frac{d(d+1)}{2}\times\frac{d(d+1)}{2}}.
\end{align*}
For $\alpha=1,2,...,\frac{d(d+1)}{2}$, we replace the elements of $\alpha$-th column in the matrix $\mathbb{A}^{\ast}$ and $\mathbb{C}^{\ast}$ by column vectors $\Big(b_{1}^{\ast1},...,b_{1}^{\frac{\ast d(d+1)}{2}}\Big)^{T}$ and $\Big(\sum\limits_{i=1}^{2}b_{i}^{\ast1},...,\sum\limits_{i=1}^{2}b_{i}^{\ast\frac{d(d+1)}{2}}\Big)^{T}$, respectively, and then denote these two new matrices by $\mathbb{A}_{2}^{\ast\alpha}$ and $\mathbb{C}_{2}^{\ast\alpha}$ as follows:
\begin{gather*}
\mathbb{A}_{2}^{\ast\alpha}=
\begin{pmatrix}
a_{11}^{\ast11}&\cdots&b_{1}^{\ast1}&\cdots&a_{11}^{\ast1\,\frac{d(d+1)}{2}} \\\\ \vdots&\ddots&\vdots&\ddots&\vdots\\\\a_{11}^{\ast\frac{d(d+1)}{2}\,1}&\cdots&b_{1}^{\ast\frac{d(d+1)}{2}}&\cdots&a_{11}^{\ast\frac{d(d+1)}{2}\,\frac{d(d+1)}{2}}
\end{pmatrix},
\end{gather*}
and
\begin{gather*}
\mathbb{C}_{2}^{\ast\alpha}=
\begin{pmatrix}
\sum\limits_{j=1}^{2} a_{1j}^{\ast11}&\cdots&\sum\limits_{i=1}^{2}b_{i}^{\ast1}&\cdots&\sum\limits_{j=1}^{2} a_{1j}^{\ast1\,\frac{d(d+1)}{2}} \\\\ \vdots&\ddots&\vdots&\ddots&\vdots\\\\\sum\limits_{j=1}^{2} a_{1j}^{\ast\frac{d(d+1)}{2}\,1}&\cdots&\sum\limits_{i=1}^{2}b_{i}^{\ast\frac{d(d+1)}{2}}&\cdots&\sum\limits_{j=1}^{2} a_{1j}^{\ast\frac{d(d+1)}{2}\,\frac{d(d+1)}{2}}
\end{pmatrix}.
\end{gather*}
Define
\begin{align}\label{GC009}
\mathbb{F}^{\ast\alpha}_{2}=\begin{pmatrix} \mathbb{A}^{\ast\alpha}_{2}&\mathbb{B}^{\ast} \\  \mathbb{C}^{\ast\alpha}_{2}&\mathbb{D}^{\ast}
\end{pmatrix},\;\,\alpha=1,2,...,\frac{d(d+1)}{2},\quad \mathbb{F}=\begin{pmatrix} \mathbb{A}^{\ast}&\mathbb{B}^{\ast} \\  \mathbb{C}^{\ast}&\mathbb{D}^{\ast}
\end{pmatrix}.
\end{align}

Denote
\begin{align}\label{NZKL001}
\bar{\varepsilon}(\gamma,d)=
\begin{cases}
\varepsilon^{\frac{\gamma^{2}(1-\gamma)}{2(1+2\gamma)(1+\gamma)^{2}}},&d=3,\\
\varepsilon^{\frac{\gamma^{2}}{2(1+2\gamma)(1+\gamma)^{2}}\min\{1+\gamma,2-\gamma\}},&d=4,\\
\varepsilon^{\frac{\gamma^{2}}{2(1+2\gamma)(1+\gamma)}},&d\geq5.
\end{cases}
\end{align}

Then, we state the second main result of this paper in the following.
\begin{theorem}\label{ZHthm003}
Let $D_{1},D_{2}\subset D\subseteq\mathbb{R}^{d}\,(d\geq3)$ be defined as above, conditions $\rm{(}${\bf{H1}}$\rm{)}$--$\rm{(}${\bf{H3}}$\rm{)}$ hold, and $\det\mathbb{F}_{2}^{\ast\alpha}\neq0,$ $\alpha=1,2,...,\frac{d(d+1)}{2}$. Let $u\in H^{1}(D;\mathbb{R}^{d})\cap C^{1}(\overline{\Omega};\mathbb{R}^{d})$ be the solution of (\ref{La.002}). Then for a sufficiently small $\varepsilon>0$ and $x\in\Omega_{R}$,
\begin{align}\label{PAGN001}
\nabla u=&\sum\limits_{\alpha=1}^{\frac{d(d+1)}{2}}\frac{\det\mathbb{F}_{2}^{\ast\alpha}}{\det \mathbb{F}^{\ast}}(1+O(\bar{\varepsilon}(\gamma,d)))\nabla\bar{u}^{\alpha}_{1}+O(1)\delta^{-\frac{1}{1+\gamma}}\|\varphi\|_{C^{0}(\partial D)},
\end{align}
where $\delta$ is defined in \eqref{deta}, the explicit auxiliary functions $\bar{u}_{1}^{\alpha}$, $\alpha=1,2,...,\frac{d(d+1)}{2}$ are defined in \eqref{zzwz002}, the constant $\mathcal{M}_{\gamma,\tau}$ is defined in \eqref{zwzh001}, the blow-up factor matrices $\mathbb{F}^{\ast}$ and $\mathbb{F}_{2}^{\ast\alpha},$ $\alpha=1,2,...,\frac{d(d+1)}{2}$, are defined by \eqref{GC009}, the rest term $\bar{\varepsilon}(\gamma,d)$ is defined by \eqref{NZKL001}.

\end{theorem}
\begin{remark}
By contrast with the results in \cite{CL2019}, the primary advantage of our idea lies in capturing the blow-up factor matrices in dimensions greater than two to obtain a unified asymptotic expansion in \eqref{PAGN001}, which completely solves the optimality of the blow-up rate of the stress in higher dimensions.

\end{remark}

For the more generalized $C^{1,\gamma}$-inclusions satisfying the following condition:
\begin{align}\label{APRHU001}
\tau_{1}|x'|^{1+\gamma}\leq (h_{1}-h_{2})(x')\leq\tau_{2}|x'|^{1+\gamma},\quad\mathrm{for}\;|x'|\leq2R,\;\tau_{i}>0,\;i=1,2,
\end{align}
by applying the proofs of Theorems \ref{ZHthm002}-\ref{ZHthm003} with a minor modification, we establish the optimal pointwise upper and lower bounds on the gradients as follows:
\begin{corollary}\label{PASL001}
Let $D_{1},D_{2}\subset D\subseteq\mathbb{R}^{d}\,(d\geq3)$ be defined as above, conditions \eqref{APRHU001} and ({\bf{H2}})--({\bf{H3}}) hold. Let $u\in H^{1}(D;\mathbb{R}^{d})\cap C^{1}(\overline{\Omega};\mathbb{R}^{d})$ be the solution of (\ref{La.002}). Then for a sufficiently small $\varepsilon>0$,

$(a)$ if $d=2$, there exist some integer $1\leq\alpha_{0}\leq2$ such that $\det\mathbb{F}_{1}^{\ast\alpha_{0}}\neq0$, then for $x\in\{x'=0'\}\cap\Omega$,
\begin{align*}
\frac{\tau_{1}^{\frac{1}{1+\gamma}}|\det\mathbb{F}_{1}^{\ast\alpha_{0}}|}{C|\mathcal{L}_{2}^{\alpha_{0}}||\det\mathbb{F}_{0}^{\ast}|}\frac{1}{\varepsilon^{\frac{1}{1+\gamma}}}\leq|\nabla u|\leq \frac{\max\limits_{1\leq\alpha\leq 2}\tau_{2}^{\frac{1}{1+\gamma}}|\mathcal{L}_{2}^{\alpha}|^{-1}|\det\mathbb{F}_{1}^{\ast\alpha}|}{|\det\mathbb{F}_{0}^{\ast}|}\frac{C}{\varepsilon^{\frac{1}{1+\gamma}}};
\end{align*}

$(b)$ if $d\geq3$, there exist some integer $1\leq \alpha_{0}\leq d$ such that $\det\mathbb{F}_{2}^{\ast\alpha_{0}}\neq0$, then for $x\in\{x'=0'\}\cap\Omega$,
\begin{align*}
\frac{|\det\mathbb{F}_{2}^{\ast\alpha_{0}}|}{|\det \mathbb{F}^{\ast}|}\frac{1}{C\varepsilon}\leq|\nabla u|\leq \frac{\max\limits_{1\leq\alpha\leq d}|\det\mathbb{F}_{2}^{\ast\alpha}|}{|\det \mathbb{F}^{\ast}|}\frac{C}{\varepsilon},
\end{align*}
where the blow-up factor matrices $\mathbb{F}_{0}^{\ast}$ and $\mathbb{F}^{\ast\alpha}_{1},$ $\alpha=1,2,$ are defined by \eqref{ZWZML001}--\eqref{ZWZML003}, the blow-up factor matrices $\mathbb{F}^{\ast}$ and $\mathbb{F}^{\ast\alpha}_{2},$ $\alpha=1,2,...,d$ are defined in \eqref{GC009}.
\end{corollary}

\begin{remark}
We construct the optimal lower bounds on the gradient in Corollary \ref{PASL001} by capturing the blow-up factor matrices, which answers the remaining question in Theorem 1.6 of \cite{CL2019}. Moreover, the gradient estimate results in \cite{CL2019} were improved in virtue of these blow-up factor matrices captured here.
\end{remark}

The rest of this paper is organized as follows. In Section \ref{SEC003}, we decompose the gradient $\nabla u$ into a singular part and a regular part. We then give the proofs of Theorems \ref{ZHthm002} and \ref{ZHthm003} in Section \ref{SEC004}, which mainly consist of the asymptotic expansions of $\nabla v_{1}^{\alpha}$ and $C_{1}^{\alpha}-C_{2}^{\alpha}$, $\alpha=1,2,...,\frac{d(d+1)}{2}$, where $v_{1}^{\alpha}$ is defined in \eqref{qaz001} and the proof of the asymptotic expression of $C_{1}^{\alpha}-C_{2}^{\alpha}$ is left in Section \ref{SEC005}. Section \ref{SEC006} is dedicated to the presentation of Example \ref{coro00389}.

\section{Preliminary}\label{SEC003}

\subsection{Properties of the elasticity tensor $\mathbb{C}^{0}$}
With regard to the isotropic elastic material, let
\begin{align*}
\mathbb{C}^{0}:=(C_{ijkl}^{0})=(\lambda\delta_{ij}\delta_{kl}+\mu(\delta_{ik}\delta_{jl}+\delta_{il}\delta_{jk})),\quad \mu>0,\;d\lambda+2\mu>0.
\end{align*}
Note that the components $C_{ijkl}^{0}$ satisfy the following symmetry property:
\begin{align}\label{symm}
C_{ijkl}^{0}=C_{klij}^{0}=C_{klji}^{0},\quad i,j,k,l=1,2,...,d.
\end{align}
For every pair of $d\times d$ matrices $\mathbb{A}=(a_{ij})$ and $\mathbb{B}=(b_{ij})$, we define
\begin{align*}
(\mathbb{C}^{0}\mathbb{A})_{ij}=\sum_{k,l=1}^{n}C_{ijkl}^{0}a_{kl},\quad\hbox{and}\quad(\mathbb{A},\mathbb{B})\equiv \mathbb{A}:\mathbb{B}=\sum_{i,j=1}^{d}a_{ij}b_{ij}.
\end{align*}
Then
$$(\mathbb{C}^{0}\mathbb{A},\mathbb{B})=(\mathbb{A}, \mathbb{C}^{0}\mathbb{B}).$$
In view of (\ref{symm}), we obtain that the tensor $\mathbb{C}^{0}$ satisfies the ellipticity condition, that is, for every $d\times d$ real symmetric matrix $\xi=(\xi_{ij})$,
\begin{align}\label{ellip}
\min\{2\mu, d\lambda+2\mu\}|\xi|^2\leq(\mathbb{C}^{0}\xi, \xi)\leq\max\{2\mu, d\lambda+2\mu\}|\xi|^2,
\end{align}
where $|\xi|^2=\sum\limits_{ij}\xi_{ij}^2.$ Furthermore,
\begin{align*}
\min\{2\mu, d\lambda+2\mu\}|\mathbb{A}+\mathbb{A}^T|^2\leq(\mathbb{C}(\mathbb{A}+\mathbb{A}^T), (\mathbb{A}+\mathbb{A}^T)).
\end{align*}
In addition, for any open set $O$ and $u, v\in C^2(O;\mathbb{R}^{d})$, we see
\begin{align}\label{Le2.01222}
\int_O(\mathbb{C}^0e(u), e(v))\,dx=-\int_O\left(\mathcal{L}_{\lambda, \mu}u\right)\cdot v+\int_{\partial O}\frac{\partial u}{\partial \nu_0}\Big|_{+}\cdot v.
\end{align}

\subsection{Solution decomposition}\label{sec_thm1}

As shown in \cite{BLL2015,BLL2017}, we decompose the solution $u$ of \eqref{La.002} as follows:
\begin{equation}\label{Decom}
u(x)=\sum_{\alpha=1}^{\frac{d(d+1)}{2}}C_1^{\alpha}v_{1}^{\alpha}(x)+\sum_{\alpha=1}^{\frac{d(d+1)}{2}}C_2^{\alpha}v_2^{\alpha}(x)+v_{0}(x),\quad~x\in\Omega,
\end{equation}
where the constants $C_{i}^{\alpha}$, $i=1,2,\,\alpha=1,2,...,\frac{d(d+1)}{2}$ are free constants to be determined by the fourth line of \eqref{La.002}, $v_{0}$ and $v_{i}^{\alpha}\in{C}^{2}(\Omega;\mathbb{R}^d)$, $i=1,2$, $\alpha=1,2,...,\frac{d(d+1)}{2}$, respectively, satisfy
\begin{equation}\label{qaz001}
\begin{cases}
\mathcal{L}_{\lambda,\mu}v_{0}=0,&\mathrm{in}~\Omega,\\
v_{0}=0,&\mathrm{on}~\partial{D}_{1}\cup\partial{D_{2}},\\
v_{0}=\varphi,&\mathrm{on}~\partial{D},
\end{cases}\quad
\begin{cases}
\mathcal{L}_{\lambda,\mu}v_{i}^{\alpha}=0,&\mathrm{in}~\Omega,\\
v_{i}^{\alpha}=\psi^{\alpha},&\mathrm{on}~\partial{D}_{i},~i=1,2,\\
v_{i}^{\alpha}=0,&\mathrm{on}~\partial{D_{j}}\cup\partial{D},~j\neq i.
\end{cases}
\end{equation}
From \eqref{Decom}, we see
\begin{align}\label{Decom002}
\nabla{u}=&\sum_{\alpha=1}^{\frac{d(d+1)}{2}}(C_{1}^\alpha-C_{2}^\alpha)\nabla{v}_{1}^\alpha+\sum_{\alpha=1}^{\frac{d(d+1)}{2}}C_{2}^\alpha\nabla({v}_{1}^\alpha+{v}_{2}^\alpha)+\nabla{v}_{0}.
\end{align}
In light of \eqref{Decom002}, we decompose $\nabla u$ into two parts as follows: the first part $\sum_{\alpha=1}^{\frac{d(d+1)}{2}}(C_{1}^\alpha-C_{2}^\alpha)\nabla{v}_{1}^\alpha$ is the singular part and blows up; the other part $\sum_{\alpha=1}^{\frac{d(d+1)}{2}}C_{2}^\alpha\nabla({v}_{1}^\alpha+{v}_{2}^\alpha)+\nabla{v}_{0}$ is the regular part and possesses exponentially decaying property in the shortest segment between two inclusions. The precise statements for these results are given in the following sections.

\subsection{A general boundary value problem}

To begin with, for two given vector-valued functions $\psi\in C^{1,\gamma}(\partial D_{1};\mathbb{R}^{d})$ and $\phi\in C^{1,\gamma}(\partial D_{2};\mathbb{R}^{d})$, we consider the general boundary value problem as follows:
\begin{equation}\label{P2.008}
\begin{cases}
\mathcal{L}_{\lambda,\mu}v:=\nabla\cdot(\mathbb{C}^{0}e(v))=0,\quad\;\,&\mathrm{in}\;\,\Omega,\\
v=\psi(x),&\mathrm{on}\;\,\partial D_{1},\\
v=\phi(x),&\mathrm{on}\;\,\partial D_{2},\\
v=0,&\mathrm{on}\;\,\partial D.
\end{cases}
\end{equation}

Define a vector-valued auxiliary function as follows:
\begin{align}\label{CAN01}
\tilde{v}=&\psi(x',\varepsilon+h_{1}(x'))\bar{v}+\phi(x',h_{2}(x'))(1-\bar{v}),\quad \mathrm{in}\;\Omega_{2R},
\end{align}
where $\bar{v}$ is defined by \eqref{zh001}. Denote
\begin{align}\label{ANC001}
\mathcal{R}_{\delta}(\psi,\phi):=&\delta^{-\frac{1}{1+\gamma}}|\psi(x',\varepsilon+h_{1}(x'))-\phi(x',h_{2}(x'))|\notag\\
&+\|\psi\|_{C^{1}(\partial D_{1})}+\|\phi\|_{C^{1}(\partial D_{2})},
\end{align}
where $\delta$ is defined in \eqref{deta}.

\begin{theorem}\label{thm8698}
Assume as above. Let $v$ be the weak solution of \eqref{P2.008}. Then for a sufficiently small $\varepsilon>0$,
\begin{align*}
\nabla v=\nabla\tilde{v}+O(1)\mathcal{R}_{\delta}(\psi,\phi),
\end{align*}
where $\delta$ is defined in \eqref{deta}, the leading term $\tilde{v}$ is defined by \eqref{CAN01}, the residual part $\mathcal{R}_{\delta}(\psi,\phi)$ is defined by \eqref{ANC001}.
\end{theorem}

For the purpose of proving Theorem \ref{thm8698}, we will utilize the adapted version of the iterate technique developed in \cite{CL2019}. To begin with, we recall the following two lemmas, which are Theorem 2.3 and Theorem 2.4 in \cite{CL2019}. For the sake of readability and presentation, in this section we write $\partial_{j}:=\partial_{x_{j}}$, $j=1,2,...,d$. Let $Q\subset\mathbb{R}^{d}\,(d\geq2)$ be a bounded domain with $C^{1,\gamma}\,(0<\gamma<1)$ boundary portion $\Gamma\subset\partial Q$. The boundary value problem is described as follows:
\begin{align}\label{ADCo1}
\begin{cases}
-\partial_{j}(C_{ijkl}^{0}\partial_{l}w^{k})=\partial_{j}f_{ij},&in\; Q,\\
w=0,&on\;\Gamma,
\end{cases}
\end{align}
where $f_{ij}\in C^{0,\gamma}(Q)$, $i,j,k,l=1,2,...,d$, and the Einstein summation convention in repeated indices is used.
\begin{lemma}\label{CL001}{$\mathrm{(}$$C^{1,\gamma}$ estimates$\mathrm{)}$.}
Let $w\in H^{1}(Q;\mathbb{R}^{d})\cap C^{1}(Q\cup\Gamma;\mathbb{R}^{d})$ be the solution of \eqref{ADCo1}. Then for any subdomain $Q'\subset\subset Q\cup\Gamma$,
\begin{align}\label{GNA001}
\|w\|_{C^{1,\gamma}(Q')}\leq C\left(\|w\|_{L^{\infty}(Q)}+[F]_{\alpha,Q}\right),
\end{align}
where $F:=(f_{ij})$ and $C=C(d,\gamma,Q',Q)$.
\end{lemma}
The H\"{o}lder semi-norm of matrix-valued function $F=(f_{ij})$ is defined as follows:
\begin{align*}
[F]_{\gamma,Q}:=\max_{1\leq i,j\leq d}[f_{ij}]_{\gamma,Q}\quad\mathrm{and}\quad[f_{ij}]_{\gamma,Q}=\sup_{x,y\in Q,x\neq y}\frac{|f_{ij}(x)-f_{ij}(y)|}{|x-y|^{\gamma}}.
\end{align*}

\begin{lemma}\label{CL002}{$\mathrm{(}$$W^{1,p}$ estimates$\mathrm{)}$.}
Assume that $Q$ and $\Gamma$ are defined as in Lemma \ref{CL001}. Let $w\in H^{1}(Q;\mathbb{R}^{d})$ be the weak solution of \eqref{ADCo1} with $f_{ij}\in C^{0,\gamma}(Q)$, $0<\gamma<1$ and $i,j=1,2,...,d$. Then, for any $2\leq p<\infty$ and $Q'\subset\subset Q\cup\Gamma$,
\begin{align}\label{LNZ001}
\|w\|_{W^{1,p}(Q')}\leq C(\|w\|_{H^{1}(Q)}+\|F\|_{L^{p}(Q)}),
\end{align}
where $C=C(\lambda,\mu,p,Q')$ and $F:=(f_{i}^{k})$. In particular, if $p>d$, we have
\begin{align}\label{LNZ002}
\|w\|_{C^{0,\bar{\gamma}}(Q')}\leq C(\|w\|_{H^{1}(Q)}+[F]_{\alpha,Q}),
\end{align}
where $0<\bar{\gamma}\leq1-d/p$ and $C=C(\lambda,\mu,\bar{\gamma},p,Q')$.
\end{lemma}
\begin{remark}
We would like to emphasize that the results in Lemmas \ref{CL001} and \ref{CL002} improve the classical $C^{1,\gamma}$ estimates and $W^{1,p}$ estimates of \cite{GM2013} in the setting with partially zero boundary data, which is vitally important to build the following iteration scheme.
\end{remark}
For readers' convenience, we leave the detailed proofs of Lemmas \ref{CL001} and \ref{CL002} in the Appendix.
\begin{proof}[The proof of Theorem \ref{thm8698}]

Without loss of generality, we let $\phi=0$ on $\partial D_{2}$ in (\ref{P2.008}). To begin with, we decompose the solution $v$ of (\ref{P2.008}) as follows:
$$v=\sum^{d}_{i=1}v_{i},$$
where $v_{i}=(v_{i}^{1},v_{i}^{2},...,v_{i}^{d})^{T}$, $i=1,2,...,d$, with $v_{i}^{j}=0$ for $j\neq i$, and $v_{i}$ verifies the following boundary value problem
\begin{align*}
\begin{cases}
  \mathcal{L}_{\lambda,\mu}v_{i}:=\nabla\cdot(\mathbb{C}^0e(v_{i}))=0,\quad&
\hbox{in}\  \Omega,  \\
v_{i}=( 0,...,0,\psi^{i}, 0,...,0)^{T},\ &\hbox{on}\ \partial{D}_{1},\\
v_{i}=0,&\hbox{on} \ \partial{D}.
\end{cases}
\end{align*}
Then we have
\begin{align*}
\nabla v=\sum^{d}_{i=1}\nabla v_{i}.
\end{align*}

Extend $\psi\in C^{1,\gamma}(\partial D_{1};\mathbb{R}^{d})$ to $\psi\in C^{1,\gamma}(\overline{\Omega};\mathbb{R}^{d})$, which verifies that $\|\psi^{i}\|_{C^{1,\gamma}(\overline{\Omega\setminus\Omega_{R}})}\leq C\|\psi^{i}\|_{C^{1,\gamma}(\partial D_{1})}$, $i=1,2,...,d.$ Let $\rho\in C^{1,\gamma}(\overline{\Omega})$ be a smooth cutoff function satisfying that $0\leq\rho\leq1$, $|\nabla\rho|\leq C$ in $\overline{\Omega}$, and
\begin{align}\label{Le2.021}
\rho=1\;\,\mathrm{in}\;\,\Omega_{\frac{3}{2}R},\quad\rho=0\;\,\mathrm{in}\;\,\overline{\Omega}\setminus\Omega_{2R}.
\end{align}
For $i=1,2,...,d$, define
\begin{align*}
\tilde{v}_{i}(x)=\left(0,...,0,[\rho(x)\psi^{i}(x',\varepsilon+h_{1}(x'))+(1-\rho(x))\psi^{i}(x)]\bar{v}(x),0,...,0\right)^{T},\quad x\in\Omega.
\end{align*}
In particular,
\begin{align*}
\tilde{v}_{i}(x)=(0,...,0,\psi^{i}(x',\varepsilon+h_{1}(x'))\bar{v}(x),0,...,0)^{T},\quad\;\,\mathrm{in}\;\Omega_{R}.
\end{align*}
In light of \eqref{Le2.021}, we derive
\begin{align*}
\|\tilde{v}_{i}\|_{C^{1}(\Omega\setminus\Omega_{R})}\leq C\|\psi^{i}\|_{C^{1}(\partial D_{1})},\quad i=1,2,...,d.
\end{align*}
%Then we have
%\begin{align*}
%\tilde{v}=\sum^{d}_{i=1}\tilde{v}_{i},\quad\mathrm{in}\;\Omega_{2R},
%\end{align*}
%where $\tilde{v}$ is defined in \eqref{CAN01} with $\phi=0$ on $\partial D_{2}$.

Write
\begin{equation*}
w_i:=v_i-\tilde{v}_i,\quad i=1,2,...,d.
\end{equation*}
Then $w_{i}$ satisfies
\begin{align}\label{Zww01}
\begin{cases}
\mathcal{L}_{\lambda,\mu}w_{i}=-\mathcal{L}_{\lambda,\mu}\tilde{v}_{i},&
\hbox{in}\;\Omega,  \\
w_{i}=0, \quad&\hbox{on}\;\partial\Omega.
\end{cases}
\end{align}
Observe that $w_{i}$ also verifies that for any constant matrix $\mathcal{M}=(\mathfrak{a}_{ij})$,
\begin{align}\label{ZWZW001}
-\mathcal{L}_{\lambda, \mu}w_{i}=\nabla\cdot(\mathbb{C}^0e(\tilde{v}_{i})-\mathcal{M}),\quad&\hbox{in}\ \Omega.
\end{align}

We now divide into three parts to prove Theorem \ref{thm8698}. For simplicity, we utilize $\|\psi^{i}\|_{C^{1}}$ to denote $\|\psi^{i}\|_{C^{1}(\partial D_{1})}$ in the following.

{\bf Step 1.} Proof of
\begin{align}\label{zzwad01}
\|\nabla w_{i}\|_{L^{2}(\Omega)}\leq C\|\psi^{i}\|_{C^{1}},\quad i=1,2,...,d.
\end{align}
In view of \eqref{Zww01}, we know
\begin{align}\label{OAP001}
\int_{\Omega}(\mathbb{C}^{0}e(w_{i}),e(w_{i}))\,dx=-\int_{\Omega}(\mathbb{C}^{0}e(\tilde{v}_{i}),e(w_{i}))\,dx.
\end{align}
On one hand, it follows from \eqref{ellip} and the first Korn's inequality that
\begin{equation}\label{QAT001}
\int_{\Omega}(\mathbb{C}^{0}e(w_{i}),e(w_{i}))\,dx\geq\frac{1}{C}\int_{\Omega}|e(w_{i})|^{2}dx\geq\frac{1}{C}\int_{\Omega}|\nabla w_{i}|^{2}dx.
\end{equation}

On the other hand, we first decompose the right hand of \eqref{OAP001} into two parts as follows:
\begin{align*}
\mathrm{I}=-\int_{\Omega\setminus\Omega_{R}}(\mathbb{C}^{0}e(\tilde{v}_{i}),e(w_{i}))\,dx,\quad \mathrm{II}=-\int_{\Omega_{R}}(\mathbb{C}^{0}e(\tilde{v}_{i}),e(w_{i}))\,dx.
\end{align*}

For the first term $\mathrm{I}$, we deduce from the H\"{o}lder inequality that
\begin{align}\label{AKTN001}
|\mathrm{I}|\leq&C\int_{\Omega\setminus\Omega_{R}}|\nabla\tilde{v}_{i}||\nabla w_{i}|\leq C\|\psi^{i}\|_{C^{1}}\|\nabla w_{i}\|_{L^{2}(\Omega\setminus\Omega_{R})}.
\end{align}
Recalling the definitions of $\mathbb{C}^{0}$ and $\tilde{v}_{i}$, it follows from a direct computation that
\begin{align*}
(\mathbb{C}^{0}e(\tilde{v}_{i}),e(w_{i}))=&\lambda\partial_{i}\tilde{v}_{i}^{i}\partial_{i}w^{i}_{i}+\mu\sum^{d}_{j=1}(\partial_{i}w^{j}_{i}+\partial_{j}w^{i}_{i})\partial_{j}\tilde{v}_{i}^{i},\quad i=1,2,...,d,
\end{align*}
where $\tilde{v}_{i}^{i}=\psi^{i}(x',\varepsilon+h_{1}(x'))\bar{v}.$ Since the case of $i=d$ is the same, it suffices to consider the case of $i\in\{1,...,d-1\}$ in the following. We first decompose $\mathrm{II}$ into two parts as follows:
\begin{align*}
\mathrm{II}_{1}=&\int_{\Omega_{R}}\lambda\partial_{i}\tilde{v}_{i}^{i}\partial_{i}w^{i}_{i}+\mu\sum^{d-1}_{j=1}(\partial_{i}w^{j}_{i}+\partial_{j}w^{i}_{i})\partial_{j}\tilde{v}_{i}^{i},\\
\mathrm{II}_{2}=&\int_{\Omega_{R}}\mu(\partial_{i}w^{d}_{i}+\partial_{d}w^{i}_{i})\partial_{d}\tilde{v}_{i}^{i}.
\end{align*}
Using the H\"{o}lder inequality again, we derive
\begin{align}\label{AHMZ001}
|\mathrm{II}_{1}|\leq&C\|\nabla_{x'}\tilde{v}_{i}^{i}\|_{L^{2}(\Omega_{R})}\|\nabla w_{i}\|_{L^{2}(\Omega_{R})}\leq C\|\psi^{i}\|_{C^{1}}\|\nabla w_{i}\|_{L^{2}(\Omega_{R})}.
\end{align}
As for $\mathrm{II}_{2}$, utilizing the Sobolev trace embedding theorem and in light of $\partial_{dd}\bar{v}=0$ in $\Omega_{R}$, it follows from integration by parts that
\begin{align*}
|\mathrm{II}_{2}|\leq&\int\limits_{\scriptstyle |x'|={R},\atop\scriptstyle
h_{2}(x')<x_{d}<\varepsilon+h_1(x')\hfill}\mu\left| w^{d}_{i}\partial_{d}\tilde{v}_{i}^{i}\nu_{d}+w^{d}_{i}\partial_{d}\tilde{v}_{i}^{i}\nu_{i}-w^{d}_{i}\partial_{i}\tilde{v}_{i}^{i}\nu_{d}\right|+\int_{\Omega_{R}}\left|\mu\partial_{i}\tilde{v}_{i}^{i}\partial_{d}w^{d}_{i}\right|\notag\\
\leq&\int\limits_{\scriptstyle |x'|={R},\atop\scriptstyle
h_{2}(x')<x_{d}<\varepsilon+h_1(x')\hfill}C\|\psi^{i}\|_{C^{1}}|w_{i}|+C\|\partial_{i}\tilde{v}_{i}^{i}\|_{L^{2}(\Omega_{R})}\|\nabla w_{i}\|_{L^{2}(\Omega_{R})}\notag\\
\leq&C\|\psi^{i}\|_{C^{1}}\|\nabla w_{i}\|_{L^{2}(\Omega)}.
\end{align*}
This, together with \eqref{AHMZ001}, yields that
\begin{align}\label{ASONZ}
|\mathrm{II}|\leq C\|\varphi^{i}\|_{C^{1}}\|\nabla w_{i}\|_{L^{2}(\Omega)}.
\end{align}
Consequently, it follows from \eqref{OAP001}--\eqref{AKTN001} and \eqref{ASONZ} that
\begin{align*}
\int_{\Omega}|\nabla w_{i}|^{2}dx\leq&C\|\psi^{i}\|_{C^{1}}\left(\int_{\Omega}|\nabla w_{i}|^{2}dx\right)^{\frac{1}{2}}.
\end{align*}
That is, \eqref{zzwad01} holds.

{\bf Part 2.}
For $i=1,2,...,d$ and $|z'|\leq R$, claim that
\begin{align}\label{step2}
 \int_{\Omega_\delta(z')}|\nabla w_{i}|^2dx
 &\leq C\delta^{d-\frac{2}{1+\gamma}}\big(|\psi^{i}(z',\varepsilon+h_{1}(z'))|^2+\delta^{\frac{2}{1+\gamma}}\|\psi^{i}\|_{C^{1}}^2\big).
\end{align}

To begin with, for $0<t<s<R$, we choose a smooth cutoff function $\eta$ satisfying that $0\leq\eta(x')\leq1$, $\eta(x')=1$ if $|x'-z'|<t$, $\eta(x')=0$ if $|x'-z'|>s$, and $|\nabla\eta(x')|\leq\frac{2}{s-t}$. Multiplying equation \eqref{ZWZW001} by $w_{i}\eta^{2}$ and utilizing integration by parts, we derive
\begin{align}\label{LMQ001}
\int_{\Omega_{s}(z')}(\mathbb{C}^{0}e(w_{i}),e(w_{i}\eta^{2}))\,dx=-\int_{\Omega_{s}(z')}(\mathbb{C}^{0}e(\tilde{v}_{i})-\mathcal{M},e(w_{i}\eta^{2}))\,dx.
\end{align}
On one hand, making use of \eqref{ITERA}, \eqref{ellip} and the first Korn's inequality, we deduce
\begin{align}\label{KLW01}
\int_{\Omega_{s}(z')}(\mathbb{C}^{0}e(w_{i}),e(w_{i}\eta^{2}))\,dx\geq&\frac{1}{C}\int_{\Omega_{s}(z')}|\nabla(w_{i}\eta)|^{2}dx-C\int_{\Omega_{s}(z')}|w_{i}|^{2}|\nabla\eta|^{2}dx.
\end{align}
On the other hand, it follows from the Young's inequality that for any $\zeta>0$,
\begin{align}\label{KLW02}
\left|\int_{\Omega_{s}(z')}(\mathbb{C}^{0}e(\tilde{v}_{i})-\mathcal{M},e(w_{i}\eta^{2}))\,dx\right|\leq&\zeta\int_{\Omega_{s}(z')}\eta^{2}|\nabla w_{i}|^{2}dx+C\int_{\Omega_{s}(z')}|\nabla\eta|^{2}|w_{i}|^{2}dx\notag\\
&+\frac{C}{\zeta}\int_{\Omega_{s}(z')}|\mathbb{C}^{0}e(\tilde{v}_{i})-\mathcal{M}|^{2}dx.
\end{align}
From \eqref{LMQ001}--\eqref{KLW02}, we know
\begin{align*}
\int_{\Omega_{t}(z')}|\nabla w_{i}|^{2}dx\leq\frac{C}{(s-t)^{2}}\int_{\Omega_{s}(z')}|w_{i}|^{2}dx+C\int_{\Omega_{s}(z')}|\mathbb{C}^{0}e(\tilde{v}_{i})-\mathcal{M}|^{2}dx.
\end{align*}
Let
\begin{align*}
\mathcal{M}=\frac{1}{|\Omega_{s}(z')|}\int_{\Omega_{s}(z')}\mathbb{C}^{0}e(\tilde{v}_{i}(y))\,dy.
\end{align*}

For $|z'|\leq R$, $0<s\leq\vartheta(\tau,\kappa_{1})\delta^{\frac{1}{1+\gamma}}$, $\vartheta(\tau,\kappa_{1})=\frac{1}{8\kappa_{1}\max\{1,\tau^{-\frac{\gamma}{1+\gamma}}\}}$, making use of conditions ({\bf{S1}}) and ({\bf{S2}}), we obtain that for $(x',x_{d})\in\Omega_{s}(z')$,
\begin{align}\label{ASK001}
|\delta(x')-\delta(z')|\leq&|h_{1}(x')-h_{1}(z')|+|h_{2}(x')-h_{2}(z')|\notag\\
\leq&(|\nabla_{x'}h_{1}(x'_{\theta_{1}})|+|\nabla_{x'}h_{2}(x'_{\theta})|)|x'-z'|\notag\\
\leq&\kappa_{1}|x'-z'|(|x'_{\theta_{1}}|^{\gamma}+|x'_{\theta}|^{\gamma})\notag\\
\leq&2\kappa_{1}s(s^{\gamma}+|z'|^{\gamma})\notag\\
\leq&\frac{\delta(z')}{2},
\end{align}
which implies that
\begin{align}\label{QWN001}
\frac{1}{2}\delta(z')\leq\delta(x')\leq\frac{3}{2}\delta(z'),\quad\mathrm{in}\;\Omega_{s}(z').
\end{align}
In view of \eqref{QWN001}, a direct computation yields that
\begin{align}\label{QWN002}
[\nabla\tilde{v}_{i}]_{\gamma,\Omega_{s}(z')}\leq C\big(|\psi^{i}(z',\varepsilon+h_{1}(z'))|\delta^{-\frac{2+\gamma}{1+\gamma}}+\|\psi^{i}\|_{C^{1}}\delta^{-1}\big)s^{1-\gamma}.
\end{align}
Due to the fact that $w_{i}=0$ on $\partial\Omega$, it follows from \eqref{QWN001}--\eqref{QWN002} that
\begin{align}\label{ADE007}
\int_{\Omega_{s}(z')}|w_{i}|^{2}\leq C\delta^{2}\int_{\Omega_{s}(z')}|\nabla w_{i}|^{2},
\end{align}
and
\begin{align}\label{ADE006}
&\int_{\Omega_{s}(z')}|\mathbb{C}^{0}e(\tilde{v}_{i})-\mathcal{M}|^{2}dx\leq  Cs^{d+1}\delta^{-\frac{3+\gamma}{1+\gamma}}\big(|\psi^{i}(z',\varepsilon+h_{1}(z'))|^{2}+\delta^{\frac{2}{1+\gamma}}\|\psi^{i}\|_{C^{1}}^{2}\big).
\end{align}

Write
\begin{align*}
F(t):=\int_{\Omega_{t}(z')}|\nabla w_{i}|^{2}.
\end{align*}
Then combining (\ref{ADE007})--(\ref{ADE006}), we obtain
\begin{align}\label{ADE008}
F(t)\leq \left(\frac{c\delta}{s-t}\right)^2F(s)+Cs^{d+1}\delta^{-\frac{3+\gamma}{1+\gamma}}\big(|\psi^{i}(z',\varepsilon+h_{1}(z'))|^{2}+\delta^{\frac{2}{1+\gamma}}\|\psi^{i}\|_{C^{1}}^{2}\big),
\end{align}
where $c$ and $C$ are universal constants independent of $\varepsilon$.

Let $k=\left[\frac{\vartheta(\tau,\kappa_{1})}{4c\delta^{\frac{\gamma}{1+\gamma}}}\right]+1$ and $t_{i}=\delta+2ci\delta,\;i=0,1,2,...,k$. Then applying (\ref{ADE008}) with $s=t_{i+1}$ and $t=t_{i}$, we have
$$F(t_{i})\leq\frac{1}{4}F(t_{i+1})+C(i+1)^{n+1}\delta^{d-\frac{2}{1+\gamma}}\big(|\psi^{i}(z',\varepsilon+h_{1}(z'))|^{2}+\delta^{\frac{2}{1+\gamma}}\|\psi^{i}\|^{2}_{C^{1}}\big).$$
This, in combination with $k$ iterations and (\ref{zzwad01}), reads that for a sufficiently small $\varepsilon>0$,
\begin{align*}
F(t_{0})\leq C\delta^{d-\frac{2}{1+\gamma}}\big(|\psi^{i}(z',\varepsilon+h_{1}(z'))|^{2}+\delta^{\frac{2}{1+\gamma}}\|\psi^{i}\|^{2}_{C^{1}}\big).
\end{align*}

{\bf Part 3.}
Proof of
\begin{align*}
|\nabla w_{i}(x)|\leq C\delta^{-\frac{1}{1+\gamma}}\big(|\psi^{i}(x',\varepsilon+h_{1}(x'))|+\delta^{\frac{1}{1+\gamma}}\|\psi^{i}\|_{C^{1}}\big),\quad i=1,2,...,d,\;x\in\Omega_{R}.
\end{align*}

By carrying out a change of variables in $\Omega_{\delta}(z')$ as follows:
\begin{align*}
\begin{cases}
x'-z'=\delta y',\\
x_{d}=\delta y_{d},
\end{cases}
\end{align*}
we rescale $\Omega_{\delta}(z')$ into $Q_{1}$, where, for $0<r\leq 1$,
\begin{align*}
Q_{r}=\left\{y\in\mathbb{R}^{d}\,\Big|\,\frac{1}{\delta}h(\delta y'+z')<y_{d}<\frac{\varepsilon}{\delta}+\frac{1}{\delta}h_{1}(\delta y'+z'),\;|y'|<r\right\}.
\end{align*}
Denote the top and bottom boundaries of $Q_{r}$ by
\begin{align*}
\Gamma^{+}_{r}=&\left\{y\in\mathbb{R}^{d}\,\Big|\,y_{d}=\frac{\varepsilon}{\delta}+\frac{1}{\delta}h_{1}(\delta y'+z'),\;|y'|<r\right\},
\end{align*}
and
\begin{align*}
\Gamma^{-}_{r}=&\left\{y\in\mathbb{R}^{d}\,\Big|\,y_{d}=\frac{1}{\delta}h(\delta y'+z'),\;|y'|<r\right\},
\end{align*}
respectively. $Q_{1}$ is actually of nearly unit size. Similar to \eqref{ASK001}, we obtain that for $x\in\Omega_{\delta}(z')$,
\begin{align*}
|\delta(x')-\delta(z')|
\leq&2\kappa_{1}\delta(\delta^{\gamma}+|z'|^{\gamma})\leq 4\kappa_{1}\max\{1,\tau^{-\frac{\gamma}{1+\gamma}}\}\delta^{\frac{1+2\gamma}{1+\gamma}}.
\end{align*}
Then we have
\begin{align*}
\left|\frac{\delta(x')}{\delta(z')}-1\right|\leq 8\max\{1,\tau^{\frac{\gamma}{1+\gamma}}\}\kappa_{1}R^{\gamma},
\end{align*}
which, together with the fact that $R>0$ is a small constant independent of $\varepsilon$, reads that $Q_{1}$ is of nearly unit size. Denote
\begin{align*}
W(y',y_{d}):=w_{i}(\delta y'+z',\delta y_{d}),\quad \tilde{V}(y',y_{d}):=\tilde{v}_{i}(\delta y'+z',\delta y_{d}),\quad y\in Q_{1}.
\end{align*}
In view of \eqref{Zww01}, we obtain that $W$ solves
\begin{align}\label{CL005}
\begin{cases}
-\partial_{j}(C_{ijkl}^{0}\partial_{l}W^{k})=\partial_{j}(C_{ijkl}^{0}\partial_{l}\tilde{V}^{k}),&\mathrm{in}\; Q_{1},\\
W=0,&\mathrm{on}\;\Gamma^{\pm}_{1}.
\end{cases}
\end{align}
Then applying Theorems \ref{CL001} and \ref{CL002} for equation \eqref{CL005} with $f_{ij}=C_{ijkl}^{0}\partial_{l}\tilde{V}^{k}$, it follows from the Poincar\'{e} inequality that
\begin{align*}
\|\nabla W\|_{L^{\infty}(Q_{1/4})}\leq&C\big(\|W\|_{L^{\infty}(Q_{1/2})}+[\nabla \tilde{V}]_{\gamma,Q_{1/2}}\big)\notag\\
\leq&C\left(\|\nabla W\|_{L^{2}(Q_{1})}+[\nabla\tilde{V}]_{\gamma,Q_{1}}\right).
\end{align*}
In the above we utilized the fact that $[C_{ijkl}^{0}\partial_{l}\tilde{V}^{k}]_{\gamma,Q_{1}}\leq[\nabla\tilde{V}]_{\gamma,Q_{1}}$.

Then back to $w$, we have
\begin{align*}
\|\nabla w_{i}\|_{L^{\infty}(\Omega_{\delta/4}(z'))}\leq\frac{C}{\delta}\big(\delta^{1-\frac{d}{2}}\|\nabla w_{i}\|_{L^{2}(\Omega_{\delta}(z'))}+\delta^{1+\gamma}[\nabla\tilde{v}_{i}]_{\gamma,\Omega_{\delta}(z')}\big),
\end{align*}
which, in combination with \eqref{step2} and \eqref{QWN002}, yields that for $z\in\Omega_{R}$,
\begin{align*}
|\nabla w(z)|\leq \|\nabla w\|_{L^{\infty}(\Omega_{\delta/4}(z'))}\leq C\delta^{-\frac{1}{1+\gamma}}\big(|\psi^{i}(z',\varepsilon+h_{1}(z'))|+\delta^{\frac{1}{1+\gamma}}\|\psi^{i}\|_{C^{1}}\big).
\end{align*}
Consequently, Theorem \ref{thm8698} holds.

\end{proof}

\section{Proofs of Theorems \ref{ZHthm002} and \ref{ZHthm003}}\label{SEC004}
For $\alpha=1,2,...,\frac{d(d+1)}{2}$, denote
\begin{align}\label{LATU01}
\bar{u}_{2}^{\alpha}=\psi_{\alpha}(1-\bar{v}).
\end{align}
Then applying Theorem \ref{thm8698} with $\psi=\psi_{\alpha},\,\phi=0$ or $\psi=0,\,\phi=\psi_{\alpha}$, $\alpha=1,2,...,\frac{d(d+1)}{2}$, we have
\begin{corollary}\label{thm86}
Assume as above. Let $v_{i}^{\alpha}\in H^{1}(\Omega;\mathbb{R}^{d})$, $i=1,2$, $\alpha=1,2,...,\frac{d(d+1)}{2}$ be a weak solution of \eqref{qaz001}. Then, for a sufficiently small $\varepsilon>0$, $x\in\Omega_{R}$,
\begin{align}\label{Le2.025}
\nabla v_{i}^{\alpha}=&\nabla\bar{u}_{1}^{\alpha}+O(1)
\begin{cases}
\delta^{-\frac{1}{1+\gamma}},&\alpha=1,2,...,d,\\
1,&\alpha=d+1,...,\frac{d(d+1)}{2},
\end{cases}
\end{align}
where $\delta$ is defined in \eqref{deta}, the leading terms $\bar{u}_{i}^{\alpha}$, $i=1,2,\,\alpha=1,2,...,\frac{d(d+1)}{2}$ are defined by \eqref{zzwz002} and \eqref{LATU01}, respectively.
\end{corollary}

A direct application of Theorem 1.1 in \cite{LLBY2014} yields that
\begin{corollary}\label{coro00z}
Assume as above. Let $v_{i}^{\ast\alpha}$ and $v_{i}^{\alpha}$, $i=1,2$, $\alpha=1,2,...,\frac{d(d+1)}{2}$ be the solutions of \eqref{qaz001}, respectively. Then, we have
\begin{align*}
|\nabla v_{0}|+\left|\sum^{2}_{i=1}\nabla v_{i}^{\alpha}\right|\leq C\delta^{-\frac{d}{2}}e^{-\frac{1}{2C\delta^{\gamma/(1+\gamma)}}},\;\;\mathrm{in}\;\Omega_{R},
\end{align*}
and
\begin{align*}
\left|\sum^{2}_{i=1}\nabla v_{i}^{\ast\alpha}\right|\leq C|x'|^{-\frac{(1+\gamma)d}{2}}e^{-\frac{1}{2C|x'|^{\gamma}}},\;\;\mathrm{in}\;\Omega_{R}^{\ast},
\end{align*}
where the constant $C$ depends on $\gamma,d,\lambda,\mu,\tau,\kappa_{1},\kappa_{2}$, but not on $\varepsilon$.
\end{corollary}
The proof of this corollary is a slight modification of Theorem 1.1 in \cite{LLBY2014} and thus omitted here.

We now state a result in terms of the boundedness of $C_{i}^{\alpha}$, $i=1,2,$ $\alpha=1,2,...,\frac{d(d+1)}{2}$. Its proof is a
slight modification of the proof of Lemma 4.1 in \cite{BLL2015}.
\begin{lemma}\label{PAK001}
Let $C_{i}^{\alpha}$, $i=1,2,\,\alpha=1,2,...,\frac{d(d+1)}{2}$ be defined in \eqref{Decom}. Then
\begin{align*}
|C_{i}^{\alpha}|\leq C,\quad i=1,2,\,\alpha=1,2,...,\frac{d(d+1)}{2},
\end{align*}
where $C$ is a positive constant independent of $\varepsilon$.
\end{lemma}

On the other hand, with regard to the asymptotic expansions of $C^{\alpha}_{1}-C_{2}^{\alpha}$, $\alpha=1,2,\cdots,\frac{d(d+1)}{2}$, we have

\begin{theorem}\label{OMG123}
Let $C_{i}^{\alpha}$, $i=1,2,\,\alpha=1,2,...,\frac{d(d+1)}{2}$ be defined in \eqref{Decom}. Then for a sufficiently small $\varepsilon>0$,
\begin{itemize}
\item[(i)] if $d=2$, for $\alpha=1,2$,
\begin{align*}
C_{1}^{\alpha}-C_{2}^{\alpha}=&
\frac{\det\mathbb{F}_{1}^{\ast\alpha}}{\det \mathbb{F}_{0}^{\ast}}\frac{\varepsilon^{\frac{\gamma}{1+\gamma}}(1+O(\varepsilon(\gamma,\sigma)))}{\mathcal{L}_{2}^{\alpha}\mathcal{M}_{\gamma,\tau}},
\end{align*}
and for $\alpha=3$,
\begin{align*}
C_{1}^{3}-C_{2}^{3}=&\frac{\det\mathbb{F}_{1}^{\ast3}}{\det \mathbb{F}_{0}^{\ast}}(1+O(\varepsilon^{\frac{\gamma}{2(1+2\gamma)}})),
\end{align*}
where the constant $\mathcal{M}_{\gamma,\tau}$ is defined in \eqref{zwzh001}, the Lam\'{e} constants $\mathcal{L}_{2}^{\alpha}$, $\alpha=1,2$ is defined in \eqref{AZ}, the blow-up factor matrices $\mathbb{F}_{0}^{\ast}$ and $\mathbb{F}^{\ast\alpha}_{1},$ $\alpha=1,2,3$ are defined by \eqref{ZWZML001}--\eqref{ZWZML003}, the rest term $\varepsilon(\gamma,\sigma)$ is defined in \eqref{GC002}.

\item[(ii)] if $d\geq3$, for $\alpha=1,2,...,\frac{d(d+1)}{2}$,
\begin{align*}
C_{1}^{\alpha}-C_{2}^{\alpha}=&\frac{\det\mathbb{F}_{2}^{\ast\alpha}}{\det \mathbb{F}^{\ast}}(1+O(\bar{\varepsilon}(\gamma,d))),
\end{align*}
where the constant $\mathcal{M}_{\gamma,\tau}$ is defined in \eqref{zwzh001}, the blow-up factor matrices $\mathbb{F}^{\ast}$ and $\mathbb{F}_{2}^{\ast\alpha},$ $\alpha=1,2,...,\frac{d(d+1)}{2}$, are defined by \eqref{GC009}, the rest term $\bar{\varepsilon}(\gamma,d)$ is defined by \eqref{NZKL001}.

\end{itemize}
\end{theorem}

Once the aforementioned results hold, we immediately give the proofs of Theorems \ref{ZHthm002} and \ref{ZHthm003}.
\begin{proof}[Proofs of Theorems \ref{ZHthm002} and \ref{ZHthm003}.]
To begin with, it follows from Corollary \ref{coro00z} and Lemma \ref{PAK001} that
\begin{align}\label{KBMA01}
\left|\sum_{\alpha=1}^{\frac{d(d+1)}{2}}C_{2}^\alpha\nabla({v}_{1}^\alpha+{v}_{2}^\alpha)+\nabla{v}_{0}\right|\leq C\delta^{-\frac{d}{2}}e^{-\frac{1}{2C\delta^{\gamma/(1+\gamma)}}},\;\;\mathrm{in}\;\Omega_{R}.
\end{align}
Then combining \eqref{Decom002}, \eqref{KBMA01}, Corollary \ref{thm86} and Theorem \ref{OMG123}, we deduce that

$(1)$ if $d=2$, then
\begin{align*}
\nabla u=&\sum^{2}_{\alpha=1}\frac{\det\mathbb{F}_{1}^{\ast\alpha}}{\det \mathbb{F}_{0}^{\ast}}\frac{\varepsilon^{\frac{\gamma}{1+\gamma}}(1+O(\varepsilon(\gamma,\sigma)))}{\mathcal{L}_{2}^{\alpha}\mathcal{M}_{\gamma,\tau}}(\nabla\bar{u}_{1}^{\alpha}+O(\delta^{-\frac{1}{1+\gamma}}))\notag\\
&+\frac{\det\mathbb{F}_{1}^{\ast3}}{\det \mathbb{F}_{0}^{\ast}}(1+O(\varepsilon^{\frac{\gamma}{2(1+2\gamma)}}))(\nabla\bar{u}_{1}^{3}+O(1))+O(1)\delta^{-\frac{d}{2}}e^{-\frac{1}{2C\delta^{\gamma/(1+\gamma)}}}\notag\\
=&\sum\limits_{\alpha=1}^{2}\frac{\det\mathbb{F}_{1}^{\ast\alpha}}{\det \mathbb{F}_{0}^{\ast}}\frac{\varepsilon^{\frac{\gamma}{1+\gamma}}(1+O(\varepsilon(\gamma,\sigma)))}{\mathcal{L}_{2}^{\alpha}\mathcal{M}_{\gamma,\tau}}\nabla\bar{u}_{1}^{\alpha}\notag\\
&+\frac{\det\mathbb{F}_{1}^{\ast3}}{\det \mathbb{F}_{0}^{\ast}}(1+O(\varepsilon^{\frac{\gamma}{2(1+2\gamma)}}))\nabla\bar{u}_{1}^{3}+O(1)\delta^{-\frac{1-\gamma}{1+\gamma}}\|\varphi\|_{C^{0}(\partial D)};
\end{align*}

$(ii)$ if $d\geq3$, then
\begin{align*}
\nabla u=&\sum^{d}_{\alpha=1}\frac{\det\mathbb{F}_{2}^{\ast\alpha}}{\det \mathbb{F}^{\ast}}(1+O(\bar{\varepsilon}(\gamma,d)))(\nabla\bar{u}_{1}^{\alpha}+O(\delta^{-\frac{1}{1+\gamma}}))\notag\\
&+\sum^{\frac{d(d+1)}{2}}_{\alpha=d+1}\frac{\det\mathbb{F}_{2}^{\ast\alpha}}{\det \mathbb{F}^{\ast}}(1+O(\bar{\varepsilon}(\gamma,d)))(\nabla\bar{u}_{1}^{\alpha}+O(1))+O(1)\delta^{-\frac{d}{2}}e^{-\frac{1}{2C\delta^{\gamma/(1+\gamma)}}}\notag\\
=&\sum\limits_{\alpha=1}^{\frac{d(d+1)}{2}}\frac{\det\mathbb{F}_{2}^{\ast\alpha}}{\det \mathbb{F}^{\ast}}(1+O(\bar{\varepsilon}(\gamma,d)))\nabla\bar{u}^{\alpha}_{1}+O(1)\delta^{-\frac{1}{1+\gamma}}\|\varphi\|_{C^{0}(\partial D)}.
\end{align*}

Therefore, we complete the proofs of Theorems \ref{ZHthm002} and \ref{ZHthm003}.

\end{proof}

\section{Proof of Theorem \ref{OMG123}}\label{SEC005}
For $i,j=1,2$ and $\alpha, \beta=1,2,...,\frac{d(d+1)}{2}$, write
\begin{align*}
a_{ij}^{\alpha\beta}:=-\int_{\partial{D}_{j}}\frac{\partial v_{i}^{\alpha}}{\partial \nu_0}\large\Big|_{+}\cdot\psi_{\beta},\quad b_j^{\beta}:=-\int_{\partial D}\frac{\partial v_{j}^{\beta}}{\partial \nu_0}\large\Big|_{+}\cdot\varphi.
\end{align*}
Then it follows from the fourth line of \eqref{La.002} that
\begin{align}\label{ATKNQ001}
\begin{cases}
\sum\limits_{\alpha=1}^{\frac{d(d+1)}{2}}(C_{1}^\alpha-C_{2}^{\alpha}) a_{11}^{\alpha\beta}+\sum\limits_{\alpha=1}^{\frac{d(d+1)}{2}}C_{2}^\alpha\sum\limits^{2}_{i=1}a_{i1}^{\alpha\beta}=b_1^\beta,\\
\sum\limits_{\alpha=1}^{\frac{d(d+1)}{2}}(C_{1}^\alpha-C_{2}^{\alpha}) a_{12}^{\alpha\beta}+\sum\limits_{\alpha=1}^{\frac{d(d+1)}{2}}C_{2}^\alpha\sum\limits^{2}_{i=1}a_{i2}^{\alpha\beta}=b_2^\beta.
%\sum\limits_{\alpha=1}^{\frac{d(d+1)}{2}}C_{1}^\alpha a_{11}^{\alpha\beta}+\sum\limits_{\alpha=1}^{\frac{d(d+1)}{2}}C_{2}^\alpha a_{21}^{\alpha\beta}=b_1^\beta,\\
%\sum\limits_{\alpha=1}^{\frac{d(d+1)}{2}}C_{1}^\alpha a_{12}^{\alpha\beta}+\sum\limits_{\alpha=1}^{\frac{d(d+1)}{2}}C_{2}^\alpha a_{22}^{\alpha\beta}=b_2^\beta.
\end{cases}
\end{align}
Adding the first line of \eqref{ATKNQ001} to the second line, we obtain
\begin{align}\label{JGALP001}
\begin{cases}
\sum\limits_{\alpha=1}^{\frac{d(d+1)}{2}}(C_{1}^\alpha-C_{2}^{\alpha}) a_{11}^{\alpha\beta}+\sum\limits_{\alpha=1}^{\frac{d(d+1)}{2}}C_{2}^\alpha \sum\limits^{2}_{i=1}a_{i1}^{\alpha\beta}=b_1^\beta,\\
\sum\limits_{\alpha=1}^{\frac{d(d+1)}{2}}(C_{1}^{\alpha}-C_{2}^{\alpha})\sum\limits^{2}_{j=1}a_{1j}^{\alpha\beta}+\sum\limits_{\alpha=1}^{\frac{d(d+1)}{2}}C_{2}^\alpha \sum\limits^{2}_{i,j=1}a_{ij}^{\alpha\beta}=\sum\limits^{2}_{i=1}b_{i}^{\beta}.
\end{cases}
\end{align}
It is worth emphasizing that as shown in \eqref{JGALP001}, we utilize all the systems of equations in linear decomposition to calculate the difference of $C_{1}^\alpha-C_{2}^{\alpha}$, $\alpha=1,2,...,\frac{d(d+1)}{2}$, which is quite different from the idea adopted in \cite{CL2019}. Moreover, our idea in this paper solves the difficulty faced in \cite{CL2019} and allows to capture the blow-up factor matrices for the generalized $C^{1,\gamma}$-inclusions and any boundary data in all dimensions and thus give a precise computation of $C_{1}^\alpha-C_{2}^{\alpha}$ in all cases.

For the sake of convenience, denote
\begin{align*}
&X^{1}=\big(C_{1}^1-C_{2}^{1},...,C_{1}^\frac{d(d+1)}{2}-C_{2}^{\frac{d(d+1)}{2}}\big)^{T},\quad X^{2}=\big(C_{2}^{1},...,C_{2}^{\frac{d(d+1)}{2}}\big)^{T},\\
&Y^{1}=(b_1^1,...,b_{1}^{\frac{d(d+1)}{2}})^{T},\quad Y^{2}=\bigg(\sum\limits^{2}_{i=1}b_{i}^{1},...,\sum\limits^{2}_{i=1}b_{i}^{\frac{d(d+1)}{2}}\bigg)^{T},
\end{align*}
and
\begin{align*}
&\mathbb{A}=(a_{11}^{\alpha\beta})_{\frac{d(d+1)}{2}\times\frac{d(d+1)}{2}},\quad \mathbb{B}=\bigg(\sum\limits^{2}_{i=1}a_{i1}^{\alpha\beta}\bigg)_{\frac{d(d+1)}{2}\times\frac{d(d+1)}{2}},\\
&\mathbb{C}=\bigg(\sum\limits^{2}_{j=1}a_{1j}^{\alpha\beta}\bigg)_{\frac{d(d+1)}{2}\times\frac{d(d+1)}{2}},\quad \mathbb{D}=\bigg(\sum\limits^{2}_{i,j=1}a_{ij}^{\alpha\beta}\bigg)_{\frac{d(d+1)}{2}\times\frac{d(d+1)}{2}}.
\end{align*}
Therefore, we rewrite \eqref{JGALP001} as
\begin{gather}\label{PLA001}
\begin{pmatrix} \mathbb{A}&\mathbb{B} \\  \mathbb{C}&\mathbb{D}
\end{pmatrix}
\begin{pmatrix}
X^{1}\\
X^{2}
\end{pmatrix}=
\begin{pmatrix}
Y^{1}\\
Y^{2}
\end{pmatrix}.
\end{gather}
The following sections aim to solve the systems of equations \eqref{PLA001}. We would like to point out that by using the symmetry of $a_{ij}^{\alpha\beta}=a_{ji}^{\beta\alpha}$, we obtain that $\mathbb{C}=\mathbb{B}^{T}$.

\begin{lemma}\label{KM323}
Assume as in Theorems \ref{ZHthm002} and \ref{ZHthm003}. Then for a sufficiently small $\varepsilon>0$,
\begin{align*}
b_{i}^{\beta}=b_{i}^{\ast\beta}+O(\varepsilon^{\frac{\gamma}{1+2\gamma}}),\quad i=1,2,\;\beta=1,2,...,\frac{d(d+1)}{2},
\end{align*}
which yields that
\begin{align*}
\sum\limits^{2}_{i=1}b_{i}^{\beta}=\sum\limits^{2}_{i=1}b_{i}^{\ast\beta}+O(\varepsilon^{\frac{\gamma}{1+2\gamma}}).
\end{align*}

\end{lemma}
\begin{proof}
Take the case of $i=1$ for instance. The case of $i=2$ is the same and thus omitted here. Recalling the definition of $b_{1}^{\beta}$, it follows from \eqref{Le2.01222} that for $\beta=1,2,...,\frac{d(d+1)}{2}$,
\begin{align*}
b_1^{\beta}-b_{1}^{\ast\beta}=-\int_{\partial D}\frac{\partial(v_{1}^{\beta}-v_{1}^{\ast\beta})}{\partial \nu_0}\large\Big|_{+}\cdot\varphi,
\end{align*}
where $v_{1}^{\ast\beta}$ and $v_{1}^{\beta}$ satisfy \eqref{qaz001111} and \eqref{qaz001}, respectively. For $0<t\leq2R$, write $\Omega_{t}^{\ast}:=\Omega^{\ast}\cap\{|x'|<t\}$. For $\beta=1,2,...,\frac{d(d+1)}{2}$, we define
\begin{align*}
\bar{u}^{\ast\beta}_{1}=&\psi_{\beta}\bar{v}^{\ast},
\end{align*}
where $\bar{v}^{\ast}$ verifies that $\bar{v}^{\ast}=1$ on $\partial D_{1}^{\ast}\setminus\{0\}$, $\bar{v}^{\ast}=0$ on $\partial D_{2}\cup\partial D$, and
\begin{align*}
\bar{v}^{\ast}(x',x_{d})=\frac{x_{d}-h_{2}(x')}{h_{1}(x')-h_{2}(x')},\;\,\mathrm{in}\;\Omega_{2R}^{\ast},\quad\|\bar{v}^{\ast}\|_{C^{2}(\Omega^{\ast}\setminus\Omega^{\ast}_{R})}\leq C.
\end{align*}
Using ({\bf{H1}})--({\bf{H2}}), we deduce that for $x\in\Omega_{R}^{\ast}$, $\beta=1,2,...,\frac{d(d+1)}{2}$,
\begin{align}\label{LKT6.003}
|\nabla_{x'}(\bar{u}_{1}^{\beta}-\bar{u}_{1}^{\ast\beta})|\leq\frac{C}{|x'|},\quad|\partial_{x_{d}}(\bar{u}_{1}^{\beta}-\bar{u}_{1}^{\ast\beta})|\leq\frac{C\varepsilon}{|x'|^{1+\gamma}(\varepsilon+|x'|^{1+\gamma})}.
\end{align}
A direct application of Corollary \ref{thm86} yields that for $\beta=1,2,...,\frac{d(d+1)}{2}$,
\begin{align}\label{LKT6.005}
|\nabla_{x'}v_{1}^{\ast\beta}|\leq\frac{C}{|x'|},\quad|\partial_{x_{d}}v_{1}^{\ast\beta}|\leq\frac{C}{|x'|^{1+\alpha}},\quad|\nabla(v_{1}^{\ast\beta}-\bar{u}_{1}^{\ast\beta})|\leq\frac{C}{|x'|},\quad x\in\Omega_{R}^{\ast}.
\end{align}
For $0<t<R$, define
\begin{align*}
\mathcal{C}_{t}:=\left\{x\in\mathbb{R}^{d}\,\Big|\;2\min_{|x'|\leq t}h_{2}(x')\leq x_{d}\leq\varepsilon+2\max_{|x'|\leq t}h_{1}(x'),\;|x'|<t\right\}.
\end{align*}

Observe that for $\beta=1,2,...,\frac{d(d+1)}{2}$, $v_{1}^{\beta}-v_{1}^{\ast\beta}$ verifies
\begin{align*}
\begin{cases}
\mathcal{L}_{\lambda,\mu}(v_{1}^{\beta}-v_{1}^{\ast\beta})=0,&\mathrm{in}\;\,D\setminus(\overline{D_{1}\cup D_{1}^{\ast}\cup D_{2}}),\\
v_{1}^{\beta}-v_{1}^{\ast\beta}=\psi_{\beta}-v_{1}^{\ast\beta},&\mathrm{on}\;\,\partial D_{1}\setminus D_{1}^{\ast},\\
v_{1}^{\beta}-v_{1}^{\ast\beta}=v_{1}^{\beta}-\psi_{\beta},&\mathrm{on}\;\,\partial D_{1}^{\ast}\setminus(D_{1}\cup\{0\}),\\
v_{1}^{\beta}-v_{1}^{\ast\beta}=0,&\mathrm{on}\;\,\partial D_{2}\cup\partial D.
\end{cases}
\end{align*}
First, in view of the standard boundary and interior estimates of elliptic systems, we obtain that for $x\in\partial D_{1}\setminus D_{1}^{\ast}$,
\begin{align}\label{LKT6.007}
|(v_{1}^{\beta}-v_{1}^{\ast\beta})(x',x_{d})|=|v_{1}^{\ast\beta}(x',x_{d}-\varepsilon)-v_{1}^{\ast\beta}(x',x_{d})|\leq C\varepsilon.
\end{align}
From \eqref{Le2.025}, we obtain that for $x\in\partial D_{1}^{\ast}\setminus(D_{1}\cup\mathcal{C}_{\varepsilon^{\theta}})$, $0<\theta<\frac{1}{1+\gamma}$,
\begin{align}\label{LKT6.008}
|(v_{1}^{\beta}-v_{1}^{\ast\beta})(x',x_{d})|=|v_{1}^{\beta}(x',x_{d})-v_{1}^{\beta}(x',x_{d}+\varepsilon)|\leq C\varepsilon^{1-(1+\gamma)\theta}.
\end{align}
Then in light of \eqref{Le2.025} and (\ref{LKT6.003})--(\ref{LKT6.005}), we derive that for $x\in\Omega_{R}^{\ast}\cap\{|x'|=\varepsilon^{\gamma}\}$,
\begin{align*}
|\partial_{x_{d}}(v_{1}^{\beta}-v_{1}^{\ast\beta})|\leq&|\partial_{x_{d}}(v_{1}^{\beta}-\bar{u}_{1}^{\beta})|+|\partial_{x_{d}}(\bar{u}_{1}^{\beta}-\bar{u}_{1}^{\ast\beta}|+|\partial_{x_{d}}(\bar{u}_{1}^{\beta}-\bar{u}_{1}^{\ast\beta})|\notag\\
\leq&C\left(\frac{1}{\varepsilon^{2(1+\gamma)\theta-1}}+\frac{1}{\varepsilon^{\theta}}\right).
\end{align*}
This, in combination with the fact that $v_{1}^{\beta}-v_{1}^{\ast\beta}=0$ on $\partial D_{2}$, yields that
\begin{align}
|(v_{1}^{\beta}-v_{1}^{\ast\beta})(x',x_{d})|=&|(v_{1}^{\beta}-v_{1}^{\ast\beta})(x',x_{d})-(v_{1}^{\beta}-v_{1}^{\ast\beta})(x',h_{2}(x'))|\notag\\
\leq&C\big(\varepsilon^{1-(1+\gamma)\theta}+\varepsilon^{\gamma\theta}\big).\label{LKT6.009}
\end{align}
Take $\theta=\frac{1}{1+2\gamma}$. Then we see from (\ref{LKT6.007})--(\ref{LKT6.009}) that
$$|v_{1}^{\beta}-v_{1}^{\ast\beta}|\leq C\varepsilon^{\frac{\gamma}{1+2\gamma}},\quad\;\,\mathrm{on}\;\,\partial\big(D\setminus\big(\overline{D_{1}\cup D_{1}^{\ast}\cup D_{2}\cup\mathcal{C}_{\varepsilon^{\frac{1}{1+2\gamma}}}}\big)\big).$$
Making use of the maximum principle for the Lam\'{e} system in \cite{MMN2007}, we obtain
\begin{align}\label{AST123}
|v_{1}^{\beta}-v_{1}^{\ast\beta}|\leq C\varepsilon^{\frac{\gamma}{1+2\gamma}},\quad\;\,\mathrm{in}\;\,D\setminus\big(\overline{D_{1}\cup D_{1}^{\ast}\cup D_{2}\cup\mathcal{C}_{\varepsilon^{\frac{1}{1+2\gamma}}}}\big),
\end{align}
which, together with the standard boundary estimates, reads that
\begin{align*}
|\nabla(v_{1}^{\beta}-v_{1}^{\ast\beta})|\leq C\varepsilon^{\frac{\gamma}{1+2\gamma}},\quad\mathrm{on}\;\,\partial D.
\end{align*}
Consequently,
\begin{align*}
|b_1^{\beta}-b_{1}^{\ast\beta}|\leq\left|\int_{\partial D}\frac{\partial(v_{1}^{\beta}-v_{1}^{\ast\beta})}{\partial \nu_0}\large\Big|_{+}\cdot\varphi\right|\leq C\|\varphi\|_{C^{0}(\partial D)}\varepsilon^{\frac{\gamma}{1+2\gamma}}.
\end{align*}

\end{proof}

Recalling the definition of $a_{ij}^{\alpha\beta}$, it follows from \eqref{Le2.01222} that for $i,j=1,2$ and $\alpha,\beta=1,2,...,\frac{d(d+1)}{2}$,
\begin{align*}
a_{ij}^{\alpha\beta}=\int_{\Omega}(\mathbb{C}^0e(v_{i}^{\alpha}), e(v_j^\beta))dx.
\end{align*}
For simplicity, we denote
\begin{align}\label{APDN001}
\tilde{\varepsilon}(\gamma,\sigma)=&
\begin{cases}
\varepsilon^{\frac{\sigma}{1+\gamma}},&\gamma>\sigma,\\
\varepsilon^{\frac{\gamma}{1+\gamma}}|\ln\varepsilon|,&\gamma=\sigma,\\
\varepsilon^{\frac{\gamma}{1+\gamma}},&0<\gamma<\sigma.
\end{cases}
\end{align}
\begin{lemma}\label{lemmabc}
Assume as above. Then, for a sufficiently small $\varepsilon>0$,

$(i)$ for $\alpha=1,2,...d$, if $d=2$,
\begin{align}\label{zzw001}
a_{11}^{\alpha\alpha}=\mathcal{L}_{2}^{\alpha}\mathcal{M}_{\gamma,\tau}\varepsilon^{-\frac{\gamma}{1+\gamma}}(1+O(\tilde{\varepsilon}(\gamma,\sigma))),
\end{align}
and if $d\geq3$,
\begin{align}\label{zzw002}
a_{11}^{\alpha\alpha}=a_{11}^{\ast\alpha\alpha}+O(1)\bar{\varepsilon}(\gamma,d),
\end{align}
where $\bar{\varepsilon}(\gamma,d)$ and $\tilde{\varepsilon}(\gamma,\sigma)$ are defined by \eqref{NZKL001} and \eqref{APDN001}, respectively.

$(ii)$ for $\alpha=d+1,...,\frac{d(d+1)}{2}$,
\begin{align}\label{zzw003}
a_{11}^{\alpha\alpha}=a_{11}^{\ast\alpha\alpha}+O(1)\varepsilon^{\frac{\gamma}{2(1+2\gamma)}};
\end{align}

$(iii)$ if $d=2$, for $\alpha,\beta=1,2,\alpha\neq \beta$, then
\begin{align}\label{LVZQ001}
a_{11}^{12}=a_{11}^{21}=O(1)|\ln\varepsilon|,
\end{align}
and if $d\geq3$, for $\alpha,\beta=1,2,...,d,\,\alpha\neq\beta$, then
\begin{align}\label{ADzc}
a_{11}^{\alpha\beta}=a_{11}^{\beta\alpha}=&a_{11}^{\ast\alpha\beta}+O(1)
\begin{cases}
\varepsilon^{\frac{\gamma^{2}}{2(1+2\gamma)(1+\gamma)^{2}}},&d=3,\\
\varepsilon^{\frac{\gamma^{2}}{2(1+2\gamma)(1+\gamma)}},&d\geq4,
\end{cases}
\end{align}
and if $d\geq2$, for $\alpha=1,2,...,d,\,\beta=d+1,...,\frac{d(d+1)}{2},$ then
\begin{align}\label{LVZQ0011gdw}
a^{\alpha\beta}_{11}=a_{11}^{\beta\alpha}=&a_{11}^{\ast\alpha\beta}+O(1)
\begin{cases}
\varepsilon^{\frac{\gamma^{2}}{2(1+2\gamma)(1+\gamma)^{2}}},&d=2,\\
\varepsilon^{\frac{\gamma^{2}}{2(1+2\gamma)(1+\gamma)}},&d\geq3,
\end{cases}
\end{align}
and if $d\geq3$, for $\alpha,\beta=d+1,...,\frac{d(d+1)}{2},\,\alpha\neq\beta$, then
\begin{align}\label{zzw007}
a_{11}^{\alpha\beta}=a_{11}^{\beta\alpha}=a_{11}^{\ast\alpha\beta}+O(1)\varepsilon^{\frac{\gamma}{2(1+2\gamma)}};
\end{align}

$(iv)$ for $\alpha,\beta=1,2,...,\frac{d(d+1)}{2}$,
\begin{align}\label{AZQ001}
\sum\limits^{2}_{i=1}a_{i1}^{\alpha\beta}=&\sum\limits^{2}_{i=1}a_{i1}^{\ast\alpha\beta}+O(\varepsilon^{\frac{\gamma}{2(1+\gamma)}}),\quad\sum\limits^{2}_{j=1}a_{1j}^{\alpha\beta}=\sum\limits^{2}_{j=1}a_{1j}^{*\alpha\beta}+O(\varepsilon^{\frac{\gamma}{2(1+\gamma)}}),
\end{align}
and
\begin{align}\label{AZQ00111}
\sum\limits^{2}_{i,j=1}a_{ij}^{\alpha\beta}=\sum\limits^{2}_{i,j=1}a_{ij}^{\ast\alpha\beta}+O(\varepsilon^{\frac{\gamma}{2(1+\gamma)}}).
\end{align}

\end{lemma}
\begin{remark}
It is worth mentioning that each element of the coefficient matrix of the equations in \eqref{PLA001} is calculated accurately, especially the elements in the off-diagonal line. This is a vital improvement by contrast with the previous work \cite{CL2019} and enables to capture the blow-up factor matrices in all dimensions.
\end{remark}

\begin{proof}
{\bf Step 1. Proofs of (\ref{zzw001})--(\ref{zzw002})}. Let $\bar{\theta}=\frac{\gamma^{2}}{2(1+2\gamma)(1+\gamma)^{2}}$. For $\varepsilon^{\bar{\theta}}\leq|z'|\leq R$, we make use of the change of variable
\begin{align*}
\begin{cases}
x'-z'=|z'|^{1+\gamma}y',\\
x_{d}=|z'|^{1+\gamma}y_{d},
\end{cases}
\end{align*}
to rescale $\Omega_{|z'|+|z'|^{1+\gamma}}\setminus\Omega_{|z'|}$ and $\Omega_{|z'|+|z'|^{1+\gamma}}^{\ast}\setminus\Omega_{|z'|}^{\ast}$ into two nearly unit-size squares (or cylinders) $Q_{1}$ and $Q_{1}^{\ast}$, respectively. For $\alpha=1,2,...,d$, denote
\begin{align*}
V_{1}^{\alpha}(y)=v_{1}^{\alpha}(z'+|z'|^{1+\gamma}y',|z'|^{1+\gamma}y_{d}),\quad\mathrm{in}\;Q_{1},
\end{align*}
and
\begin{align*}
V_{1}^{\ast\alpha}(y)=v_{1}^{\ast\alpha}(z'+|z'|^{1+\gamma}y',|z'|^{1+\gamma}y_{d}),\quad\mathrm{in}\;Q_{1}^{\ast}.
\end{align*}
In light of the fact that $0<V_{1}^{\alpha},V_{1}^{\ast\alpha}<1$, it follows from the standard elliptic estimates that
\begin{align*}
\|V_{1}^{\alpha}\|_{C^{1,\gamma}(Q_{1})}\leq C,\quad\|V_{1}^{\ast\alpha}\|_{C^{1,\gamma}(Q_{1}^{\ast})}\leq C.
\end{align*}
A direct application of an interpolation with \eqref{AST123} gives that
\begin{align*}
|\nabla(V_{1}^{\alpha}-V_{1}^{\ast\alpha})|\leq C\varepsilon^{\frac{\gamma}{1+2\gamma}(1-\frac{1}{1+\gamma})}\leq C\varepsilon^{\frac{\gamma^{2}}{(1+2\gamma)(1+\gamma)}}.
\end{align*}
Then back to $v_{1}^{\alpha}-v_{1}^{\ast\alpha}$ and in view of $\varepsilon^{\bar{\theta}}\leq|z'|\leq R$, we obtain
\begin{align*}
|\nabla(v_{1}^{\alpha}-v_{1}^{\ast\alpha})(x)|\leq C\varepsilon^{\frac{\gamma^{2}}{(1+2\gamma)(1+\gamma)}}|z'|^{-1-\gamma}\leq C\varepsilon^{(1+\gamma)\bar{\theta}},\quad x\in\Omega^{\ast}_{|z'|+|z'|^{1+\gamma}}\setminus\Omega_{|z'|}^{\ast},
\end{align*}
which yields that
\begin{align}\label{con035}
|\nabla(v_{1}^{\alpha}-v_{1}^{\ast\alpha})|\leq C\varepsilon^{(1+\gamma)\bar{\theta}},\quad\mathrm{in}\;D\setminus\big(\overline{D_{1}\cup D_{1}^{\ast}\cup\mathcal{C}_{\varepsilon^{\bar{\theta}}}}\big).
\end{align}

For $\alpha=1,2,...,d$, we decompose $a_{11}^{\alpha\alpha}$ into three parts as follows:
\begin{align*}
a_{11}^{\alpha\alpha}=&\int_{\Omega\setminus\Omega_{R}}(\mathbb{C}^{0}e(v_{1}^{\alpha}),e(v_{1}^{\alpha}))+\int_{\Omega_{\varepsilon^{\bar{\theta}}}}(\mathbb{C}^{0}e(v_{1}^{\alpha}),e(v_{1}^{\alpha}))+\int_{\Omega_{R}\setminus\Omega_{\varepsilon^{\bar{\theta}}}}(\mathbb{C}^{0}e(v_{1}^{\alpha}),e(v_{1}^{\alpha}))\nonumber\\
=&:\mathrm{I}+\mathrm{II}+\mathrm{III}.
\end{align*}
For the first term $\mathrm{I}$, due to the fact that $|\nabla v_{1}^{\alpha}|$ is bounded in $D_{1}^{\ast}\setminus(D_{1}\cup\Omega_{R})$ and $D_{1}\setminus D_{1}^{\ast}$ and the volume of $D_{1}^{\ast}\setminus(D_{1}\cup\Omega_{R})$ and $D_{1}\setminus D_{1}^{\ast}$ is of order $O(\varepsilon)$, we deduce from (\ref{con035}) that
\begin{align}\label{YUA123}
\mathrm{I}=&\int_{D\setminus(D_{1}\cup D_{1}^{\ast}\cup D_{2}\cup\Omega_{R})}(\mathbb{C}^{0}e(v_{1}^{\alpha}),e(v_{1}^{\alpha}))+O(1)\varepsilon\notag\\
=&\int_{D\setminus(D_{1}\cup D_{1}^{\ast}\cup D_{2}\cup\Omega_{R})}(\mathbb{C}^{0}e(v^{\ast\alpha}_{1}),e(v^{\ast\alpha}_{1}))+2\int_{D\setminus(D_{1}\cup D_{1}^{\ast}\cup D_{2}\cup\Omega_{R})}(\mathbb{C}^{0}e(v_{1}^{\alpha}-u^{\ast\alpha}_{1}),e(v^{\ast\alpha}_{1}))\notag\\
&+\int_{D\setminus(D_{1}\cup D_{1}^{\ast}\cup D_{2}\cup\Omega_{R})}(\mathbb{C}^{0}e(v_{1}^{\alpha}-v^{\ast\alpha}_{1}),e(v_{1}^{\alpha}-v^{\ast\alpha}_{1}))\notag\\
=&\int_{\Omega^{\ast}\setminus\Omega^{\ast}_{R}}(\mathbb{C}^{0}e(v^{\ast\alpha}_{1}),e(v^{\ast\alpha}_{1}))+O(1)\varepsilon^{(1+\gamma)\bar{\theta}}.
\end{align}

As for the second term $\mathrm{II}$, recalling the definition of $\bar{u}_{1}^{\alpha}$, it follows from Corollary \ref{thm86} that
\begin{align}\label{YUA03365}
\mathrm{II}=&\int_{\Omega_{\varepsilon^{\bar{\theta}}}}(\mathbb{C}^{0}e(\bar{u}_{1}^{\alpha}),e(\bar{u}_{1}^{\alpha}))+2\int_{\Omega_{\varepsilon^{\bar{\theta}}}}(\mathbb{C}^{0}e(v_{1}^{\alpha}-\bar{u}_{1}^{\alpha}),e(\bar{u}_{1}^{\alpha}))\notag\\
&+\int_{\Omega_{\varepsilon^{\bar{\theta}}}}(\mathbb{C}^{0}e(v_{1}^{\alpha}-\bar{u}_{1}^{\alpha}),e(v_{1}^{\alpha}-\bar{u}_{1}^{\alpha}))\notag\\
=&\,\mathcal{L}_{d}^{\alpha}\int_{|x'|<\varepsilon^{\bar{\theta}}}\frac{dx'}{\varepsilon+h_{1}(x')-h_{2}(x')}+O(1)
\begin{cases}
|\ln\varepsilon|,&d=2,\\
\varepsilon^{(d-2)\bar{\theta}},&d\geq3,
\end{cases}
\end{align}
where $\mathcal{L}_{d}^{\alpha}$ is defined in (\ref{AZ}).

With regard to the last term $\mathrm{III}$, we further split it into three parts as follows:
\begin{align*}
\mathrm{III}_{1}=&\int_{(\Omega_{R}\setminus\Omega_{\varepsilon^{\bar{\theta}}})\setminus(\Omega^{\ast}_{R}\setminus\Omega^{\ast}_{\varepsilon^{\bar{\theta}}})}(\mathbb{C}^{0}e(v_{1}^{\alpha}),e(v_{1}^{\alpha})),\\
\mathrm{III}_{2}=&\int_{\Omega^{\ast}_{R}\setminus\Omega^{\ast}_{\varepsilon^{\bar{\theta}}}}(\mathbb{C}^{0}e(v_{1}^{\alpha}-v_{1}^{\ast\alpha}),e(v_{1}^{\alpha}-v_{1}^{\ast\alpha}))+2\int_{\Omega^{\ast}_{R}\setminus\Omega^{\ast}_{\varepsilon^{\bar{\theta}}}}(\mathbb{C}^{0}e(v_{1}^{\alpha}-v_{1}^{\ast\alpha}),e(v_{1}^{\ast\alpha})),\\
\mathrm{III}_{3}=&\int_{\Omega^{\ast}_{R}\setminus\Omega^{\ast}_{\varepsilon^{\bar{\theta}}}}(\mathbb{C}^{0}e(v_{1}^{\ast\alpha}),e(v_{1}^{\ast\alpha})).
\end{align*}
Due to the fact that the thickness of $(\Omega_{R}\setminus\Omega_{\varepsilon^{\bar{\theta}}})\setminus(\Omega^{\ast}_{R}\setminus\Omega^{\ast}_{\varepsilon^{\bar{\theta}}})$ is of order $O(\varepsilon)$, it follows from (\ref{Le2.025}) that
\begin{align}\label{YUA0333355}
\mathrm{III}_{1}\leq\,C\varepsilon\int_{\varepsilon^{\bar{\theta}}<|x'|<R}\frac{dx'}{|x'|^{2(1+\gamma)}}\leq C\varepsilon^{1+(d-3-2\gamma)\bar{\theta}}.
\end{align}
For $\alpha=1,2,...,d$, applying Corollary \ref{thm86} for $v_{1}^{\ast\alpha}$ defined in (\ref{qaz001111}), we derive
\begin{align}\label{LGA01}
|\nabla_{x'}v^{\ast\alpha}_{1}|\leq\frac{C}{|x'|},\quad|\partial_{x_{d}}v^{\ast\alpha}_{1}|\leq\frac{C}{|x'|^{1+\gamma}},\quad|\nabla(v_{1}^{\ast\alpha}-\bar{u}_{1}^{\ast\alpha})|\leq\frac{C}{|x'|}.
\end{align}
Combining (\ref{con035}) and \eqref{LGA01}, we obtain
\begin{align}\label{YUA036}
|\mathrm{III}_{2}|\leq C\varepsilon^{(1+\gamma)\bar{\theta}}.
\end{align}
For $\mathrm{III}_{3}$, it follows from (\ref{LGA01}) again that for $d=2$,
\begin{align}\label{QPQ01}
\mathrm{III}_{3}=&\int_{\Omega_{R}^{\ast}\setminus\Omega^{\ast}_{\varepsilon^{\bar{\theta}}}}(\mathbb{C}^{0}e(\bar{u}_{1}^{\ast\alpha}),e(\bar{u}_{1}^{\ast\alpha}))+2\int_{\Omega_{R}^{\ast}\setminus\Omega^{\ast}_{\varepsilon^{\bar{\theta}}}}(\mathbb{C}^{0}e(v_{1}^{\ast\alpha}-\bar{u}_{1}^{\ast\alpha}),e(\bar{u}_{1}^{\ast\alpha}))\notag\\
&+\int_{\Omega_{R}^{\ast}\setminus\Omega^{\ast}_{\varepsilon^{\bar{\theta}}}}(\mathbb{C}^{0}e(v_{1}^{\ast\alpha}-\bar{u}_{1}^{\ast\alpha}),e(v_{1}^{\ast\alpha}-\bar{u}_{1}^{\ast\alpha}))\notag\\
=&\,\mathcal{L}_{2}^{\alpha}\int_{\varepsilon^{\bar{\theta}}<|x_{1}|<R}\frac{dx_{1}}{h_{1}(x_{1})-h_{2}(x_{1})}+O(1)|\ln\varepsilon|;
\end{align}
for $d\geq3$,
\begin{align}\label{YUAQPQ01}
\mathrm{III}_{3}=&\int_{\Omega_{R}^{\ast}\setminus\Omega^{\ast}_{\varepsilon^{\bar{\theta}}}}(\mathbb{C}^{0}e(\bar{u}_{1}^{\ast\alpha}),e(\bar{u}_{1}^{\ast\alpha}))+2\int_{\Omega_{R}^{\ast}\setminus\Omega^{\ast}_{\varepsilon^{\bar{\theta}}}}(\mathbb{C}^{0}e(v_{1}^{\ast\alpha}-\bar{u}_{1}^{\ast\alpha}),e(\bar{u}_{1}^{\ast\alpha}))\notag\\
&+\int_{\Omega_{R}^{\ast}\setminus\Omega^{\ast}_{\varepsilon^{\bar{\theta}}}}(\mathbb{C}^{0}e(v_{1}^{\ast\alpha}-\bar{u}_{1}^{\ast\alpha}),e(v_{1}^{\ast\alpha}-\bar{u}_{1}^{\ast\alpha}))\notag\\
=&\mathcal{L}_{d}^{\alpha}\int_{\varepsilon^{\bar{\theta}}<|x'|<R}\frac{dx'}{h_{1}(x')-h_{2}(x')}-\int_{\Omega^{\ast}\setminus\Omega^{\ast}_{R}}(\mathbb{C}^{0}e(v_{1}^{\ast\alpha}),e(v_{1}^{\ast\alpha}))\notag\\
&+M^{\ast}_{\alpha}+O(1)\varepsilon^{(d-2)\bar{\theta}},
\end{align}
where
\begin{align*}
M_{\alpha}^{\ast}=&\int_{\Omega^{\ast}\setminus\Omega^{\ast}_{R}}(\mathbb{C}^{0}e(v_{1}^{\ast\alpha}),e(v_{1}^{\ast\alpha}))+2\int_{\Omega_{R}^{\ast}}(\mathbb{C}^{0}e(v_{1}^{\ast\alpha}-\bar{u}_{1}^{\ast\alpha}),e(\bar{u}_{1}^{\ast\alpha}))\notag\\
&+\int_{\Omega_{R}^{\ast}}(\mathbb{C}^{0}e(v_{1}^{\ast\alpha}-\bar{u}_{1}^{\ast\alpha}),e(v_{1}^{\ast\alpha}-\bar{u}_{1}^{\ast\alpha}))\\
&+
\begin{cases}
\int_{\Omega_{R}^{\ast}}(\lambda+\mu)(\partial_{x_{\alpha}}\bar{v}^{\ast})^{2}+\mu\sum\limits^{d-1}_{i=1}(\partial_{x_{i}}\bar{v}^{\ast})^{2},&\alpha=1,...,d-1,\\
\int_{\Omega_{R}^{\ast}}\mu\sum\limits^{d-1}_{i=1}(\partial_{x_{i}}\bar{v}^{\ast})^{2},&\alpha=d.
\end{cases}
\end{align*}
Consequently, combining (\ref{YUA123})--(\ref{YUA0333355}) and (\ref{YUA036})--(\ref{YUAQPQ01}), we conclude that
\begin{align}\label{aaaa01}
a_{11}^{\alpha\alpha}=&\mathcal{L}_{d}^{\alpha}\left(\int_{\varepsilon^{\bar{\theta}}<|x'|<R}\frac{dx'}{h_{1}(x')-h_{2}(x')}+\int_{|x'|<\varepsilon^{\bar{\theta}}}\frac{dx'}{\varepsilon+h_{1}(x')-h_{2}(x')}\right)\nonumber\\
&+
\begin{cases}
O(1)|\ln\varepsilon|,&d=2,\\
M^{\ast}_{\alpha}+O(1)\varepsilon^{\bar{\theta}\min\{1+\gamma,d-2\}},&d\geq3.
\end{cases}
\end{align}
On one hand, if $d=2$, then
\begin{align}\label{ZZQ01}
&\int_{|x_{1}|<R}\frac{1}{\varepsilon+h_{1}-h_{2}}+\int_{\varepsilon^{\bar{\theta}}<|x_{1}|<R}\frac{\varepsilon}{(h_{1}-h_{2})(\varepsilon+h_{1}-h_{2})}\notag\\
&=\int_{|x_{1}|<R}\frac{1}{\varepsilon+\tau|x_{1}|^{1+\gamma}}+\int_{|x_{1}|<R}\left(\frac{1}{\varepsilon+h_{1}-h_{2}}-\frac{1}{\varepsilon+\tau|x_{1}|^{1+\gamma}}\right)+O(1)\varepsilon^{\frac{\gamma^{2}+4\gamma+2}{2(1+\gamma)^{2}}}\notag\\
&=2\int_{0}^{R}\frac{1}{\varepsilon+\tau s^{1+\gamma}}+O(1)\int_{0}^{R}\frac{s^{\beta}}{\varepsilon+\tau s^{1+\gamma}}\notag\\
&=\frac{2\Gamma_{\gamma}}{(1+\gamma)\tau^{\frac{1}{1+\gamma}}}\varepsilon^{-\frac{\gamma}{1+\gamma}}
\begin{cases}
1+O(1)\varepsilon^{\frac{\sigma}{1+\gamma}},&\gamma>\sigma,\\
1+O(1)\varepsilon^{\frac{\gamma}{1+\gamma}}|\ln\varepsilon|,&\gamma=\sigma,\\
1+O(1)\varepsilon^{\frac{\gamma}{1+\gamma}},&0<\gamma<\sigma;\\
\end{cases}
\end{align}
On the other hand, if $d\geq3$, then
\begin{align}\label{ZZQ02}
&\int_{|x'|<R}\frac{1}{h_{1}-h_{2}}-\int_{|x'|<\varepsilon^{\bar{\theta}}}\frac{\varepsilon}{(h_{1}-h_{2})(\varepsilon+h_{1}-h_{2})}\notag\\
&=\int_{\Omega^{\ast}}|\partial_{x_{d}}\bar{u}^{\ast\alpha}_{1}|^{2}+O(1)\varepsilon^{(d-2-\gamma)\bar{\theta}}.
\end{align}
Therefore, combining \eqref{aaaa01}--\eqref{ZZQ02}, we complete the proofs of (\ref{zzw001})--(\ref{zzw002}).

{\bf Step 2. Proof of (\ref{zzw003})}. Observe that for $\alpha=d+1,...,\frac{d(d+1)}{2}$, there exist two indices $1\leq i_{\alpha}<j_{\alpha}\leq d$ such that $\psi_{\alpha}=(0,...,0,x_{j_{\alpha}},0,...,0,-x_{i_{\alpha}},0,...,0)$. Pick $\tilde{\theta}=\frac{\gamma}{2(1+2\gamma)(1+\gamma)}$. For $\alpha=d+1,...,\frac{d(d+1)}{2}$, similarly as before, we decompose $a_{11}^{\alpha\alpha}$ as follows:
\begin{align*}
a_{11}^{\alpha\alpha}=&\int_{\Omega\setminus\Omega_{R}}(\mathbb{C}^{0}e(v_{1}^{\alpha}),e(v_{1}^{\alpha}))+\int_{\Omega_{\varepsilon^{\tilde{\theta}}}}(\mathbb{C}^{0}e(v_{1}^{\alpha}),e(v_{1}^{\alpha}))+\int_{\Omega_{R}\setminus\Omega_{\varepsilon^{\tilde{\theta}}}}(\mathbb{C}^{0}e(v_{1}^{\alpha}),e(v_{1}^{\alpha}))\nonumber\\
=&:\mathrm{I}+\mathrm{II}+\mathrm{III}.
\end{align*}

First, utilizing \eqref{LKT6.003}--\eqref{AST123} with a slight modification, it follows that for $\alpha=d+1,...,\frac{d(d+1)}{2}$,
\begin{align}\label{QNWE001}
|v_{1}^{\alpha}-v_{1}^{\ast\alpha}|\leq C\varepsilon^{\frac{1+\gamma}{1+2\gamma}},\quad\;\,\mathrm{in}\;\,D\setminus\big(\overline{D_{1}\cup D_{1}^{\ast}\cup D_{2}\cup\mathcal{C}_{\varepsilon^{\frac{1}{1+2\gamma}}}}\big).
\end{align}
Similarly as above, \eqref{QNWE001}, in combination with the rescale argument, the interpolation inequality and the standard elliptic estimates, reads that for $\alpha=d+1,...,\frac{d(d+1)}{2}$,
\begin{align}\label{QNWE002}
|\nabla(v_{1}^{\alpha}-v_{1}^{\ast\alpha})|\leq C\varepsilon^{(1+\gamma)\tilde{\theta}},\quad\;\,\mathrm{in}\;\,D\setminus\big(\overline{D_{1}\cup D_{1}^{\ast}\cup D_{2}\cup\mathcal{C}_{\varepsilon^{\tilde{\theta}}}}\big).
\end{align}

For the first part $\mathrm{I}$, similar to \eqref{YUA123}, we have
\begin{align}\label{QNWE003}
\mathrm{I}=&\int_{D\setminus(D_{1}\cup D_{1}^{\ast}\cup D_{2}\cup\Omega_{R})}(\mathbb{C}^{0}e(v_{1}^{\alpha}),e(v_{1}^{\alpha}))+O(1)\varepsilon\notag\\
=&\int_{\Omega^{\ast}\setminus\Omega^{\ast}_{R}}(\mathbb{C}^{0}e(v^{\ast\alpha}_{1}),e(v^{\ast\alpha}_{1}))+O(1)\varepsilon^{(1+\gamma)\tilde{\theta}}.
\end{align}

With regard to the second part $\mathrm{II}$, we further decompose it as follows:
\begin{align*}
\mathrm{II}
=&\int_{\Omega_{\varepsilon^{\bar{\theta}}}}[(\mathbb{C}^{0}e(\bar{u}^{\alpha}_{1}),e(\bar{u}^{\alpha}_{1}))+2(\mathbb{C}^{0}e(\bar{u}^{\alpha}_{1}),e(v^{\alpha}_{1}-\bar{u}_{1}^{\alpha}))+(\mathbb{C}^{0}e(v_{1}^{\alpha}-\bar{u}^{\alpha}_{1}),e(v_{1}^{\alpha}-\bar{u}^{\alpha}_{1}))].
\end{align*}
It follows from a direct computation that for $\alpha=d+1,...,\frac{d(d+1)}{2}$,
\begin{align*}
(\mathbb{C}^{0}e(\bar{u}^{\alpha}_{1}),e(\bar{u}^{\alpha}_{1}))=&\mu(x_{i_{\alpha}}^{2}+x_{j_{\alpha}}^{2})\sum^{d}_{k=1}(\partial_{x_{k}}\bar{v})^{2}+(\lambda+\mu)(x_{j_{\alpha}}\partial_{x_{i_{\alpha}}}\bar{v}-x_{i_{\alpha}}\partial_{x_{j_{\alpha}}}\bar{v})^{2}.
\end{align*}
Then in view of Corollary \ref{thm86}, we derive
\begin{align}\label{QNWE005}
\mathrm{II}=&O(1)\varepsilon^{(d-\gamma)\tilde{\theta}}.
\end{align}

As for the last part $\mathrm{III}$, it can be further split as follows:
\begin{align*}
\mathrm{III}_{1}=&\int_{(\Omega_{R}\setminus\Omega_{\varepsilon^{\tilde{\theta}}})\setminus(\Omega^{\ast}_{R}\setminus\Omega^{\ast}_{\varepsilon^{\tilde{\theta}}})}(\mathbb{C}^{0}e(v_{1}^{\alpha}),e(v_{1}^{\alpha}))+\int_{\Omega^{\ast}_{R}\setminus\Omega^{\ast}_{\varepsilon^{\tilde{\theta}}}}(\mathbb{C}^{0}e(v_{1}^{\alpha}-v_{1}^{\ast\alpha}),e(v_{1}^{\alpha}-v_{1}^{\ast\alpha}))\notag\\
&+2\int_{\Omega^{\ast}_{R}\setminus\Omega^{\ast}_{\varepsilon^{\tilde{\theta}}}}(\mathbb{C}^{0}e(v_{1}^{\alpha}-v_{1}^{\ast\alpha}),e(v_{1}^{\ast\alpha})),\\
\mathrm{III}_{2}=&\int_{\Omega^{\ast}_{R}\setminus\Omega^{\ast}_{\varepsilon^{\tilde{\theta}}}}(\mathbb{C}^{0}e(v_{1}^{\ast\alpha}),e(v_{1}^{\ast\alpha})).
\end{align*}
In light of the fact that the thickness of $(\Omega_{R}\setminus\Omega_{\varepsilon^{\tilde{\theta}}})\setminus(\Omega^{\ast}_{R}\setminus\Omega^{\ast}_{\varepsilon^{\tilde{\theta}}})$ is $\varepsilon$, it follows from (\ref{Le2.025}), \eqref{LGA01} and \eqref{QNWE002} that
\begin{align}\label{QNWE006}
\mathrm{III}_{1}=O(1)\varepsilon^{(1+\gamma)\tilde{\theta}}.
\end{align}
With regard to $\mathrm{III}_{2}$, similarly as in \eqref{QNWE005}, we have
\begin{align*}
\int_{\Omega^{\ast}_{\varepsilon^{\tilde{\theta}}}}(\mathbb{C}^{0}e(v_{1}^{\ast\alpha}),e(v_{1}^{\ast\alpha}))=O(1)\varepsilon^{(d-\gamma)\tilde{\theta}}.
\end{align*}
This yields that
\begin{align*}
\mathrm{III}_{2}=&\int_{\Omega^{\ast}_{R}}(\mathbb{C}^{0}e(v_{1}^{\ast\alpha}),e(v_{1}^{\ast\alpha}))-\int_{\Omega^{\ast}_{\varepsilon^{\tilde{\theta}}}}(\mathbb{C}^{0}e(v_{1}^{\ast\alpha}),e(v_{1}^{\ast\alpha}))\notag\\
=&\int_{\Omega^{\ast}_{R}}(\mathbb{C}^{0}e(v_{1}^{\ast\alpha}),e(v_{1}^{\ast\alpha}))+O(1)\varepsilon^{(d-\gamma)\tilde{\theta}}.
\end{align*}
This, in combination with \eqref{QNWE003}--\eqref{QNWE006}, reads that for $\alpha=d+1,...,\frac{d(d+1)}{2}$,
\begin{align*}
a_{11}^{\alpha\alpha}=&a_{11}^{\ast\alpha\alpha}+O(1)\varepsilon^{(1+\gamma)\tilde{\theta}}.
\end{align*}

{\bf Step 3. Proofs of (\ref{LVZQ001})--(\ref{zzw007})}. In light of the symmetry of $a_{11}^{\alpha\beta}$, we only need to consider the case of $\alpha<\beta$ in the following. Pick
\begin{align*}
\hat{\theta}=&
\begin{cases}
\frac{\gamma^{2}}{2(1+2\gamma)(1+\gamma)^{2}},&\alpha=1,2,...,d,\,\beta=1,2,...,\frac{d(d+1)}{2},\,\alpha<\beta,\\
\frac{\gamma}{2(1+2\gamma)(1+\gamma)},&\alpha,\beta=d+1,...,\frac{d(d+1)}{2},\,\alpha<\beta.
\end{cases}
\end{align*}
Similarly as above, for $\alpha,\beta=1,2,...,\frac{d(d+1)}{2}$, $\alpha<\beta$, we split $a_{11}^{\alpha\beta}$ into three terms as follows:
\begin{align*}
a_{11}^{\alpha\beta}=&\int_{\Omega\setminus\Omega_{R}}(\mathbb{C}^{0}e(v_{1}^{\alpha}),e(v_{1}^{\beta}))+\int_{\Omega_{\varepsilon^{\hat{\theta}}}}(\mathbb{C}^{0}e(v_{1}^{\alpha}),e(v_{1}^{\beta}))+\int_{\Omega_{R}\setminus\Omega_{\varepsilon^{\hat{\theta}}}}(\mathbb{C}^{0}e(v_{1}^{\alpha}),e(v_{1}^{\beta}))\nonumber\\
=&:\mathrm{I}+\mathrm{II}+\mathrm{III}.
\end{align*}

Applying the same argument used in \eqref{YUA123} to the first term $\mathrm{I}$, we obtain
\begin{align}\label{KKAA123}
\mathrm{I}=&\int_{D\setminus(D_{1}\cup D_{1}^{\ast}\cup D_{2}\cup\Omega_{R})}(\mathbb{C}^{0}e(v_{1}^{\alpha}),e(v_{1}^{\beta}))+O(1)\varepsilon\notag\\
=&\int_{D\setminus(D_{1}\cup D_{1}^{\ast}\cup D_{2}\cup\Omega_{R})}[(\mathbb{C}^{0}e(v^{\ast\alpha}_{1}),e(v^{\ast\alpha}_{1}))+(\mathbb{C}^{0}e(v_{1}^{\alpha}-v^{\ast\alpha}_{1}),e(v_{1}^{\beta}-v^{\ast\beta}_{1}))]\notag\\
&+\int_{D\setminus(D_{1}\cup D_{1}^{\ast}\cup D_{2}\cup\Omega_{R})}[(\mathbb{C}^{0}e(v^{\ast\alpha}_{1}),e(v^{\beta}_{1}-v^{\ast\beta}_{1}))+(\mathbb{C}^{0}e(v^{\alpha}_{1}-v^{\ast\alpha}_{1}),e(v^{\ast\beta}_{1}))]\notag\\
=&\int_{\Omega^{\ast}\setminus\Omega^{\ast}_{R}}(\mathbb{C}^{0}e(v^{\ast\alpha}_{1}),e(v^{\ast\beta}_{1}))+O(1)\varepsilon^{(1+\gamma)\hat{\theta}}.
\end{align}

With regard to the second term $\mathrm{II}$, we further decompose it as follows:
\begin{align}\label{NAT001}
\mathrm{II}=&\int_{\Omega_{\varepsilon^{\hat{\theta}}}}(\mathbb{C}^{0}e(v_{1}^{\alpha}),e(v_{1}^{\beta}))\notag\\
=&\int_{\Omega_{\varepsilon^{\hat{\theta}}}}(\mathbb{C}^{0}e(\bar{u}_{1}^{\alpha}),e(\bar{u}^{\beta}_{1}))+\int_{\Omega_{\varepsilon^{\hat{\theta}}}}(\mathbb{C}^{0}e(v_{1}^{\alpha}-\bar{u}_{1}^{\alpha}),e(v_{1}^{\beta}-\bar{u}^{\beta}_{1}))\notag\\
&+\int_{\Omega_{\varepsilon^{\hat{\theta}}}}(\mathbb{C}^{0}e(\bar{u}^{\alpha}_{1}),e(v^{\beta}_{1}-\bar{u}^{\beta}_{1}))+\int_{\Omega_{\varepsilon^{\hat{\theta}}}}(\mathbb{C}^{0}e(v_{1}^{\alpha}-\bar{u}^{\alpha}_{1}),e(\bar{u}^{\beta}_{1})).
\end{align}
By a direct calculation, we have

$(i)$ for $\alpha,\beta=1,2,...,d,$ $\alpha<\beta$, then
\begin{align}\label{ZH0000}
(\mathbb{C}^{0}e(\bar{u}_{1}^{\alpha}),e(\bar{u}^{\beta}_{1}))=(\lambda+\mu)\partial_{x_{\alpha}}\bar{v}\partial_{x_{\beta}}\bar{v};
\end{align}

$(ii)$ for $\alpha=1,2,...,d$, $\beta=d+1,...,\frac{d(d+1)}{2}$, there exist two indices $1\leq i_{\beta}<j_{\beta}\leq d$ such that
$\bar{u}_{1}^{\beta}=\psi_{\beta}\bar{v}=(0,...,0,x_{j_{\beta}}\bar{v},0,...,0,-x_{i_{\beta}}\bar{v},0,...,0)$. If $i_{\beta}\neq\alpha,\,j_{\beta}\neq\alpha$, then
\begin{align}\label{ZH000}
(\mathbb{C}^{0}e(\bar{u}_{1}^{\alpha}),e(\bar{u}^{\beta}_{1}))=\lambda\partial_{x_{\alpha}}\bar{v}(x_{j_{\beta}}\partial_{i_{\beta}}\bar{v}-x_{i_{\beta}}\partial_{x_{j_{\beta}}}\bar{v}),
\end{align}
and if $i_{\beta}=\alpha,\,j_{\beta}\neq\alpha$, then
\begin{align}\label{ZH001}
(\mathbb{C}^{0}e(\bar{u}_{1}^{\alpha}),e(\bar{u}^{\beta}_{1}))=\mu x_{j_{\beta}}\sum^{d}_{k=1}(\partial_{x_{k}}\bar{v})^{2}+(\lambda+\mu)\partial_{x_{\alpha}}\bar{v}(x_{j_{\beta}}\partial_{i_{\beta}}\bar{v}-x_{i_{\beta}}\partial_{x_{j_{\beta}}}\bar{v}),
\end{align}
and if $i_{\beta}\neq\alpha,\,j_{\beta}=\alpha$, then
\begin{align}\label{ZH002}
(\mathbb{C}^{0}e(\bar{u}_{1}^{\alpha}),e(\bar{u}^{\beta}_{1}))=&-\mu x_{i_{\beta}}\sum^{d}_{k=1}(\partial_{x_{k}}\bar{v})^{2}+(\lambda+\mu)\partial_{x_{\alpha}}\bar{v}(x_{j_{\beta}}\partial_{i_{\beta}}\bar{v}-x_{i_{\beta}}\partial_{x_{j_{\beta}}}\bar{v});
\end{align}

$(iii)$ for $\alpha,\beta=d+1,...,\frac{d(d+1)}{2}$, $\alpha<\beta$, there exist four indices $1\leq i_{\alpha}<j_{\alpha}\leq d$ and $1\leq i_{\beta}<j_{\beta}\leq d$ such that $\bar{u}_{1}^{\alpha}=\psi_{\alpha}\bar{v}=(0,...,0,x_{j_{\alpha}}\bar{v},0,...,0,-x_{i_{\alpha}}\bar{v},0,...,0)$
and
$\bar{u}_{1}^{\beta}=\psi_{\beta}\bar{v}=(0,...,0,x_{j_{\beta}}\bar{v},0,...,0,-x_{i_{\beta}}\bar{v},0,...,0)$. Since $\alpha<\beta$, we also have $j_{\beta}\leq j_{\alpha}$. If $i_{\alpha}\neq i_{\beta},\,j_{\alpha}\neq j_{\beta},\,i_{\alpha}\neq j_{\beta}$, then
\begin{align}\label{ZH003}
(\mathbb{C}^{0}e(\bar{u}_{1}^{\alpha}),e(\bar{u}^{\beta}_{1}))=\lambda(x_{j_{\alpha}}\partial_{x_{i_{\alpha}}}\bar{v}-x_{i_{\alpha}}\partial_{x_{j_{\alpha}}}\bar{v})(x_{j_{\beta}}\partial_{x_{i_{\beta}}}\bar{v}-x_{i_{\beta}}\partial_{x_{j_{\beta}}}\bar{v}),
\end{align}
and if $i_{\alpha}=i_{\beta},\,j_{\alpha}\neq j_{\beta}$, then
\begin{align}\label{ZH004}
(\mathbb{C}^{0}e(\bar{u}_{1}^{\alpha}),e(\bar{u}^{\beta}_{1}))=&\mu x_{j_{\alpha}}x_{j_{\beta}}\sum^{d}_{k=1}(\partial_{x_{k}}\bar{v})^{2}+\mu x_{j_{\alpha}}\partial_{x_{j_{\beta}}}\bar{v}(x_{j_{\beta}}\partial_{x_{j_{\beta}}}\bar{v}-x_{i_{\alpha}}\partial_{x_{i_{\alpha}}}\bar{v})\notag\\
&+(\lambda+\mu)(x_{j_{\alpha}}\partial_{x_{i_{\alpha}}}\bar{v}-x_{i_{\alpha}}\partial_{x_{j_{\alpha}}}\bar{v})(x_{j_{\beta}}\partial_{x_{i_{\beta}}}\bar{v}-x_{i_{\beta}}\partial_{x_{j_{\beta}}}\bar{v}),
\end{align}
and if $i_{\alpha}\neq i_{\beta},\,j_{\alpha}=j_{\beta}$, then
\begin{align}\label{ZH005}
(\mathbb{C}^{0}e(\bar{u}_{1}^{\alpha}),e(\bar{u}^{\beta}_{1}))=&\mu x_{i_{\alpha}}x_{i_{\beta}}\sum^{d}_{k=1}(\partial_{x_{k}}\bar{v})^{2}+\mu x_{i_{\alpha}}\partial_{x_{i_{\beta}}}\bar{v}(x_{i_{\beta}}\partial_{x_{i_{\beta}}}\bar{v}-x_{j_{\alpha}}\partial_{x_{j_{\alpha}}}\bar{v})\notag\\
&+(\lambda+\mu)(x_{j_{\alpha}}\partial_{x_{i_{\alpha}}}\bar{v}-x_{i_{\alpha}}\partial_{x_{j_{\alpha}}}\bar{v})(x_{j_{\alpha}}\partial_{x_{i_{\beta}}}\bar{v}-x_{i_{\beta}}\partial_{x_{j_{\alpha}}}\bar{v}),
\end{align}
and if $i_{\beta}<j_{\beta}=i_{\alpha}<j_{\alpha}$, then
\begin{align}\label{ZH006}
(\mathbb{C}^{0}e(\bar{u}_{1}^{\alpha}),e(\bar{u}^{\beta}_{1}))=&-\mu x_{i_{\beta}}x_{j_{\alpha}}\sum^{d}_{k=1}(\partial_{x_{k}}\bar{v})^{2}+\mu x_{j_{\alpha}}\partial_{x_{i_{\beta}}}\bar{v}(x_{i_{\alpha}}\partial_{x_{i_{\alpha}}}\bar{v}-x_{i_{\beta}}\partial_{x_{i_{\beta}}}\bar{v})\notag\\
&+(\lambda+\mu)(x_{j_{\alpha}}\partial_{x_{i_{\alpha}}}\bar{v}-x_{i_{\alpha}}\partial_{x_{j_{\alpha}}}\bar{v})(x_{i_{\alpha}}\partial_{x_{i_{\beta}}}\bar{v}-x_{i_{\beta}}\partial_{x_{i_{\alpha}}}\bar{v}).
\end{align}

Consequently, in light of the fact that
\begin{align*}
\left|\int^{\varepsilon+h_{1}(x')}_{h_{2}(x')}x_{d}\,dx_{d}\right|\leq|\varepsilon+h_{1}(x')|\delta(x')\leq C(\varepsilon+|x'|^{1+\alpha})^{2},\quad \mathrm{in}\; B'_{R},
\end{align*}
it follows from \eqref{NAT001}--\eqref{ZH006}, Corollary \ref{thm86}, the symmetry of integral region and the parity of integrand that
\begin{align}\label{MAH01}
\mathrm{II}=&O(1)
\begin{cases}
|\ln\varepsilon|,&d=2,\,\alpha=1,\beta=2,\\
\varepsilon^{(d-2)\hat{\theta}},&d\geq3,\,\alpha,\beta=1,2,...,d,\,\alpha<\beta,\\
\varepsilon^{(d-1)\hat{\theta}},&d\geq2,\,\alpha=1,2,...,d,\,\beta=d+1,...,\frac{d(d+1)}{2},\\
\varepsilon^{d\hat{\theta}},&d\geq3,\,\alpha,\beta=d+1,...,\frac{d(d+1)}{2},\,\alpha<\beta.
\end{cases}
\end{align}

As for $\mathrm{III}$, it can be further split as follows:
\begin{align*}
\mathrm{III}_{1}=&\int_{(\Omega_{R}\setminus\Omega_{\varepsilon^{\bar{\theta}}})\setminus(\Omega^{\ast}_{R}\setminus\Omega^{\ast}_{\varepsilon^{\bar{\theta}}})}(\mathbb{C}^{0}e(v_{1}^{\alpha}),e(v_{1}^{\beta}))+\int_{\Omega^{\ast}_{R}\setminus\Omega^{\ast}_{\varepsilon^{\bar{\theta}}}}(\mathbb{C}^{0}e(v_{1}^{\alpha}-v_{1}^{\ast\alpha}),e(v_{1}^{\beta}-v_{1}^{\ast\beta}))\notag\\
&+\int_{\Omega^{\ast}_{R}\setminus\Omega^{\ast}_{\varepsilon^{\bar{\theta}}}}(\mathbb{C}^{0}e(v_{1}^{\alpha}-v_{1}^{\ast\alpha}),e(v_{1}^{\ast\beta}))+\int_{\Omega^{\ast}_{R}\setminus\Omega^{\ast}_{\varepsilon^{\bar{\theta}}}}(\mathbb{C}^{0}e(v_{1}^{\ast\alpha}),e(v_{1}^{\beta}-v_{1}^{\ast\beta})),\\
\mathrm{III}_{2}=&\int_{\Omega^{\ast}_{R}\setminus\Omega^{\ast}_{\varepsilon^{\bar{\theta}}}}(\mathbb{C}^{0}e(v_{1}^{\ast\alpha}),e(v_{1}^{\ast\beta})).
\end{align*}
Since the thickness of $(\Omega_{R}\setminus\Omega_{\varepsilon^{\bar{\theta}}})\setminus(\Omega^{\ast}_{R}\setminus\Omega^{\ast}_{\varepsilon^{\bar{\theta}}})$ is $\varepsilon$, we deduce from (\ref{Le2.025}), (\ref{con035}), \eqref{LGA01} and \eqref{QNWE002} that
\begin{align}\label{KAM001}
\mathrm{III}_{1}=O(1)\varepsilon^{(1+\gamma)\hat{\theta}}.
\end{align}

With regard to $\mathrm{III}_{2}$, on one hand, for $d=2$, $\alpha=1,\beta=2$, we have
\begin{align*}
\mathrm{III}_{2}=&\int_{\Omega_{R}^{\ast}\setminus\Omega^{\ast}_{\varepsilon^{\hat{\theta}}}}(\mathbb{C}^{0}e(\bar{u}_{1}^{\ast1}),e(\bar{u}_{1}^{\ast2}))+\int_{\Omega_{R}^{\ast}\setminus\Omega^{\ast}_{\varepsilon^{\hat{\theta}}}}(\mathbb{C}^{0}e(v_{1}^{\ast1}-\bar{u}_{1}^{\ast1}),e(v_{1}^{\ast2}-\bar{u}_{1}^{\ast2}))\notag\\
&+\int_{\Omega_{R}^{\ast}\setminus\Omega^{\ast}_{\varepsilon^{\hat{\theta}}}}(\mathbb{C}^{0}e(v_{1}^{\ast1}-\bar{u}_{1}^{\ast1}),e(\bar{u}_{1}^{\ast2}))+\int_{\Omega_{R}^{\ast}\setminus\Omega^{\ast}_{\varepsilon^{\hat{\theta}}}}(\mathbb{C}^{0}e(\bar{u}_{1}^{\ast1}),e(v_{1}^{\ast2}-\bar{u}_{1}^{\ast2})),
\end{align*}
which, in combination with $(\mathbb{C}^{0}e(\bar{u}_{1}^{\ast1}),e(\bar{u}_{1}^{\ast2}))=(\lambda+\mu)\partial_{x_{1}}\bar{v}^{\ast}\partial_{x_{2}}\bar{v}^{\ast}$, reads that
\begin{align}\label{PLAM001}
\mathrm{III}_{2}=O(1)|\ln\varepsilon|.
\end{align}

On the other hand, for $d\geq3,\,\alpha,\beta=1,2,...,d,\,\alpha<\beta$, for $d\geq2,\,\alpha=1,2,...,d,\,\beta=d+1,...,\frac{d(d+1)}{2},\alpha<\beta$, or for $d\geq3,\,\alpha,\beta=d+1,...,\frac{d(d+1)}{2},\,\alpha<\beta$, similarly as in \eqref{NAT001}, applying \eqref{ZH0000}--\eqref{ZH006} with $\bar{v}$ replaced by $\bar{v}^{\ast}$ for $\int_{\Omega^{\ast}_{\varepsilon^{\hat{\theta}}}}(\mathbb{C}^{0}e(v_{1}^{\ast\alpha}),e(v_{1}^{\ast\alpha}))$, we deduce
\begin{align}\label{KAM002}
&\mathrm{III}_{2}-\int_{\Omega^{\ast}_{R}}(\mathbb{C}^{0}e(v_{1}^{\ast\alpha}),e(v_{1}^{\ast\beta}))\notag\\
=&-\int_{\Omega^{\ast}_{\varepsilon^{\hat{\theta}}}}(\mathbb{C}^{0}e(v_{1}^{\ast\alpha}),e(v_{1}^{\ast\beta}))\notag\\
=&O(1)
\begin{cases}
\varepsilon^{(d-2)\hat{\theta}},&d\geq3,\,\alpha,\beta=1,2,...,d,\,\alpha<\beta,\\
\varepsilon^{(d-1)\hat{\theta}},&d\geq2,\,\alpha=1,2,...,d,\,\beta=d+1,...,\frac{d(d+1)}{2},\\
\varepsilon^{d\hat{\theta}},&d\geq3,\,\alpha,\beta=d+1,...,\frac{d(d+1)}{2},\,\alpha<\beta.
\end{cases}
\end{align}
Therefore, combining \eqref{KAM001}--\eqref{KAM002}, we derive that
\begin{align*}
\mathrm{III}=\int_{\Omega_{R}\setminus\Omega_{\varepsilon^{\hat{\theta}}}}(\mathbb{C}^{0}e(v_{1}^{\alpha}),e(v_{1}^{\beta}))=O(1)|\ln\varepsilon|,\quad d=2,\,\alpha=1,\beta=2,
\end{align*}
and
\begin{align*}
&\mathrm{III}-\int_{\Omega^{\ast}_{R}}(\mathbb{C}^{0}e(v_{1}^{\ast\alpha}),e(v_{1}^{\ast\beta}))\notag\\
=&O(1)
\begin{cases}
\varepsilon^{\hat{\theta}\min\{1+\gamma,d-2\}},&d\geq3,\,\alpha,\beta=1,2,...,d,\,\alpha<\beta,\\
\varepsilon^{\hat{\theta}\min\{1+\gamma,d-1\}},&d\geq2,\,\alpha=1,2,...,d,\,\beta=d+1,...,\frac{d(d+1)}{2},\\
\varepsilon^{\hat{\theta}d},&d\geq3,\,\alpha,\beta=d+1,...,\frac{d(d+1)}{2},\,\alpha<\beta.
\end{cases}
\end{align*}
This, together with \eqref{KKAA123} and \eqref{MAH01}, gives that
\begin{align*}
a_{12}=O(1)|\ln\varepsilon|,\quad d=2,
\end{align*}
and
\begin{align*}
a_{11}^{\alpha\beta}=a_{11}^{\ast\alpha\beta}+&O(1)
\begin{cases}
\varepsilon^{\hat{\theta}\min\{1+\gamma,d-2\}},&d\geq3,\,\alpha,\beta=1,2,...,d,\,\alpha<\beta,\\
\varepsilon^{\hat{\theta}\min\{1+\gamma,d-1\}},&d\geq2,\,\alpha=1,2,...,d,\,\beta=d+1,...,\frac{d(d+1)}{2},\\
\varepsilon^{\hat{\theta}(1+\gamma)},&d\geq3,\,\alpha,\beta=d+1,...,\frac{d(d+1)}{2},\,\alpha<\beta.
\end{cases}
\end{align*}

{\bf Step 4. Proofs of (\ref{AZQ001})--(\ref{AZQ00111})}. Note that for every $\alpha=1,2,...,\frac{d(d+1)}{2}$, $v_{1}^{\alpha}+v_{2}^{\alpha}-v_{1}^{\ast\alpha}-v_{2}^{\ast\alpha}$ solves
\begin{align*}
\begin{cases}
\mathcal{L}_{\lambda,\mu}(v_{1}^{\alpha}+v_{2}^{\alpha}-v_{1}^{\ast\alpha}-v_{2}^{\ast\alpha})=0,&\mathrm{in}\;\,D\setminus(\overline{D_{1}\cup D_{1}^{\ast}\cup D_{2}}),\\
v_{1}^{\alpha}+v_{2}^{\alpha}-v_{1}^{\ast\alpha}-v_{2}^{\ast\alpha}=\psi_{\alpha}-v_{1}^{\ast\beta}-v_{2}^{\ast\alpha},&\mathrm{on}\;\,\partial D_{1}\setminus D_{1}^{\ast},\\
v_{1}^{\alpha}+v_{2}^{\alpha}-v_{1}^{\ast\alpha}-v_{2}^{\ast\alpha}=v_{1}^{\alpha}+v_{2}^{\alpha}-\psi_{\alpha},&\mathrm{on}\;\,\partial D_{1}^{\ast}\setminus(D_{1}\cup\{0\}),\\
v_{1}^{\alpha}+v_{2}^{\alpha}-v_{1}^{\ast\alpha}-v_{2}^{\ast\alpha}=0,&\mathrm{on}\;\,\partial D_{2}\cup\partial D.
\end{cases}
\end{align*}
Similarly as above, it follows from the standard boundary and interior estimates of elliptic systems that for $x\in\partial D_{1}\setminus D_{1}^{\ast}$,
\begin{align}\label{LAQ007}
&|v_{1}^{\alpha}+v_{2}^{\alpha}-v_{1}^{\ast\alpha}-v_{2}^{\ast\alpha}|\notag\\
=&|(v_{1}^{\ast\alpha}+v_{2}^{\ast\alpha})(x',x_{d}-\varepsilon)-(v_{1}^{\ast\alpha}+v_{2}^{\ast\alpha})(x',x_{d})|\leq C\varepsilon,
\end{align}
while, in light of Corollary \ref{coro00z}, we obtain that for $x\in\partial D_{1}^{\ast}\setminus(D_{1}\cup\mathcal{C}_{\varepsilon^{\frac{1}{1+\gamma}}})$,
\begin{align}\label{LAQ008}
&|(v_{1}^{\alpha}+v_{2}^{\alpha}-v_{1}^{\ast\alpha}-v_{2}^{\ast\alpha})(x',x_{d})|\notag\\
=&|(v_{1}^{\alpha}+v_{2}^{\alpha})(x',x_{d})-(v_{1}^{\alpha}+v_{2}^{\alpha})(x',x_{d}+\varepsilon)|\leq C\varepsilon.
\end{align}
Based on the fact that $v_{1}^{\alpha}+v_{2}^{\alpha}-v_{1}^{\ast\alpha}-v_{2}^{\ast\alpha}=0$ on $\partial D_{2}$, it follows from Corollary \ref{coro00z} again that for $x\in\Omega_{R}^{\ast}\cap\{|x'|=\varepsilon^{\frac{1}{1+\gamma}}\}$,
\begin{align}
&|(v_{1}^{\alpha}+v_{2}^{\alpha}-v_{1}^{\ast\alpha}-v_{2}^{\ast\alpha})(x',x_{d})|\notag\\
=&|(v_{1}^{\alpha}+v_{2}^{\alpha}-v_{1}^{\ast\alpha}-v_{2}^{\ast\alpha})(x',x_{d})-(v_{1}^{\alpha}+v_{2}^{\alpha}-v_{1}^{\ast\alpha}-v_{2}^{\ast\alpha})(x',h_{2}(x'))|\notag\\ \leq& C\delta^{\frac{\gamma}{1+\gamma}}\varepsilon\leq C\varepsilon^{\frac{1+2\gamma}{1+\gamma}},\label{LAQ009}
\end{align}
where in the last line of \eqref{LAQ009} we utilized the fact that the exponential function decays faster than the power function. Consequently, it follows from \eqref{LAQ007}--\eqref{LAQ009} that
\begin{align}\label{MIH01}
|v_{1}^{\alpha}+v_{2}^{\alpha}-v_{1}^{\ast\alpha}-v_{2}^{\ast\alpha}|\leq C\varepsilon,\quad\;\,\mathrm{on}\;\,\partial \big(D\setminus\big(\overline{D_{1}\cup D_{1}^{\ast}\cup D_{2}\cup\mathcal{C}_{\varepsilon^{\frac{1}{1+\gamma}}}}\big)\big).
\end{align}
Similar to \eqref{con035}, utilizing \eqref{MIH01}, the maximum principle, the rescale argument, the interpolation inequality and the standard elliptic estimates, we obtain
\begin{align}\label{ZQWZW001}
|\nabla(v_{1}^{\alpha}+v_{2}^{\alpha}-v_{1}^{\ast\alpha}-v_{2}^{\ast\alpha})|\leq C\varepsilon^{\frac{\gamma}{2(1+\gamma)}},\;\;\mathrm{in}\;\,D\setminus\big(\overline{D_{1}\cup D_{1}^{\ast}\cup D_{2}\cup\mathcal{C}_{\varepsilon^{\frac{1}{2(1+\gamma)^{2}}}}}\big).
\end{align}

Let $\tilde{\theta}=\frac{1}{2(1+\gamma)^{2}}$. We first decompose $\sum\limits^{2}_{i=1}a_{i1}^{\alpha\beta}$ into three parts as follows:
\begin{align*}
\sum\limits^{2}_{i=1}a_{i1}^{\alpha\beta}=&\int_{\Omega\setminus\Omega_{R}}(\mathbb{C}^{0}e(v_{1}^{\alpha}+v_{2}^{\alpha}),e(v_{1}^{\beta}))+\int_{\Omega_{\varepsilon^{\tilde{\theta}}}}(\mathbb{C}^{0}e(v_{1}^{\alpha}+v_{2}^{\alpha}),e(v_{1}^{\beta}))\\
&+\int_{\Omega_{R}\setminus\Omega_{\varepsilon^{\tilde{\theta}}}}(\mathbb{C}^{0}e(v_{1}^{\alpha}+v_{2}^{\alpha}),e(v_{1}^{\beta}))\\
=&:\mathrm{I}+\mathrm{II}+\mathrm{III}.
\end{align*}
With regard to the first part $\mathrm{I}$, by the same argument as in \eqref{KKAA123}, we deduce from \eqref{ZQWZW001} that
\begin{align}\label{GAZ0011}
\mathrm{I}=\int_{\Omega^{\ast}\setminus\Omega^{\ast}_{R}}(\mathbb{C}^{0}e(v_{1}^{\ast\alpha}+v_{2}^{\ast\alpha}),e(v_{1}^{\ast\beta}))+O(1)\varepsilon^{\frac{\gamma}{2(1+\gamma)}}.
\end{align}
As for the second part $\mathrm{II}$, utilizing Corollaries \ref{thm86} and \ref{coro00z}, we deduce
\begin{align}\label{GAZ001}
|\mathrm{II}|\leq \int_{|x'|\leq \varepsilon^{\tilde{\theta}}}C(\varepsilon+|x'|^{1+\gamma})^{\frac{\gamma}{1+\gamma}}\leq C\varepsilon^{\frac{\gamma(d+\gamma-1)}{2(1+\gamma)^{2}}}.
\end{align}
For the third part $\mathrm{III}$, it can be further split as follows:
\begin{align*}
\mathrm{III}_{1}=&\int_{\Omega^{\ast}_{R}\setminus\Omega^{\ast}_{\varepsilon^{\tilde{\theta}}}}\sum^{2}_{i=1}\Big[(\mathbb{C}^{0}e(v_{i}^{\alpha}-v_{i}^{\ast\alpha}),e(v_{1}^{\ast\beta}))+(\mathbb{C}^{0}e(v_{i}^{\ast\alpha}),e(v_{1}^{\beta}-v_{1}^{\ast\beta}))\notag\\
&\quad\quad\quad\quad\quad\quad+(\mathbb{C}^{0}e(v_{i}^{\alpha}-v_{i}^{\ast\alpha}),e(v_{1}^{\beta}-v_{1}^{\ast\beta}))\Big],\\
\mathrm{III}_{2}=&\int_{(\Omega_{R}\setminus\Omega_{\varepsilon^{\tilde{\theta}}})\setminus(\Omega^{\ast}_{R}\setminus\Omega^{\ast}_{\varepsilon^{\tilde{\theta}}})}(\mathbb{C}^{0}e(v_{1}^{\alpha}+v_{2}^{\alpha}),e(v_{1}^{\beta})),\\
\mathrm{III}_{3}=&\int_{\Omega^{\ast}_{R}\setminus\Omega^{\ast}_{\varepsilon^{\tilde{\theta}}}}(\mathbb{C}^{0}e(v_{1}^{\ast\alpha}+v_{2}^{\ast\alpha}),e(v_{1}^{\ast\beta})).
\end{align*}
First, it follows from \eqref{ZQWZW001} that
\begin{align}\label{GAP001}
|\mathrm{III}_{1}|\leq C\varepsilon^{\frac{\gamma}{2(1+\gamma)}}.
\end{align}
Second, making use of Corollaries \ref{thm86}-\ref{coro00z}, we deduce
\begin{align}\label{GAZ00195}
|\mathrm{III}_{2}|\leq \int_{\varepsilon^{\tilde{\theta}}\leq |x'|\leq R}\frac{C\varepsilon(\varepsilon+|x'|^{1+\gamma})^{\frac{\gamma}{1+\gamma}}}{|x'|^{1+\gamma}}\leq C
\begin{cases}
\varepsilon|\ln\varepsilon|,&d=2,\\
\varepsilon,&d\geq3.
\end{cases}
\end{align}
As for $\mathrm{III}_{3}$, in light of \eqref{LKT6.005}, it follows from Corollaries \ref{thm86} and \ref{coro00z} again that
\begin{align*}
\mathrm{III}_{3}=&\int_{\Omega^{\ast}_{R}}(\mathbb{C}^{0}e(v_{1}^{\ast\alpha}+v_{2}^{\ast\alpha}),e(v_{1}^{\ast\beta}))-\int_{\Omega^{\ast}_{\varepsilon^{\tilde{\theta}}}}(\mathbb{C}^{0}e(v_{1}^{\ast\alpha}+v_{2}^{\ast\alpha}),e(v_{1}^{\ast\beta}))\\
=&\int_{\Omega^{\ast}_{R}}(\mathbb{C}^{0}e(v_{1}^{\ast\alpha}+v_{2}^{\ast\alpha}),e(v_{1}^{\ast\beta}))+O(1)\varepsilon^{\frac{\gamma(d+\gamma-1)}{2(1+\gamma)^{2}}},
\end{align*}
which, together with \eqref{GAP001}--\eqref{GAZ00195}, gives that
\begin{align}\label{FLW001}
\mathrm{III}_{4}=&\int_{\Omega^{\ast}_{R}}(\mathbb{C}^{0}e(v_{1}^{\ast\alpha}+v_{2}^{\ast\alpha}),e(v_{1}^{\ast\beta}))+O(1)\varepsilon^{\frac{\gamma}{2(1+\gamma)}}.
\end{align}
Hence, combining \eqref{GAZ0011}--\eqref{GAZ001} and \eqref{FLW001}, we obtain
\begin{align*}
\sum\limits^{2}_{i=1}a_{i1}^{\alpha\beta}=\sum\limits^{2}_{i=1}a_{i1}^{\ast\alpha\beta}+O(1)\varepsilon^{\frac{\gamma}{2(1+\gamma)}}.
\end{align*}
Similarly, we have
\begin{align*}
\sum\limits^{2}_{j=1}a_{1j}^{\alpha\beta}=\sum\limits^{2}_{j=1}a_{1j}^{*\alpha\beta}+O(\varepsilon^{\frac{\gamma}{2(1+\gamma)}}),\quad \sum\limits^{2}_{i,j=1}a_{ij}^{\alpha\beta}=\sum\limits^{2}_{i,j=1}a_{ij}^{*\alpha\beta}+O(\varepsilon^{\frac{\gamma}{2(1+\gamma)}}).
\end{align*}

Therefore, we prove that \eqref{AZQ001}--\eqref{AZQ00111} hold.

\end{proof}
Before giving the proof of Theorem \ref{OMG123}, we first list a result on the linear space of rigid displacement $\Psi$ with its proof seen in Lemma 6.1 of \cite{BLL2017}.
\begin{lemma}\label{GLW}
Let $\xi$ be an element of $\Psi$, defined by \eqref{LAK01} with $d\geq2$. If $\xi$ vanishes at $d$ distinct points $\bar{x}_{1}$, $i=1,2,...,d$, which do not lie on a $(d-1)$-dimensional plane, then $\xi=0$.
\end{lemma}

\begin{proof}[Proof of Theorem \ref{OMG123}]
We now divide into two parts to complete the proof of Theorem \ref{OMG123}.

{\bf Step 1.} If $d=2$, we define
\begin{gather*}
\mathbb{F}_{0}:=\begin{pmatrix}a_{11}^{33}&\sum\limits_{i=1}^{2}a_{i1}^{31}&\sum\limits_{i=1}^{2}a_{i1}^{32}&\sum\limits_{i=1}^{2}a_{i1}^{33} \\ \sum\limits_{j=1}^{2}a_{1j}^{13}&\sum\limits_{i,j=1}^{2}a_{ij}^{11}&\sum\limits_{i,j=1}^{2}a_{ij}^{12}&\sum\limits_{i,j=1}^{2}a_{ij}^{13}\\
\sum\limits_{j=1}^{2}a_{1j}^{23}&\sum\limits_{i,j=1}^{2}a_{ij}^{21}&\sum\limits_{i,j=1}^{2}a_{ij}^{22}&\sum\limits_{i,j=1}^{2}a_{ij}^{23}\\
\sum\limits_{j=1}^{2}a_{1j}^{33}&\sum\limits_{i,j=1}^{2}a_{ij}^{31}&\sum\limits_{i,j=1}^{2}a_{ij}^{32}&\sum\limits_{i,j=1}^{2}a_{ij}^{33}
\end{pmatrix}.
\end{gather*}
For $\alpha=1,2$, denote
\begin{gather*}
\mathbb{F}_{1}^{\alpha}:=\begin{pmatrix} b_{1}^{\alpha}&a_{11}^{\alpha3}&\sum\limits_{i=1}^{2}a_{i1}^{\alpha1}&\sum\limits_{i=1}^{2}a_{i1}^{\alpha2}&\sum\limits_{i=1}^{2}a_{i1}^{\alpha3} \\
b_{1}^{3}&a_{11}^{33}&\sum\limits_{i=1}^{2}a_{i1}^{31}&\sum\limits_{i=1}^{2}a_{i1}^{32}&\sum\limits_{i=1}^{2}a_{i1}^{33}\\
\sum\limits_{i=1}^{2}b_{i}^{1}&\sum\limits_{j=1}^{2}a_{1j}^{13}&\sum\limits_{i,j=1}^{2}a_{ij}^{11}&\sum\limits_{i,j=1}^{2}a_{ij}^{12}&\sum\limits_{i,j=1}^{2}a_{ij}^{13}\\
\sum\limits_{i=1}^{2}b_{i}^{2}&\sum\limits_{j=1}^{2}a_{1j}^{23}&\sum\limits_{i,j=1}^{2}a_{ij}^{21}&\sum\limits_{i,j=1}^{2}a_{ij}^{22}&\sum\limits_{i,j=1}^{2}a_{ij}^{23}\\
\sum\limits_{i=1}^{2}b_{i}^{3}&\sum\limits_{j=1}^{2}a_{1j}^{33}&\sum\limits_{i,j=1}^{2}a_{ij}^{31}&\sum\limits_{i,j=1}^{2}a_{ij}^{32}&\sum\limits_{i,j=1}^{2}a_{ij}^{33}
\end{pmatrix},
\end{gather*}
and for $\alpha=3$,
\begin{gather*}
\mathbb{F}_{1}^{3}:=\begin{pmatrix}
b_{1}^{3}&\sum\limits_{i=1}^{2}a_{i1}^{31}&\sum\limits_{i=1}^{2}a_{i1}^{32}&\sum\limits_{i=1}^{2}a_{i1}^{33} \\ \sum\limits_{i=1}^{2}b_{i}^{1}&\sum\limits_{i,j=1}^{2}a_{ij}^{11}&\sum\limits_{i,j=1}^{2}a_{ij}^{12}&\sum\limits_{i,j=1}^{2}a_{ij}^{13}\\
\sum\limits_{i=1}^{2}b_{i}^{2}&\sum\limits_{i,j=1}^{2}a_{ij}^{21}&\sum\limits_{i,j=1}^{2}a_{ij}^{22}&\sum\limits_{i,j=1}^{2}a_{ij}^{23}\\
\sum\limits_{i=1}^{2}b_{i}^{3}&\sum\limits_{i,j=1}^{2}a_{ij}^{31}&\sum\limits_{i,j=1}^{2}a_{ij}^{32}&\sum\limits_{i,j=1}^{2}a_{ij}^{33}
\end{pmatrix}.
\end{gather*}

Then it follows from Lemma \ref{KM323} and \eqref{zzw002} that
\begin{align*}
\det\mathbb{F}_{1}^{\alpha}=\det\mathbb{F}_{1}^{\ast\alpha}+O(\varepsilon^{\frac{\gamma^{2}}{2(1+2\gamma)(1+\gamma)^{2}}}),\quad \alpha=1,2,
\end{align*}
and
\begin{align*}
\det\mathbb{F}_{1}^{3}=\det\mathbb{F}_{1}^{\ast\alpha}+O(\varepsilon^{\frac{\gamma}{2(1+\gamma)}}),\quad\det\mathbb{F}_{0}=\det\mathbb{F}^{\ast}_{0}+O(\varepsilon^{\frac{\gamma}{2(1+2\gamma)}}),
\end{align*}
which yields that for $\alpha=1,2,$
\begin{align}\label{QGH01}
\frac{\det\mathbb{F}_{1}^{\alpha}
}{\det\mathbb{F}_{0}}=&\frac{\det\mathbb{F}_{1}^{\ast\alpha}}{\det\mathbb{F}_{0}^{\ast}}\frac{1}{1-{\frac{\det\mathbb{F}_{0}^{\ast}-\det\mathbb{F}_{0}}{\det\mathbb{F}_{0}^{\ast}}}}+\frac{\det\mathbb{F}_{1}^{\alpha}-\det\mathbb{F}_{1}^{\ast\alpha}}{\det\mathbb{F}_{0}}\notag\\
=&\frac{\det\mathbb{F}^{\ast\alpha}_{1}}{\det\mathbb{F}_{0}^{\ast}}(1+O(\varepsilon^{\frac{\gamma^{2}}{2(1+2\gamma)(1+\gamma)^{2}}})),
\end{align}
and
\begin{align}\label{DBU01}
\frac{\det\mathbb{F}_{1}^{3}
}{\det\mathbb{F}_{0}}=&\frac{\det\mathbb{F}_{1}^{\ast3}}{\det\mathbb{F}_{0}^{\ast}}\frac{1}{1-{\frac{\det\mathbb{F}_{0}^{\ast}-\det\mathbb{F}_{0}}{\det\mathbb{F}_{0}^{\ast}}}}+\frac{\det\mathbb{F}_{1}^{3}-\det\mathbb{F}_{1}^{\ast3}}{\det\mathbb{F}_{0}}\notag\\
=&\frac{\det\mathbb{F}^{\ast3}_{1}}{\det\mathbb{F}_{0}^{\ast}}(1+O(\varepsilon^{\frac{\gamma}{2(1+2\gamma)}})).
\end{align}
We now claim that $\det\mathbb{F}_{0}^{\ast}\neq0$. In fact, for any $\xi=(\xi_{1},\xi_{2},\xi_{3},\xi_{4})^{T}\neq0$, we see from ellipticity condition \eqref{ellip} that
\begin{align*}
\xi^{T}\mathbb{F}_{0}^{\ast}\xi=&\int_{\Omega^{\ast}}\Bigg(\mathbb{C}^{0}e\bigg(\xi_{1}v_{1}^{\ast3}+\sum^{3}_{\alpha=1}\xi_{\alpha+1}(v_{1}^{\ast\alpha}+v_{2}^{\ast\alpha})\bigg), e\bigg(\xi_{1}v_{1}^{\ast3}+\sum^{3}_{\beta=1}\xi_{\beta+1}(v_{1}^{\ast\beta}+v_{2}^{\ast\beta})\bigg)\Bigg)\notag\\
\geq&\frac{1}{C}\int_{\Omega^{\ast}}\bigg|e\bigg(\xi_{1}v_{1}^{\ast3}+\sum^{3}_{\alpha=1}\xi_{\alpha+1}(v_{1}^{\ast\alpha}+v_{2}^{\ast\alpha})\bigg)\bigg|^{2}>0,
\end{align*}
where in the last inequality we used the fact that $e\big(\xi_{1}v_{1}^{\ast3}+\sum^{3}_{\alpha=1}\xi_{\alpha+1}(v_{1}^{\ast\alpha}+v_{2}^{\ast\alpha})\big)$ is not identically zero. Otherwise, if $$e\bigg(\xi_{1}v_{1}^{\ast3}+\sum^{3}_{\alpha=1}\xi_{\alpha+1}(v_{1}^{\ast\alpha}+v_{2}^{\ast\alpha})\bigg)=0,$$ then
\begin{align}\label{DAK}
\xi_{1}v_{1}^{\ast3}+\sum^{3}_{\alpha=1}\xi_{\alpha+1}(v_{1}^{\ast\alpha}+v_{2}^{\ast\alpha})=\sum^{3}_{i=1}a_{i}\psi_{i},
\end{align}
for some constants $a_{i}$, $i=1,2,3$. In view of the fact that $v_{1}^{\ast\alpha}=v_{2}^{\ast\alpha}=0$ on $\partial D$, it follows from \eqref{DAK} that $\sum^{3}_{i=1}a_{i}\psi_{i}=0$, which implies that $a_{i}=0$, $i=1,2,3$. Since
\begin{align*}
0=&\xi_{1}v_{1}^{\ast3}+\sum^{3}_{\alpha=1}\xi_{\alpha+1}(v_{1}^{\ast\alpha}+v_{2}^{\ast\alpha})\notag\\
=&
\begin{cases}
\sum\limits^{2}_{\alpha=1}\xi_{\alpha+1}\psi_{\alpha}+(\xi_{1}+\xi_{4})\psi_{3},&\mathrm{on}\;\partial D_{1}^{\ast},\\
\sum\limits^{3}_{\alpha=1}\xi_{\alpha+1}\psi_{\alpha},&\mathrm{on}\;\partial D_{2},
\end{cases}
\end{align*}
then we obtain that $\xi=0$. This is a contradiction.

In light of \eqref{zzw001}, we obtain that for $i=1,2,$
\begin{align}\label{AJDC0}
\frac{1}{a_{ii}}=&\frac{\varepsilon^{\frac{\gamma}{1+\gamma}}}{\mathcal{L}_{2}^{i}\mathcal{M}_{\gamma,\tau}}\frac{1}{1-\frac{\mathcal{L}_{2}^{i}\mathcal{M}_{\gamma,\tau}-\varepsilon^{\frac{\gamma}{1+\gamma}}a_{ii}}{\mathcal{L}_{2}^{i}\mathcal{M}_{\gamma,\tau}}}=\frac{\varepsilon^{\frac{\gamma}{1+\gamma}}(1+O(\tilde{\varepsilon}(\gamma,\sigma)))}{\mathcal{L}_{2}^{i}\mathcal{M}_{\gamma,\tau}},
\end{align}
where $\tilde{\varepsilon}(\gamma,\sigma)$ is defined by \eqref{APDN001}.
Then combining \eqref{PLA001}, \eqref{QGH01}--\eqref{DBU01} and \eqref{AJDC0}, it follows from Cramer's rule that for $\alpha=1,2,$
\begin{align*}
C_{1}^{\alpha}-C_{2}^{\alpha}=&\frac{\prod\limits_{i\neq\alpha}^{2}a_{11}^{ii}\det\mathbb{F}_{1}^{\alpha}}{\prod\limits_{i=1}^{2}a_{11}^{ii}\det \mathbb{F}_{0}}(1+O(\varepsilon^{\frac{\gamma}{1+\gamma}}|\ln\varepsilon|))=\frac{\det\mathbb{F}_{1}^{\ast\alpha}}{\det \mathbb{F}_{0}^{\ast}}\frac{\varepsilon^{\frac{\gamma}{1+\gamma}}(1+O(\varepsilon(\gamma,\sigma)))}{\mathcal{L}_{2}^{i}\mathcal{M}_{\gamma,\tau}},
\end{align*}
and for $\alpha=3$,
\begin{align*}
C_{1}^{3}-C_{2}^{3}=&\frac{\det\mathbb{F}_{1}^{3}}{\det \mathbb{F}_{0}}(1+O(\varepsilon^{\frac{\gamma}{1+\gamma}}))=\frac{\det\mathbb{F}_{1}^{\ast3}}{\det \mathbb{F}_{0}^{\ast}}(1+O(\varepsilon^{\frac{\gamma}{2(1+2\gamma)}})).
\end{align*}

{\bf Step 2.} If $d\geq3$, we replace the elements of $\alpha$-th column in the matrices $\mathbb{A}$ and $\mathbb{C}$ by column vectors $\Big(b_{1}^{1},...,b_{1}^{\frac{d(d+1)}{2}}\Big)^{T}$ and $\Big(\sum\limits_{i=1}^{2}b_{i}^{1},...,\sum\limits_{i=1}^{2}b_{i}^{\frac{d(d+1)}{2}}\Big)^{T}$, respectively, and then denote these two new matrices by $\mathbb{A}_{2}^{\alpha}$ and $\mathbb{C}_{2}^{\alpha}$ as follows:
\begin{gather*}
\mathbb{A}_{2}^{\alpha}=
\begin{pmatrix}
a_{11}^{11}&\cdots&b_{1}^{1}&\cdots&a_{11}^{1\,\frac{d(d+1)}{2}} \\\\ \vdots&\ddots&\vdots&\ddots&\vdots\\\\a_{11}^{\frac{d(d+1)}{2}\,1}&\cdots&b_{1}^{\frac{d(d+1)}{2}}&\cdots&a_{11}^{\frac{d(d+1)}{2}\,\frac{d(d+1)}{2}}
\end{pmatrix},
\end{gather*}
and
\begin{gather*}
\mathbb{C}_{2}^{\alpha}=
\begin{pmatrix}
\sum\limits_{j=1}^{2} a_{1j}^{11}&\cdots&\sum\limits_{i=1}^{2}b_{i}^{1}&\cdots&\sum\limits_{j=1}^{2} a_{1j}^{1\,\frac{d(d+1)}{2}} \\\\ \vdots&\ddots&\vdots&\ddots&\vdots\\\\\sum\limits_{j=1}^{2} a_{1j}^{\frac{d(d+1)}{2}\,1}&\cdots&\sum\limits_{i=1}^{2}b_{i}^{\frac{d(d+1)}{2}}&\cdots&\sum\limits_{j=1}^{2} a_{1j}^{\frac{d(d+1)}{2}\,\frac{d(d+1)}{2}}
\end{pmatrix}.
\end{gather*}
Define
\begin{align*}
\mathbb{F}^{\alpha}_{2}=\begin{pmatrix} \mathbb{A}^{\alpha}_{2}&\mathbb{B} \\  \mathbb{C}^{\alpha}_{2}&\mathbb{D}
\end{pmatrix},\quad \mathbb{F}=\begin{pmatrix} \mathbb{A}&\mathbb{B} \\  \mathbb{C}&\mathbb{D}
\end{pmatrix}.
\end{align*}
Then it follows from Lemmas \ref{KM323}--\ref{lemmabc} that
\begin{align*}
\det\mathbb{F}_{2}^{\alpha}=\det\mathbb{F}_{2}^{\ast\alpha}+O(\bar{\varepsilon}(\gamma,d)),\quad\det\mathbb{F}=\det\mathbb{F}^{\ast}+O(\bar{\varepsilon}(\gamma,d)).
\end{align*}
Similarly as before, we obtain that $\det\mathbb{F}^{\ast}\neq0$. Thus we obtain
\begin{align*}
\frac{\det\mathbb{F}_{2}^{\alpha}
}{\det\mathbb{F}}=&\frac{\det\mathbb{F}_{2}^{\ast\alpha}}{\det\mathbb{F}^{\ast}}\frac{1}{1-{\frac{\det\mathbb{F}^{\ast}-\det\mathbb{F}}{\det\mathbb{F}^{\ast}}}}+\frac{\det\mathbb{F}_{2}^{\alpha}-\det\mathbb{F}_{2}^{\ast\alpha}}{\det\mathbb{F}}=\frac{\det\mathbb{F}^{\ast\alpha}_{2}}{\det\mathbb{F}^{\ast}}(1+O(\bar{\varepsilon}(\gamma,d))),
\end{align*}
which, together with \eqref{PLA001} and Cramer's rule, reads that for $\alpha=1,2,...,\frac{d(d+1)}{2}$,
\begin{align*}
C_{1}^{\alpha}-C_{2}^{\alpha}=&\frac{\det\mathbb{F}_{2}^{\alpha}}{\det \mathbb{F}}=\frac{\det\mathbb{F}_{2}^{\ast\alpha}}{\det \mathbb{F}^{\ast}}(1+O(\bar{\varepsilon}(\gamma,d))).
\end{align*}

\end{proof}

\section{An example of two adjacent curvilinear squares with rounded-off angles}\label{SEC006}

\begin{figure}[htb]
\center{\includegraphics[width=0.45\textwidth]{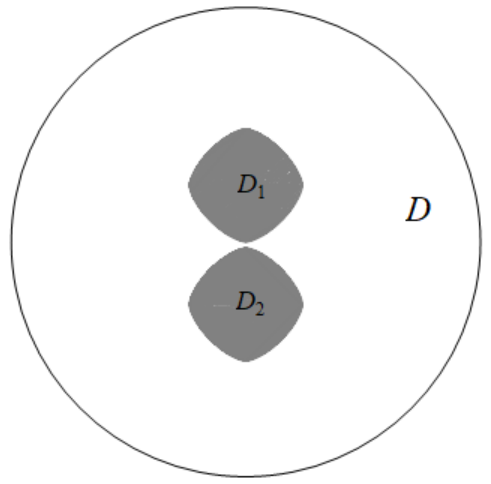}}
\caption{Curvilinear squares with rounded-off angles, $\gamma=\frac{1}{2}$}
\end{figure}
In this section, we aim to give a more precise characterization of the stress concentration for two adjacent curvilinear squares with rounded-off angles in dimension two, see Figure 1. Suppose that the interfacial boundaries of inclusions $\partial D_{1}$ and $\partial D_{2}$ can be, respectively, represented by
\begin{align}\label{ZZW986}
|x_{1}|^{1+\gamma}+|x_{2}-\varepsilon-r_{1}|^{1+\gamma}=&r_{1}^{1+\gamma}\quad\mathrm{and}\quad|x_{1}|^{1+\gamma}+|x_{2}+r_{2}|^{1+\gamma}=r_{2}^{1+\gamma},
\end{align}
where $r_{i}$, $i=1,2$ are two positive constants independent of $\varepsilon$. Define
\begin{align}\label{ZCZ009}
\tau_{0}:=\frac{1}{1+\gamma}\left(\frac{1}{r_{1}^{\gamma}}+\frac{1}{r_{2}^{\gamma}}\right).
\end{align}
Then, we obtain
\begin{example}\label{coro00389}
Assume as above, condition \eqref{ZZW986} holds. Let $u\in H^{1}(D;\mathbb{R}^{2})\cap C^{1}(\overline{\Omega};\mathbb{R}^{2})$ be the solution of (\ref{La.002}). Then for a sufficiently small $\varepsilon>0$ and $x\in\Omega_{r_{0}}$, $0<r_{0}<\frac{1}{2}\min\{r_{1},r_{2}\}$ is a small constant independent of $\varepsilon$,
\begin{align}\label{ALPN001}
\nabla u=&\sum\limits^{2}_{\alpha=1}\frac{\det\mathbb{F}_{1}^{\ast\alpha}}{\det \mathbb{F}_{0}^{\ast}}\frac{\varepsilon^{\frac{\gamma}{1+\gamma}}}{\mathcal{L}_{2}^{\alpha}\mathcal{M}_{\gamma,\tau_{0}}}\frac{1+O(\varepsilon^{\frac{\gamma^{2}}{2(1+2\gamma)(1+\gamma)^{2}}})}{1+\mathcal{G}^{\ast}_{\alpha}\varepsilon^{\frac{\gamma}{1+\gamma}}}\nabla\bar{u}_{1}^{\alpha}\notag\\
&+\frac{\det\mathbb{F}_{1}^{\ast3}}{\det \mathbb{F}_{0}^{\ast}}(1+O(\varepsilon^{\frac{\gamma}{2(1+2\gamma)}}))\nabla\bar{u}_{1}^{3}+O(1)\delta^{-\frac{1-\alpha}{1+\alpha}}\|\varphi\|_{C^{1}(\partial D)},
\end{align}
where $\delta$ is defined in \eqref{deta}, the explicit auxiliary functions $\bar{u}_{1}^{\alpha}$, $\alpha=1,2,3$ are defined in \eqref{zzwz002} in the case of $d=2$, the constant $\mathcal{M}_{\gamma,\tau_{0}}$ is defined in \eqref{zwzh001} with $\tau=\tau_{0}$, the Lam\'{e} constants $\mathcal{L}_{2}^{\alpha}$, $\alpha=1,2$ is defined in \eqref{AZ}, the blow-up factor matrices $\mathbb{F}_{0}^{\ast}$ and $\mathbb{F}^{\ast\alpha}_{1},$ $\alpha=1,2,3$ are defined by \eqref{ZWZML001}--\eqref{ZWZML003}, the rest term $\bar{\varepsilon}(\gamma,\sigma)$ is defined in \eqref{GC002}, the geometry constants $\mathcal{G}^{\ast}_{\alpha}$, $\alpha=1,2$ are defined by \eqref{QKLP001} below.
\end{example}
\begin{remark}
From the view of industrial application and numerical computation, this type of axisymmetric inclusions considered in Example \ref{coro00389} is more realistic than the generalized $C^{1,\gamma}$-inclusions due to its explicit regular shapes. We then give a more precise characterization in terms of the singular behavior of the stress concentration in virtue of the $\varepsilon$-independent geometry constant $\mathcal{G}^{\ast}_{\alpha}$, $\alpha=1,2$ captured in \eqref{ALPN001}.

\end{remark}

\begin{lemma}
Assume as in Example \ref{coro00389}. Then, for a sufficiently small $\varepsilon>0$, $\alpha=1,2,$
\begin{align}\label{KAN001}
a_{11}^{\alpha\alpha}=&\mathcal{L}_{2}^{\alpha}\mathcal{M}_{\gamma,\tau_{0}}\varepsilon^{-\frac{\gamma}{1+\gamma}}+\mathcal{K}^{\ast}_{\alpha}+O(1)|\ln\varepsilon|,
\end{align}
where $\mathcal{M}_{\gamma,\tau_{0}}$ is defined by \eqref{zwzh001} with $\tau=\tau_{0}$, $\mathcal{L}_{2}^{\alpha}$, $\alpha=1,2$ are defined in \eqref{AZ} with $d=2$, $\mathcal{K}^{\ast}_{\alpha}$, $\alpha=1,2$ are defined by \eqref{LKM} below.

\end{lemma}

\begin{proof}
Pick $\theta=\frac{\gamma^{2}}{2(1+2\gamma)(1+\gamma)^{2}}$. Similarly as in \eqref{aaaa01}, we obtain that for $\alpha=1,2$,
\begin{align*}
a_{11}^{\alpha\alpha}=&\mathcal{L}_{2}^{\alpha}\left(\int_{\varepsilon^{\theta}<|x_{1}|<r_{0}}\frac{dx_{1}}{h_{1}(x_{1})-h_{2}(x_{1})}+\int_{|x_{1}|<\varepsilon^{\theta}}\frac{dx_{1}}{\varepsilon+h_{1}(x_{1})-h_{2}(x_{1})}\right)+O(1)|\ln\varepsilon|.
\end{align*}

To begin with, it follows from Taylor expansion that
\begin{align}\label{AMBZ001}
h_{1}(x_{1})-h_{2}(x_{1})=\tau_{0}|x_{1}|^{1+\gamma}+O(|x_{1}|^{2+2\gamma}),\quad |x_{1}|\leq r_{0},
\end{align}
where $\tau_{0}$ is defined in \eqref{ZCZ009}. Using \eqref{AMBZ001}, we have
\begin{align*}
\int_{\varepsilon^{\theta}<|x_{1}|<r_{0}}\left(\frac{1}{h_{1}-h_{2}}-\frac{1}{\tau_{0}|x_{1}|^{1+\gamma}}\right)dx_{1}=\int_{\varepsilon^{\theta}<|x_{1}|<r_{0}}O(1)dx_{1}=C^{\ast}+O(1)\varepsilon^{\theta},
\end{align*}
where $C^{\ast}$ depends on $\tau_{0},r_{0}$, but not on $\varepsilon$. Then
\begin{align*}
\int_{\varepsilon^{\theta}<|x_{1}|<r_{0}}\frac{dx_{1}}{h_{1}-h_{2}}=&\int_{\varepsilon^{\theta}<|x_{1}|<r_{0}}\frac{dx_{1}}{\tau_{0}|x_{1}|^{1+\gamma}}+C^{\ast}+O(1)\varepsilon^{\theta}.
\end{align*}
Analogously, we have
\begin{align*}
\int_{|x_{1}|<\varepsilon^{\theta}}\frac{dx_{1}}{\varepsilon+h_{1}-h_{2}}=&\int_{|x_{1}|<\varepsilon^{\theta}}\frac{dx_{1}}{\varepsilon+\tau_{0}|x_{1}|^{1+\gamma}}+O(1)\varepsilon^{\theta}.
\end{align*}
Therefore, the energy $a_{11}^{\alpha\alpha}$ becomes
\begin{align*}
a_{11}^{\alpha\alpha}=&\mathcal{L}_{2}^{\alpha}\left(\int_{\varepsilon^{\theta}<|x_{1}|<r_{0}}\frac{dx_{1}}{\tau_{0}|x_{1}|^{1+\gamma}}+\int_{|x_{1}|<\varepsilon^{\theta}}\frac{dx_{1}}{\varepsilon+\tau_{0}|x_{1}|^{1+\gamma}}\right)+C^{\ast}+O(1)|\ln\varepsilon|.
\end{align*}
Observe that
\begin{align*}
&\int_{\varepsilon^{\theta}<|x_{1}|<r_{0}}\frac{dx_{1}}{\tau_{0}|x_{1}|^{1+\gamma}}+\int_{|x_{1}|<\varepsilon^{\theta}}\frac{dx_{1}}{\varepsilon+\tau_{0}|x_{1}|^{1+\gamma}}\notag\\
=&\int_{-\infty}^{+\infty}\frac{1}{\varepsilon+\tau_{0}|x_{1}|^{1+\gamma}}-\int_{|x_{1}|>r_{0}}\frac{dx_{1}}{\tau_{0}|x_{1}|^{1+\gamma}}+\int_{|x_{1}|>\varepsilon^{\theta}}\frac{\varepsilon}{\tau_{0}|x_{1}|^{1+\gamma}(\varepsilon+\tau_{0}|x_{1}|^{1+\gamma})}\notag\\
=&\mathcal{M}_{\gamma,\tau_{0}}\varepsilon^{-\frac{\gamma}{1+\gamma}}-\frac{2}{\gamma\tau_{0}r_{0}^{\gamma}}+O(1)\varepsilon^{1-(1+2\gamma)\theta},
\end{align*}
we deduce that for $\alpha=1,2,$
\begin{align*}
a_{11}^{\alpha\alpha}=&\mathcal{L}_{2}^{\alpha}\mathcal{M}_{\gamma,\tau_{0}}\varepsilon^{-\frac{\alpha}{1+\alpha}}+\mathcal{K}^{\ast}_{\alpha}+O(1)|\ln\varepsilon|,
\end{align*}
where
\begin{align}\label{LKM}
\mathcal{K}^{\ast}_{\alpha}=C^{\ast}-\frac{2\mathcal{L}_{2}^{\alpha}}{\gamma\tau_{0}r_{0}^{\gamma}}.
\end{align}

\end{proof}

\begin{proof}[Proof of Example \ref{coro00389}]

Denote
\begin{align}\label{QKLP001}
\mathcal{G}^{\ast}_{\alpha}=\frac{\mathcal{K}^{\ast}_{\alpha}}{\mathcal{L}_{2}^{\alpha}\mathcal{M}_{\gamma,\tau_{0}}},\quad \alpha=1,2.
\end{align}
Making use of \eqref{KAN001}, we obtain
\begin{align}\label{DYA001}
\frac{1}{a_{11}^{\alpha\alpha}}=&\frac{\varepsilon^{\frac{\gamma}{1+\gamma}}}{\mathcal{L}_{2}^{\alpha}\mathcal{M}_{\gamma,\tau_{0}}}\frac{1}{1-\frac{\mathcal{L}_{2}^{\alpha}\mathcal{M}_{\gamma,\tau_{0}}-\varepsilon^{\frac{\gamma}{1+\gamma}}a_{11}^{\alpha\alpha}}{\mathcal{L}_{2}^{\alpha}\mathcal{M}_{\gamma,\tau_{0}}}}=\frac{\varepsilon^{\frac{\gamma}{1+\gamma}}}{\mathcal{L}_{2}^{\alpha}\mathcal{M}_{\gamma,\tau_{0}}}\frac{1}{1+\mathcal{G}^{\ast}_{\alpha}\varepsilon^{\frac{\gamma}{1+\gamma}}+O(\varepsilon^{\frac{\gamma}{1+\gamma}}|\ln\varepsilon|)}\notag\\
=&\frac{\varepsilon^{\frac{\gamma}{1+\gamma}}}{\mathcal{L}_{2}^{\alpha}\mathcal{M}_{\gamma,\tau_{0}}}\frac{1+O(\varepsilon^{\frac{\gamma}{1+\gamma}}|\ln\varepsilon|)}{1+\mathcal{G}^{\ast}_{\alpha}\varepsilon^{\frac{\gamma}{1+\gamma}}}.
\end{align}
Then in view of \eqref{QGH01}--\eqref{DBU01} and \eqref{DYA001}, it follows from \eqref{PLA001} and Cramer's rule that for $\alpha=1,2,$
\begin{align*}
C_{1}^{\alpha}-C_{2}^{\alpha}=&\frac{\prod\limits_{i\neq\alpha}^{2}a_{11}^{ii}\det\mathbb{F}_{1}^{\alpha}}{\prod\limits_{i=1}^{2}a_{11}^{ii}\det \mathbb{F}_{0}}(1+O(\varepsilon^{\frac{\gamma}{1+\gamma}}|\ln\varepsilon|))\notag\\
=&\frac{\det\mathbb{F}_{1}^{\ast\alpha}}{\det \mathbb{F}_{0}^{\ast}}\frac{\varepsilon^{\frac{\gamma}{1+\gamma}}}{\mathcal{L}_{2}^{\alpha}\mathcal{M}_{\gamma,\tau_{0}}}\frac{1+O(\varepsilon^{\frac{\gamma^{2}}{2(1+2\gamma)(1+\gamma)^{2}}})}{1+\mathcal{G}^{\ast}_{\alpha}\varepsilon^{\frac{\gamma}{1+\gamma}}},
\end{align*}
and for $\alpha=3$,
\begin{align*}
C_{1}^{3}-C_{2}^{3}=&\frac{\det\mathbb{F}_{1}^{3}}{\det \mathbb{F}_{0}}(1+O(\varepsilon^{\frac{\gamma}{1+\gamma}}))=\frac{\det\mathbb{F}_{1}^{\ast3}}{\det \mathbb{F}_{0}^{\ast}}(1+O(\varepsilon^{\frac{\gamma}{2(1+2\gamma)}})).
\end{align*}
This, in combination with decomposition \eqref{Decom002}, Corollaries \ref{thm86}--\ref{coro00z}, Lemma \ref{PAK001} and Theorem \ref{OMG123}, reads that
\begin{align*}
\nabla u=&\sum\limits^{2}_{\alpha=1}\frac{\det\mathbb{F}_{1}^{\ast\alpha}}{\det \mathbb{F}_{0}^{\ast}}\frac{\varepsilon^{\frac{\gamma}{1+\gamma}}}{\mathcal{L}_{2}^{\alpha}\mathcal{M}_{\gamma,\tau_{0}}}\frac{1+O(\varepsilon^{\frac{\gamma^{2}}{2(1+2\gamma)(1+\gamma)^{2}}})}{1+\mathcal{G}^{\ast}_{\alpha}\varepsilon^{\frac{\gamma}{1+\gamma}}}(\nabla\bar{u}_{1}^{\alpha}+O(\delta^{-\frac{1}{1+\alpha}}))\notag\\
&+\frac{\det\mathbb{F}_{1}^{\ast3}}{\det \mathbb{F}_{0}^{\ast}}(1+O(\varepsilon^{\frac{\gamma}{2(1+2\gamma)}}))(\nabla\bar{u}_{1}^{3}+O(1))+O(1)\delta^{-\frac{d}{2}}e^{-\frac{1}{2C\delta^{\gamma/(1+\gamma)}}}\notag\\
=&\sum\limits^{2}_{\alpha=1}\frac{\det\mathbb{F}_{1}^{\ast\alpha}}{\det \mathbb{F}_{0}^{\ast}}\frac{\varepsilon^{\frac{\gamma}{1+\gamma}}}{\mathcal{L}_{2}^{\alpha}\mathcal{M}_{\gamma,\tau_{0}}}\frac{1+O(\varepsilon^{\frac{\gamma^{2}}{2(1+2\gamma)(1+\gamma)^{2}}})}{1+\mathcal{G}^{\ast}_{\alpha}\varepsilon^{\frac{\gamma}{1+\gamma}}}\nabla\bar{u}_{1}^{\alpha}\notag\\
&+\frac{\det\mathbb{F}_{1}^{\ast3}}{\det \mathbb{F}_{0}^{\ast}}(1+O(\varepsilon^{\frac{\gamma}{2(1+2\gamma)}}))\nabla\bar{u}_{1}^{3}+O(1)\delta^{-\frac{1-\alpha}{1+\alpha}}\|\varphi\|_{C^{1}(\partial D)}.
\end{align*}

\end{proof}

\section{Appendix:\,The proofs of Lemmas \ref{CL001} and \ref{CL002}}

\subsection{$C^{1,\gamma}$ estimates.}
The proof of Lemma \ref{CL001} is based on the Campanato's approach, which was presented previously, for example, in \cite{GM2013}. Assume that $Q\subseteq\mathbb{R}^{d}$ is a Lipschitz domain. Define the Campanato space $\mathcal{L}^{2,\lambda}(Q)$, $\lambda\geq0$ as follows:
\begin{align*}
\mathcal{L}^{2,\lambda}(Q):=\bigg\{u\in L^{2}(Q):\,\sup\limits_{\scriptstyle x_{0}\in Q\atop\scriptstyle\;\;
\rho>0\hfill}\frac{1}{\rho^{\lambda}}\int_{B_{\rho}(x_{0})\cap Q}|u-u_{x_{0},\rho}|^{2}dx<+\infty \bigg\},
\end{align*}
where $u_{x_{0},\rho}:=\frac{1}{|Q\cap B_{\rho}(x_{0})|}\int_{Q\cap B_{\rho}(x_{0})}u(x)\,dx$. We endow the Campanato space  $\mathcal{L}^{2,\lambda}(Q)$ with the norm
\begin{align*}
\|u\|_{\mathcal{L}^{2,\lambda}(Q)}:=\|u\|_{L^{2}(Q)}+[u]_{\mathcal{L}^{2,\lambda}(Q)},
\end{align*}
where
\begin{align*}
[u]^{2}_{\mathcal{L}^{2,\lambda}(Q)}:=\sup\limits_{\scriptstyle x_{0}\in Q\atop\scriptstyle\;\;
\rho>0\hfill}\frac{1}{\rho^{\lambda}}\int_{B_{\rho}(x_{0})\cap Q}|u-u_{x_{0},\rho}|^{2}dx.
\end{align*}

A direct application of Theorem 5.14 in \cite{GM2013} gives that
\begin{theorem}\label{ASDL0}
Let $Q\subset\mathbb{R}^{d}$ be a Lipschitz domain. Let $w\in H^{1}(Q;\mathbb{R}^{d})$ be a solution of
\begin{align*}
-\partial_{j}(C^{0}_{ijkl}\partial_{l}w^{k})=\partial_{j}f_{ij}, \quad in\; Q,
\end{align*}
with $f_{ij}\in C^{0,\gamma}(Q)$, $0<\gamma<1$, and constant coefficients $C_{ijkl}^{0}$ satisfying \eqref{ellip}. Then $\nabla w\in\mathcal{L}^{2,d+2\alpha}_{loc}(Q)$ and for $B_{R}:=B_{R}(x_{0})\subset Q$,
\begin{align*}
\|\nabla w\|_{\mathcal{L}^{2,d+2\alpha}(B_{R/2})}\leq C(\|\nabla w\|_{L^{2}(B_{R})}+[F]_{\mathcal{L}^{2,d+2\alpha}(B_{R})}),
\end{align*}
where $F=(f_{ij})$ and $C=C(d,\gamma,R)$.
\end{theorem}
Due to the equivalence that the Campanato space $\mathcal{L}^{2,\lambda}(Q)$ is equivalent to the H\"{o}lder space $C^{0,\gamma}(Q)$ in the case of $d<\gamma\leq d+2$ and $\gamma=\frac{\lambda-d}{2}$, it follows from the proof of Theorem \ref{ASDL0} (Theorem 5.14 of \cite{GM2013}) that
\begin{corollary}\label{CO001}
Assume as in Lemma \ref{CL001}. Let $w$ be the solution of \eqref{ADCo1}. Then for $B_{R}:=B_{R}(x_{0})\subset Q$,
\begin{align}\label{LAG001}
[\nabla w]_{\gamma,B_{R/2}}\leq C\Big(\frac{1}{R^{1+\gamma}}\|w\|_{L^{\infty}(B_{R})}+[F]_{\gamma,B_{R}} \Big),
\end{align}
where $C=C(d,\gamma,R)$.

\end{corollary}

\begin{proof}[Proof of Lemma \ref{CL001}]
In view of $\Gamma\in C^{1,\gamma}$, then at each point $x_{0}\in\Gamma$, there exists a neighbourhood $U$ containing $x_{0}$ and a homeomorphism $\Psi\in C^{1,\gamma}(U)$ such that
\begin{align*}
\Psi(U\cap Q)=\mathcal{B}^{+}_{1}=\{y\in\mathcal{B}_{R}(0):y_{d}>0\},\quad \Psi(U\cap \Gamma)=\partial\mathcal{B}^{+}_{1}=\{y\in\mathcal{B}_{1}(0):y_{d}=0\},
\end{align*}
where $\mathcal{B}_{1}(0):=\{y\in\mathbb{R}^{d}:|y|<1\}$. Under the transformation $y=\Psi(x)=(\Psi^{1}(x),...,\Psi^{d}(x))$, we denote
\begin{align*}
\mathcal{W}(y):=w(\Phi^{-1}(y)),\quad \mathcal{J}:=\frac{\partial((\Psi^{-1})^{1},...,(\Psi^{-1})^{d})}{\partial(y^{1},...,y^{d})},\quad |\mathcal{J}(y)|:=\det\mathcal{J}(y),
\end{align*}
and
\begin{align*}
\mathcal{C}^{0}_{ijkl}(y):=&C^{0}_{i\hat{j}k\hat{l}}|\mathcal{J}(y)|(\partial_{\hat{l}}(\Psi^{-1})^{l}(y))^{-1}\partial_{\hat{j}}\Psi^{j}(\Psi^{-1}(y)),\\
\mathcal{F}_{ij}(y):=&|\mathcal{J}(y)|\partial_{\hat{l}}\Psi^{j}(\Psi^{-1}(y))f_{i\hat{l}}(\Psi^{-1}(y)).
\end{align*}
Therefore, recalling equation \eqref{ADCo1}, we know that $\mathcal{W}$ solves
\begin{align*}
\begin{cases}
-\partial_{j}(\mathcal{C}^{0}_{ijkl}(y)\partial_{l}\mathcal{W}^{k})=\partial_{j}\mathcal{F}_{ij},&\quad\mathrm{in}\;\mathcal{B}^{+}_{R},\\
\mathcal{W}=0,&\quad\mathrm{on}\;\partial \mathcal{B}_{R}^{+}\cap\partial\mathbb{R}^{d}_{+},
\end{cases}
\end{align*}
where $0<R\leq1$. Let $y_{0}=\Psi(x_{0})$. By freezing the coefficients, we have
\begin{align*}
-\partial_{j}(\mathcal{C}^{0}_{ijkl}(y_{0})\partial_{l}\mathcal{W}^{k})=\partial_{j}((\mathcal{C}^{0}_{ijkl}(y)-\mathcal{C}^{0}_{ijkl}(y_{0}))\partial_{l}\mathcal{W}^{k})+\partial_{j}\mathcal{F}_{ij}.
\end{align*}
Then it follows from the equivalence between the Campanato space and the H\"{o}lder space and the proof of Theorem 7.1 (Theorem 5.14 of \cite{GM2013}) again that
\begin{align*}
[\nabla\mathcal{W}]_{\gamma,\mathcal{B}^{+}_{R/2}}\leq& C\Big(\frac{1}{R^{1+\gamma}}\|\mathcal{W}\|_{L^{\infty}(\mathcal{B}^{+}_{R})}+[\mathcal{F}]_{\gamma,\mathcal{B}^{+}_{R}}\Big)\notag\\
&+C[(\mathcal{C}^{0}_{ijkl}(y)-\mathcal{C}^{0}_{ijkl}(y_{0}))\partial_{l}\mathcal{W}^{k}]_{\gamma,\mathcal{B}^{+}_{R}},
\end{align*}
where $\mathcal{F}:=(\mathcal{F}_{ij})$. In view of $\mathcal{C}^{0}_{ijkl}(y)\in C^{0,\gamma}$, we obtain
\begin{align*}
[(\mathcal{C}^{0}_{ijkl}(y)-\mathcal{C}^{0}_{ijkl}(y_{0}))\partial_{l}\mathcal{W}^{k}]_{\gamma,\mathcal{B}^{+}_{R}}\leq C(R^{\gamma}[\nabla\mathcal{W}]_{\gamma,\mathcal{B}^{+}_{R}}+\|\nabla\mathcal{W}\|_{L^{\infty}(\mathcal{B}^{+}_{R})}).
\end{align*}
A direct application of the interpolation inequality (for example, see Lemma 6.32 in \cite{GT1998}) gives that
\begin{align*}
\|\nabla\mathcal{W}\|_{L^{\infty}(\mathcal{B}^{+}_{R})}\leq R^{\gamma}[\nabla\mathcal{W}]_{\gamma,\mathcal{B}^{+}_{R}}+\frac{C}{R}\|\mathcal{W}\|_{L^{\infty}(\mathcal{B}^{+}_{R})},
\end{align*}
where $C=C(d)$. Hence, we obtain
\begin{align*}
[\nabla\mathcal{W}]_{\gamma,\mathcal{B}^{+}_{R/2}}\leq&C\Big(\frac{1}{R^{1+\gamma}}\|\mathcal{W}\|_{L^{\infty}(\mathcal{B}^{+}_{R})}+R^{\gamma}[\nabla\mathcal{W}]_{\gamma,\mathcal{B}^{+}_{R}}+[\mathcal{F}]_{\gamma,\mathcal{B}^{+}_{R}}\Big),
\end{align*}
which, together with the fact that $\Psi$ is a homeomorphism, yields that
\begin{align*}
[\nabla w]_{\gamma,\mathcal{N}^{'}}\leq&C\Big(\frac{1}{R^{1+\gamma}}\|w\|_{L^{\infty}(\mathcal{N})}+R^{\gamma}[\nabla w]_{\gamma,\mathcal{N}}+[F]_{\gamma,\mathcal{N}}\Big),
\end{align*}
where $\mathcal{N}'=\Psi^{-1}(\mathcal{B}^{+}_{R/2})$, $\mathcal{N}=\Psi^{-1}(\mathcal{B}^{+}_{R})$, and $C=C(d,\gamma,\Psi)$. Observe that there is a constant $0<\sigma<1$, independent of $R$, such that $B_{\sigma R}(x_{0})\cap Q\subset\mathcal{N}'$.

Consequently, for any domain $Q'\subset\subset Q\cup\Gamma$ and $x_{0}\in Q'\cap\Gamma$, there exist $\mathcal{R}_{0}:=\mathcal{R}_{0}(x_{0})$ and $C_{0}:=C_{0}(d,\gamma,x_{0})$ such that
\begin{align}\label{HMA003}
[\nabla w]_{\gamma,B_{\mathcal{R}_{0}}(x_{0})\cap Q'}\leq&C_{0}\Big(\mathcal{R}_{0}^{\gamma}[\nabla w]_{\gamma,Q'}+\frac{1}{\mathcal{R}_{0}^{1+\gamma}}\|w\|_{L^{\infty}(Q)}+[F]_{\gamma,Q}\Big).
\end{align}
Then it follows from the finite covering theorem that there exist finite $B_{\mathcal{R}_{i}}(x_{i})\in \{B_{\mathcal{R}_{0}/2}(x_{0})|\,x_{0}\in\Gamma\cap Q'\}$, $i=1,2,...,K$, covering $\Gamma\cap Q'$. Use $C_{i}$ to denote the constant in \eqref{HMA003} corresponding to $x_{i}$ and write
\begin{align*}
\overline{C}:=\max\limits_{1\leq i\leq K}\{C_{i}\},\quad \overline{\mathcal{R}}:=\min\limits_{1\leq i\leq K}\big\{\frac{\mathcal{R}_{i}}{2}\big\}.
\end{align*}
Thus, for any $x_{0}\in\Gamma\cap Q'$, there exists some $1\leq i_{0}\leq K$ such that $B_{\overline{R}}(x_{0})\subset B_{\overline{R}_{i_{0}}}(x_{i_{0}})$ and
\begin{align}\label{HMA005}
[\nabla w]_{\gamma,B_{\overline{R}}(x_{0})\cap Q'}\leq&\overline{C}\Big(\overline{R}^{\gamma}[\nabla w]_{\gamma,Q'}+\frac{1}{\overline{R}^{1+\gamma}}\|w\|_{L^{\infty}(Q)}+[F]_{\gamma,Q}\Big).
\end{align}

We further establish the estimates on $Q'$ in the following. Let $\widetilde{C}$ be the constant in \eqref{LAG001} of Corollary \ref{CO001}. Define
\begin{align*}
\widehat{C}:=\max\{\overline{C},\widetilde{C}\},\quad\widehat{\mathcal{R}}:=\min\{(3\widehat{C})^{-1/\gamma},\overline{\mathcal{R}}\}.
\end{align*}
Note that for any $x_{1},x_{2}\in Q'$, there are three cases to occur:
\begin{enumerate}
\item[$(i)$]
$|x_{1}-x_{2}|\geq\frac{\widehat{\mathcal{R}}}{2}$;
\item[$(ii)$]
there exists some $1\leq i_{0}\leq K$ such that $x_{1},x_{2}\in B_{\widehat{R}/2}(x_{i_{0}})\cap Q'$;
\item[$(iii)$]
$x_{1},x_{2}\in B_{\widehat{R}/2}\subset Q'$.
\end{enumerate}
If case $(i)$ holds, then we have
\begin{align*}
\frac{|\nabla w(x_{1})-\nabla w(x_{2})|}{|x_{1}-x_{2}|^{\gamma}}\leq \frac{2^{1+\gamma}}{\widehat{R}^{\gamma}}\|\nabla w\|_{L^{\infty}(Q')}.
\end{align*}
If case $(ii)$ holds, then it follows from \eqref{HMA005} that
\begin{align*}
\frac{|\nabla w(x_{1})-\nabla w(x_{2})|}{|x_{1}-x_{2}|^{\gamma}}\leq&[\nabla w]_{\gamma,B_{\widehat{R}/2(x_{i_{0}})\cap Q'}}\notag\\
\leq&\widehat{C}\Big(\widehat{R}^{\gamma}[\nabla w]_{\gamma,Q'}+\frac{1}{\widehat{R}^{1+\gamma}}\|w\|_{L^{\infty}(Q)}+[F]_{\gamma,Q}\Big).
\end{align*}
If case $(iii)$ holds, then we see from Corollary \ref{CO001} that
\begin{align*}
\frac{|\nabla w(x_{1})-\nabla w(x_{2})|}{|x_{1}-x_{2}|^{\gamma}}\leq&[\nabla w]_{\gamma,B_{\widehat{R}/2}}\leq\widehat{C}\Big(\frac{1}{\widehat{R}^{1+\gamma}}\|w\|_{L^{\infty}(Q)}+[F]_{\gamma,Q}\Big).
\end{align*}
Therefore, we deduce
\begin{align*}
[\nabla w]_{\gamma,Q'}\leq\widehat{C}\Big(\widehat{R}^{\gamma}[\nabla w]_{\gamma,Q'}+\frac{1}{\widehat{R}^{1+\gamma}}\|w\|_{L^{\infty}(Q)}+[F]_{\gamma,Q}\Big)+\frac{2^{1+\gamma}}{\widehat{R}^{\gamma}}\|\nabla w\|_{L^{\infty}(Q')}.
\end{align*}
Applying the interpolation inequality (see Lemma 6.32 in \cite{GT1998}) again, we obtain
\begin{align*}
\frac{2^{1+\gamma}}{\widehat{R}^{\gamma}}\|\nabla w\|_{L^{\infty}(Q')}\leq \frac{1}{3}[\nabla w]_{\gamma,Q'}+\frac{C}{\widehat{R}^{1+\gamma}}\|w\|_{L^{\infty}(Q')},
\end{align*}
where $C=C(d,\gamma)$. Due to the fact that $\widehat{R}\leq(3\widehat{C})^{-1/\gamma}$, we deduce that
\begin{align*}
[\nabla w]_{\gamma,Q'}\leq C(\|w\|_{L^{\infty}(Q)}+[F]_{\gamma,Q}),
\end{align*}
where $C=C(d,\gamma,Q',Q)$. This, together with the interpolation inequality, yields that \eqref{GNA001} holds.

\end{proof}

\subsection{$W^{1,p}$ estimates}
\begin{proof}[Proof of Lemma \ref{CL002}]
To begin with, we establish the $W^{1,p}$ interior estimates. Due to the fact that $w\neq0$ on $\partial B_{R}$ for any $B_{R}\subset Q$, we pick a smooth cut-off function $\eta\in C_{0}^{\infty}(B_{R})$ such that
\begin{align*}
0\leq\eta\leq1,\quad\eta=1\;\mathrm{in}\; B_{\rho},\quad|\nabla\eta|\leq\frac{C}{R-\rho}.
\end{align*}
Recalling equation \eqref{ADCo1}, we know that $\eta w$ solves
\begin{align*}
\int_{B_{R}}C^{0}_{ijkl}\partial_{l}(\eta w^{k})\partial_{j}\varphi^{i}=\int_{B_{R}}(G_{i}\varphi^{i}+\widetilde{F}\partial_{j}\varphi^{i}),\quad\forall\varphi\in C^{\infty}_{0}(B_{R};\mathbb{R}^{d}),
\end{align*}
where
\begin{align*}
G_{i}:=f_{ij}\partial_{j}\eta-C^{0}_{ijkl}\partial_{l}w^{k}\partial_{j}\eta,\quad\widetilde{F}_{ij}:=f_{ij}\eta+C^{0}_{ijkl}w^{k}\partial_{l}\eta.
\end{align*}
Let $v\in H^{1}_{0}(B_{R};\mathbb{R}^{d})$ be the weak solution of
\begin{align}\label{LLAM0}
-\Delta v^{i}=G_{i}.
\end{align}
Then $\eta w$ verifies
\begin{align*}
\int_{B_{R}}C^{0}_{ijkl}\partial_{l}(\eta w^{k})\partial_{j}\varphi^{i}=\int_{B_{R}}\widehat{F}\partial_{j}\varphi^{i},\quad\forall\varphi\in C^{\infty}_{0}(B_{R};\mathbb{R}^{d}),
\end{align*}
where $\widehat{F}_{ij}:=\widetilde{F}_{ij}+\partial_{j}v^{i}$.

Since $f_{ij}\in C^{0,\gamma}$, then we obtain that $f_{ij}\in L^{p}(B_{R})$ for any $d\leq p<\infty$. Suppose that $w\in W^{1,q}(B_{R};\mathbb{R}^{d})$, $q\geq2$. Then we know
\begin{align}\label{MBL001}
G_{i}&\in L^{p\wedge q}(B_{R}),\quad\mathrm{where}\;p\wedge q:=\min\{p,q\},
\end{align}
and
\begin{align}\label{MBL002}
\widetilde{F}_{ij}\in L^{p\wedge q^{\ast}}(B_{R}),\quad\mathrm{where}\;q^{\ast}:=&
\begin{cases}
\frac{dq}{d-q},&q<d,\\
2q,&q\geq d.
\end{cases}
\end{align}
Utilizing $L^{2}$ estimate for equation \eqref{LLAM0}, we see that $\nabla^{2}v\in L^{2}(B_{R})$ and
\begin{align*}
-\Delta(\partial_{j}v^{i})=\partial_{j}G_{i}.
\end{align*}
This, in combination with \eqref{MBL001} and Theorem 7.1 in \cite{GM2013}, gives that $\nabla(\partial_{j}v^{i})\in L^{p\wedge q}(B_{R})$. Then applying the Sobolev embedding theorem, we get $\partial_{j}v^{i}\in L^{(p\wedge q)^{\ast}}$. Together with \eqref{MBL002}, this yields that $\widehat{F}_{ij}\in L^{p\wedge q^{\ast}}(B_{R})$. Then it follows from Theorem 7.1 in \cite{GM2013} again that
\begin{align*}
\|\nabla(\eta w)\|_{L^{p\wedge q^{\ast}}(B_{R})}\leq C\|\widehat{F}\|_{L^{p\wedge q^{\ast}}(B_{R})},
\end{align*}
where $C=C(d,\lambda,\mu,p,q)$ and $\widehat{F}:=(\widehat{F}_{ij})$, $i,j=1,2,...,d$. In light of the definition of $G_{i}$ and $\widetilde{F}_{ij}$ and using \eqref{MBL001}--\eqref{MBL002}, we obtain
\begin{align}\label{AGTQ001}
\|\nabla w\|_{L^{p\wedge q^{\ast}}(B_{\rho})}\leq\frac{C}{R-\rho}(\|w\|_{W^{1,p}(B_{R})}+\|F\|_{L^{p}(B_{R})}),
\end{align}
where $C=C(d,\lambda,\mu,p,q)$.

We next demonstrate that $\nabla w\in L^{p}(B_{R/2})$. Introduce a series of balls with radii as follows:
\begin{align*}
R_{k}=R\Big(\frac{1}{2}+\frac{1}{2^{k+1}}\Big),\quad k\geq0.
\end{align*}
Picking $\rho=R_{1}$ and $q=2$ in \eqref{AGTQ001}, we have
\begin{align*}
\|\nabla w\|_{L^{p\wedge 2^{\ast}}(B_{R_{1}})}\leq\frac{C}{R}(\|w\|_{W^{1,2}(B_{R})}+\|F\|_{L^{p}(B_{R})}).
\end{align*}
If $p\leq2^{\ast}$, then we complete the proof. If $p>2^{\ast}$, then $\nabla w\in L^{2^{\ast}}(B_{R_{1}})$ and
\begin{align*}
\|\nabla w\|_{L^{2^{\ast}}(B_{R_{1}})}\leq\frac{C}{R}(\|w\|_{W^{1,2}(B_{R})}+\|F\|_{L^{p}(B_{R})}),
\end{align*}
which, together with choosing $R=R_{1}$, $\rho=R_{2}$ and $q=2^{\ast}$ in \eqref{AGTQ001}, gives that
\begin{align*}
\|\nabla w\|_{L^{p\wedge2^{\ast\ast}}(B_{R_{2}})}\leq&\frac{C}{R}(\|w\|_{W^{1,2^{\ast}}(B_{R_{1}})}+\|F\|_{L^{p}(B_{R_{1}})})\notag\\
\leq&\frac{C}{R^{2}}(\|w\|_{W^{1,2}(B_{R})}+\|F\|_{L^{p}(B_{R})}).
\end{align*}
If $p\leq2^{\ast\ast}$, then we complete the proof. If $p>2^{\ast\ast}$, after repeating the above argument with finite steps, we derive that $\nabla w\in L^{p}(B_{R/2})$ and
\begin{align}\label{JANT001}
\|\nabla w\|_{L^{p}(B_{R/2})}\leq C(\|w\|_{H^{1}(B_{R})}+\|F\|_{L^{p}(B_{R})}),
\end{align}
where $C=C(d,\lambda,\mu,p,\mathrm{dist}(B_{R},\partial Q))$.

Finally, we utilize the method of locally flattening the boundary to establish the $W^{1,p}$ estimates near boundary $\Gamma$, which is almost the same to the proof in Lemma \ref{CL001}. With the same notations as before, we obtain that $\mathcal{W}(y):=w(\Psi^{-1}(y))\in H^{1}(\mathcal{B}^{+}_{R},\mathbb{R}^{d})$ solves
\begin{align*}
\int_{\mathcal{B}^{+}_{R}}\mathcal{C}^{0}_{ijkl}(y)\partial_{l}\mathcal{W}^{k}\partial_{j}\varphi^{i}dy=\int_{\mathcal{B}^{+}_{R}}\mathcal{F}_{ij}\partial_{j}\varphi^{i}dy,\quad\forall\varphi\in H^{1}_{0}(\mathcal{B}^{+}_{R},\mathbb{R}^{d}).
\end{align*}
Making use of the proof of Theorem 7.2 of \cite{GM2013}, we derive that for any $d\leq p<\infty$,
\begin{align*}
\|\nabla\mathcal{W}\|_{L^{p}(\mathcal{B}^{+}_{R/2})}\leq C(\|\mathcal{W}\|_{H^{1}(\mathcal{B}^{+}_{R})}+\|\mathcal{F}\|_{L^{p}(\mathcal{B}^{+}_{R})}),
\end{align*}
where $C=C(\lambda,\mu,p,R,\Psi)$. Then back to $w$, we get
\begin{align*}
\|\nabla w\|_{L^{p}(\mathcal{N}^{'})}\leq C(\|\mathcal{W}\|_{H^{1}(\mathcal{N})}+\|\mathcal{F}\|_{L^{p}(\mathcal{N})}),
\end{align*}
where $\mathcal{N}'=\Psi^{-1}(\mathcal{B}^{+}_{R/2})$, $\mathcal{N}=\Psi^{-1}(\mathcal{B}^{+}_{R})$ and $C=C(\lambda,\mu,p,R,\Psi)$. Furthermore, there exists a constant $0<\sigma<1$ independent of $R$ such that $B_{\sigma R}\cap Q\subset\mathcal{N}'$.

Hence, for each $x_{0}\in Q'\cap\Gamma$, there exists $R_{0}:=R_{0}(x_{0})>0$ such that
\begin{align}\label{JANT002}
\|\nabla w\|_{L^{p}(B_{\sigma R_{0}}(x_{0})\cap Q')}\leq C(\|w\|_{H^{1}(Q)}+\|F\|_{L^{p}(Q)}),
\end{align}
where $C=C(\lambda,\mu,p,x_{0},R)$. Then combining \eqref{JANT001}--\eqref{JANT002} and utilizing the finite covering theorem, we deduce that
\begin{align*}
\|\nabla w\|_{L^{p}(Q')}\leq C(\|w\|_{H^{1}(Q)}+\|F\|_{L^{p}(Q)}),
\end{align*}
where $C=C(\lambda,\mu,p,Q',Q)$. Together with the Poincar\'{e} inequality, this gives that \eqref{LNZ001} holds.

Observe that for any constant matrix $\mathcal{M}=(\mathfrak{a}_{ij})$, $i,j=1,2,...,d$, $w$ verifies \eqref{ADCo1} with $F-\mathcal{M}$ substituting for $F$. Consequently, by utilizing the continuous injection that $W^{1,p}\hookrightarrow C^{0,\gamma}$, $0<\gamma\leq1-d/p$, we conclude that \eqref{LNZ002} holds.

\end{proof}

%\noindent{\bf{\large Acknowledgements.}} X. Hao is greatly indebted to Professor HaiGang Li for his constant encouragement and support. Z.W. Zhao would like to thank School of Mathematical Sciences at Beijing Normal University for the stimulating environment. Z.W. Zhao was partially supported by NSFC (11971061) and BJNSF (1202013).

\end{document}